\documentclass[11pt,leqno]{article}
\usepackage{amsmath, amscd, amsthm, amssymb, graphics, xypic, mathrsfs, setspace, fancyhdr, bm, pdfsync, enumitem, times}
\usepackage[usenames, dvipsnames, svgnames, table]{xcolor}
\usepackage[letterpaper,top=1.05in, bottom=1.05in, left=1.05in, right=1.05in]{geometry}
\usepackage[colorlinks=true,pagebackref=true]{hyperref}
\usepackage{comment}
\usepackage[textsize=scriptsize]{todonotes}
\usepackage{tikz-cd}
\hypersetup{backref,hidelinks}
\usepackage{etoolbox}

\newtoggle{final}
\toggletrue{final}

\iftoggle{final} {
\newcommand{\aravind}[1]{} 
\newcommand{\tom}[1]{} 
\newcommand{\mike}[1]{} 
\newcommand{\NB}[1]{}
\newcommand{\TODO}[1]{}
\renewcommand{\todo}[1]{}
}{ 
\newcommand{\aravind}[1]{\textcolor{red}{#1}} 
\newcommand{\tom}[1]{\todo[color=blue!40]{#1}} 
\newcommand{\mike}[1]{\textcolor{green}{#1}} 
\newcommand{\NB}[1]{\todo[color=gray!40]{#1}}
\newcommand{\TODO}[1]{\todo[color=red]{#1}}
}

\newcommand{\tensor}{\otimes}
\newcommand{\colim}{\operatorname{colim}}

\newcommand{\Spec}{\operatorname{Spec}}

\newcommand{\weq}{\simeq}
\newcommand{\isomto}{{\stackrel{\sim}{\;\longrightarrow\;}}}

\newcommand{\sma}{{\scriptstyle{\wedge}\,}}
\newcommand{\coker}{\operatorname{coker}}
\newcommand{\longhookrightarrow}{\lhook\joinrel\longrightarrow}

\newcommand{\op}[1]{\operatorname{#1}}
\newcommand{\sslash}{/\mkern-6mu/}

\def\ph{\mathord-}

\newcommand{\lra}[1]{\langle #1 \rangle}

\renewcommand{\hom}{\operatorname{Hom}}
\newcommand{\Map}{\operatorname{Map}}
\newcommand{\iMap}{\underline{\Map}}
\newcommand{\id}{\mathrm{id}}

\newcommand{\real}{{\mathbb R}}
\newcommand{\PP}{\mathbb{P}}

\newcommand{\Q}{{\mathbb Q}}
\newcommand{\Z}{{\mathbb Z}}

\newcommand{\A}{{\mathbb A}}
\newcommand{\D}{{\mathrm D}}

\newcommand{\aone}{{\mathbb A}^1}
\newcommand{\pone}{{\mathbb P}^1}

\renewcommand{\1}{{\rm 1\hspace*{-0.4ex}%
\rule{0.1ex}{1.52ex}\hspace*{0.2ex}}}

\newcommand{\gm}[1]{{{\mathbb G}_{m}^{#1}}}
\newcommand{\Gm}{{\gm{}}}
\renewcommand{\L}{{\mathrm L}}

\newcommand{\et}{\text{\'et}}
\newcommand{\ret}{\text{r{\'e}t}}

\newcommand{\ho}[2][]{\Spc_{#1}({#2})}
\newcommand{\DM}{{\mathbf{DM}}}

\newcommand{\bpi}{\bm{\pi}}

\newcommand{\eff}{\operatorname{eff}}
\newcommand{\veff}{\operatorname{veff}}
\newcommand{\Nis}{{\operatorname{Nis}}}
\newcommand{\SH}{{\mathrm{SH}}}

\newcommand{\Shv}{{\mathrm{Shv}}}
\newcommand{\Sm}{\mathrm{Sm}}
\newcommand{\Sch}{\mathrm{Sch}}
\newcommand{\Cor}{\mathrm{Cor}}
\newcommand{\Spc}{\mathrm{Spc}}
\newcommand{\Mod}{\mathrm{Mod}}
\newcommand{\Ab}{\mathrm{Ab}}
\newcommand{\Grp}{\mathrm{Grp}}
\newcommand{\Set}{\mathrm{Set}}
\newcommand{\K}{{{\mathbf K}}}
\newcommand{\Th}{{\operatorname{Th}}}

\renewcommand{\H}{{{\mathbf H}}}
\newcommand{\Alg}{{\mathbf{Alg}}}
\newcommand{\Sym}{{\operatorname{Sym}}}
\newcommand{\Lmot}{{\operatorname{L}_{\mathrm{mot}}}}
\newcommand{\Lret}{{\operatorname{L}_{\ret}}}
\newcommand{\Lretna}{{\operatorname{L}_{\ret}^{naive}}}

\newcommand{\opcat}{\mathrm{op}}

\newcommand{\Singaone}{\operatorname{Sing}^{\aone}\!\!}

\renewcommand{\setminus}{\smallsetminus}

\def\adj{\leftrightarrows}
\def\Alg{\mathrm{Alg}}
\def\Mon{\mathrm{Mon}}
\def\CMon{\mathrm{CMon}}
\def\CAlg{\mathrm{CAlg}}

\newcommand{\SHS}{\SH^{S^1}\!}
\newcommand{\cof}{\mathrm{cof}}
\newcommand{\fib}{\mathrm{fib}}

\def\naive{\mathrm{naive}}

\newcommand{\Addresses}{{
		\bigskip
		\footnotesize
		
		A.~Asok, Department of Mathematics, University of Southern California, 3620 S.~Vermont Ave., Los Angeles, CA 90089-2532, United States; E-mail address: asok@usc.edu
		\medskip
		
		T.~Bachmann, Department of Mathematics, JGU Mainz, Staudingerweg 9, 55128 Mainz, Germany; E-mail address: \url{tom.bachmann@zoho.com}
		\medskip
		
		M.J.~Hopkins, Department of Mathematics, Harvard University, One Oxford Street, Cambridge, MA 02138, United States \textit{E-mail address:} \url{mjh@math.harvard.edu}
}}

\newcounter{intro}
\theoremstyle{plain}
\newtheorem{theorem}{Theorem}[subsection]
\newtheorem{amplification}[theorem]{Amplification}
\newtheorem{lem}[theorem]{Lemma}
\newtheorem{cor}[theorem]{Corollary}
\newtheorem{proposition}[theorem]{Proposition}
\newtheorem*{claim*}{Claim} 

\newtheorem*{thm*}{Theorem}
\newtheorem*{problem*}{Problem}

\newtheorem{thmintro}{Theorem}

\theoremstyle{definition}
\newtheorem{defn}[theorem]{Definition}
\newtheorem{construction}[theorem]{Construction}
\newtheorem{notation}[theorem]{Notation}

\theoremstyle{remark}
\newtheorem{rem}[theorem]{Remark}
\newtheorem{remintro}[thmintro]{Remark}

\newtheorem{ex}[theorem]{Example}

\newtheorem{entry}[theorem]{}
\newtheorem{discussion}[theorem]{Discussion}

\numberwithin{equation}{subsection}

\begin{document}
\pagestyle{fancy}
\renewcommand{\sectionmark}[1]{\markright{\thesection\ #1}}
\fancyhead{}
\fancyhead[LO,R]{\bfseries\footnotesize\thepage}
\fancyhead[LE]{\bfseries\footnotesize\rightmark}
\fancyhead[RO]{\bfseries\footnotesize\rightmark}
\chead[]{}
\cfoot[]{}
\setlength{\headheight}{1cm}

\author{Aravind Asok\thanks{Aravind Asok was partially supported by National Science Foundation Awards DMS-1802060 and DMS-2101898} \and Tom Bachmann \and Michael J. Hopkins}

\title{{\bf On ${\mathbb P}^1$-stabilization in unstable motivic homotopy theory}}
\date{}
\maketitle

\begin{abstract}
We analyze stabilization with respect to ${\mathbb P}^1$ in the Morel--Voevodsky unstable motivic homotopy theory.  We introduce a refined notion of cellularity (a.k.a., biconnectivity) in various motivic homotopy categories taking into account both the simplicial and Tate circles.  Under suitable cellularity hypotheses, we refine the Whitehead theorem by showing that a map of nilpotent motivic spaces can be seen to be an equivalence if it so after taking (Voevodsky) motives.  We then establish a version of the Freudenthal suspension theorem for ${\mathbb P}^1$-suspension, again under suitable cellularity hypotheses.  As applications, we resolve Murthy's conjecture on splitting of corank $1$ vector bundles on smooth affine algebras over algebraically closed fields having characteristic $0$ and compute new unstable motivic homotopy of motivic spheres.
\end{abstract}

\begin{footnotesize}
\tableofcontents
\end{footnotesize}

\iftoggle{final}{}{
\section*{Thoughts (by Tom)}
\subsection*{Truncation and cellularity}
It is not true without further assumptions that if $\mathscr X \in O(S^{d,d'})$ then $L^{p,q} \mathscr X \in O(S^{d,d'})$ as well.
For a stable example, take $\mathscr X = H\Z(n)[n]$, $d=d'=n$, $p=1$, $q=0$.
Then $L^{p,q} \mathscr X = K_n^M \ne \mathscr X$.
The point is (I think) that it is possible that $S^{d,d'} \not\in O(S^{p,q})$ and also $S^{p,q} \not\in O(S^{d,d'})$.

\subsection*{Strange fact}
Since $\Lmot: \Shv_\Nis(\Sm_k) \to \Shv_\Nis(\Sm_k)$ preserves finite products, it seems to me that one can learn that $\Omega^\infty: \Shv_\Nis(\Sm_k, \SH)_{>0} \to \Shv_\Nis(\Sm_k)$ commutes with $\Lmot$ (compare with commutative monoids).
I feel that at some point this seemed problematic to me...

\subsection*{Small object argument for non-connected spaces}
\begin{rem}
For general $\mathscr X \in \Spc(k)$ (neither pointed nor connected), we can still construct $\L_A \mathscr X$ by a small object argument.
Set $\mathscr X_0 = \mathscr X$ and iteratively define $\mathscr X_{i+1}$ by
\begin{equation*}
\begin{CD}
\coprod_{S^n \times A \times X \to \mathscr X_i} S^n \times A \times X @>>> \mathscr X_i \\
@VVV @VVV \\
\coprod_{S^n \times A \times X \to \mathscr X_i} S^n \times X @>>> \mathscr X_{i+1}. \\
\end{CD}
\end{equation*}
Set $\mathscr X_\infty = \colim_i \mathscr X_i$.
Then $\mathscr X \to \mathscr X_\infty$ is an $A$-localization.
We do not use this in the sequel.
\end{rem}

\begin{proof}
Let $X \in \Sm_k$.
We must show that $\mathscr X_\infty(X) \xrightarrow{p^*} \mathscr X_\infty^A(X)$ (induced by $p: A \to *$) is an equivalence.
The pointing $* \to A$ induces a retraction $\mathscr X_\infty^A(X) \to \mathscr X_\infty(X)$, so it is enough to show that $p^*$ induces a surjection on all homotopy groups.
An element of $\pi_0  \mathscr X_\infty^A(X)$ comes from a map $\alpha: A \times X \to \mathscr X_n$ for some $n$.
Then $\alpha \amalg \alpha: S^0 \times A \times X \to \mathscr X_n$ factors over $S^0 \times X \to \mathscr X_{n+1}$, and hence $\alpha$ factors over $X$ in $\mathscr X_{n+1}$.
This shows that $\pi_0(p^*)$ is surjective.
Now we show surjectivity on the higher homotopy groups.
Pick a base point $b: * \to \mathscr X_\infty(X)$, corresponding to a map $X \to \mathscr X_\infty$.
Comparing definitions, we see that a pointed map $s: S^n \to \mathscr X_\infty^A(X)$ corresponds to a map $S^n \times X \times A \to \mathscr{X}_\infty$ in the category $\Spc(k)_{X \times A/}$.
Consequently we must solve the following filling problem in $\Spc(k)_{X \times A/}$
\tom{for some reason this tikz doesn't render right anymore}
By compactness, both $b$ and $s$ lift to $\mathscr X_n$ for some $n$.
Since the forgetful functor $\Spc(k)_{X \times A/} \to \Spc(k)$ preserves pushouts, the desired filler exists in $\mathscr X_{n+1}$.
\end{proof}
\begin{discussion}
This took me embarrassingly long to work out, but I suppose it must be standard.
One difficulty is as follows.
Suppose given a map of connected spaces $X \to Y$, and suppose that every map $S^n \to Y$ lifts to $X$, \emph{in the category of unpointed spaces}.
This does \emph{not} seem to imply that $\pi_1(X) \to \pi_1(Y)$ is surjective, and I do not see an easy fix for this.
Hence why the discussion and translation of pointed maps is crucial in the above proof.
\end{discussion}

\begin{discussion}
Using Lemma \ref{lem:localizationmappingspaces}(5), one can show that if $\mathscr X$ is pointed and connected then the nullification can be obtained by attaching cells of the form $S^n \wedge A \wedge X_+$.\NB{This operation preserves connectivity.}
This is not true without the connectivity assumption.

I would like it if the construction for pointed connected spaces could be seen as an easy consequence of this.
If $\mathscr X$ is pointed connected, then locally on $X$ any map $A \times X \to X$ restricts to null along $X \to A \times X$ and hence factors over $*$, and hence at least locally \[ \mathscr X \amalg_{A \times X} X \weq \mathscr{X}/A \wedge X_+. \]
I don't know how to deal with $S^n \times A \times X \to S^n \times X$.
\end{discussion}
}

\section{Introduction}
\label{s:intro}
If $n$ is an integer $\geq 2$, and $X$ is an $(n-1)$-connected CW complex, then the classical Freudenthal suspension theorem asserts that the unit map $X \to \Omega \Sigma X$ has $(2n-2)$-connected fiber, i.e., induces an isomorphism on homotopy groups in degrees $\leq 2n-2$ and an epimorphism in degree $2n-1$.  The suspension theorem gives rise to the notion of ``stable'' phenomena in topology, which according to Adams occur in ``sufficiently large dimensions", in contrast to ``unstable'' phenomena, which occur in some definite dimensions.  Typically, stable phenomena are more amenable to calculation than unstable phenomena, while unstable phenomena are more closely related to geometry; this distinction perhaps led Eilenberg to quip ``we can distinguish two cases--the stable case and the interesting case'' \cite[p. 125]{Adams}.  Granted control of the stabilization process, one may often ``unwind'' an unstable problem into a sequence of stable problems whose solutions are more computationally accessible.  

In algebraic geometry, the influence of homotopy theory has followed a different course.  Stable structures, e.g., cohomology theories of various sorts like algebraic K-theory have been studied for a long time.  However, stable and unstable homotopy categories housing these structures have only comparatively recently aided in analysis of such theories.  For concreteness, fix a field $k$ and write $\Sm_k$ for the category of smooth $k$-varieties and $\ho{k}$ for the Morel--Voevodsky unstable motivic homotopy category \cite{MV}.  The stable motivic homotopy category, denoted $\SH(k)$, is built by stabilizing with respect to {\em two} circles: a ``simplicial'' circle $S^1 =: S^{1,0}$ and a ``Tate'' circle $\Gm =: S^{1,1}$; equivalently, one may stabilize with respect to $\pone \weq S^{2,1} = S^1 \wedge \gm{}$.  Stabilizing with respect to $\pone$ is a ``natural" choice because known cohomology theories, e.g., algebraic K-theory, are representable in the resulting setup \cite{VICM}.  The category $\SH(k)$ has been analyzed intensely since its construction, following analogies both structural and computational with classical stable homotopy theory.  

To date, the link between unstable and stable motivic homotopy theory has been harder to quantify.  Morel, in his foundational work \cite{MField} analyzed stabilization with respect to $S^1$.  His results imply that stabilization with respect to the simplicial circle behaves as one expects from the classical Freudenthal suspension theorem: if $\mathscr{X}$ is an $(n-1)$-connected motivic space, then the map $\mathscr{X} \to \Omega\Sigma \mathscr{X}$ has $(2n-2)$-connected fiber.  However, stabilization with respect to $S^1$ is not ``geometrically" well-behaved.  More importantly, it seems not as well connected with the ``stable" phenomena suggested above and is therefore not so helpful from a computational standpoint.   

All of the above points were understood shortly after Morel and Voevodsky introduced the motivic homotopy category.  Morel's early computations (eventually published in \cite{MField}) imply that a stabilization theorem for $\pone$-suspension of form similar to $S^1$-suspension cannot hold without additional hypotheses.  Moreover, Voevodsky wrote: ``[a]n analog of the Fre[u]denthal suspension theorem must exist, but we have no idea what it is" (\cite[p. 12]{Voevodskymyview}).  That additional hypotheses are necessary is perhaps not surprising given the form of the suspension theorem in equivariant homotopy theory \cite[Theorem II.2.10]{tomDieck}.


Thus, to date, the situation can be described as follows: unstable motivic homotopy theory does control geometric phenomena (e.g., vector bundles on smooth affine varieties), and stable motivic homotopy theory is calculable, but only in very special cases has one been able to achieve sufficient control over the stabilization process to link unstable and stable phenomena.  Our goal in this paper is to bridge this gap by establishing a version of the Freudenthal suspension theorem for $\pone$-suspension in unstable motivic homotopy theory.  

In order to state the result, we introduce a notion of cellularity, extending that described by Dror \cite{Farjoun}.  Given a pointed (compact) motivic space $A$, we construct a class of motivic spaces we call ``weakly $A$-cellular''. It contains $A \wedge X_+$ for any smooth scheme $X$, is stable under (weak) equivalences, (homotopy) colimits and ``(homotopy) cofiber extensions"; we write $O(A)$ for the class of weakly $A$-cellular spaces.  We will be particularly interested in the class of weakly $S^{p,q}$-cellular motivic spaces, $p \geq q$. 
 
The class $O(S^{n+1,0})$ is the class of $n$-connected motivic spaces, while $O(S^{q,q})$ is an unstable analog of the notion of $q$-effective spectra studied by various authors.  For applications to be described momentarily, it will suffice to observe that the punctured affine space ${\mathbb A}^n \setminus 0 \in O(S^{4,2})$ if $n \geq 3$.  With that terminology in hand, we may state our $\pone$-stabilization theorem as a ``weak-cellular estimate'' on the fiber of the stabilization map.

\begin{thmintro}[See Theorem~\ref{thm:beyondthediagonal}]
	\label{thmintro:freudenthal}
	Assume $k$ is a field that has characteristic $0$.  If $\mathscr{X} \in \Spc(k)$ is a pointed motivic space and lies in $O(S^{p,q})$ with $p -q \ge 2, q \geq 2$, then the fiber of the stabilization map
	\[
	\mathscr{X} \longrightarrow \Omega^{2,1}\Sigma^{2,1} \mathscr{X}
	\]
	lies in $O(S^{a,2q})$ where $a = min(2p-1,p+2q-1)$.
\end{thmintro}

\begin{remintro}
	In the body of the text, we will also establish a version of Theorem~\ref{thmintro:freudenthal} when $k$ is a field having characteristic exponent $p > 0$.  Roughly speaking, the above statement holds after inverting $p$; see Theorem~\ref{thm:beyondthediagonal} for more details.  Recent unpublished results of F. Morel imply that it should suffice to assume only that $q \geq 1$ as well; see Remarks~\ref{rem:morelimprovement} and \ref{rem:freudenthalhypotheses} for further discussion. 
\end{remintro}

\subsection*{Applications} \tom{I have not revised this}
Theorem~\ref{thmintro:freudenthal} has a host of consequences, some of which will be developed in forthcoming work.  We content ourselves with listing here one concrete application to the theory of algebraic vector bundles and some computational consequences.  Building on a long line of classical questions in the theory of projective modules over commutative rings, it is natural to search for ``explicit" obstructions to splitting positive corank bundles on (smooth) affine varieties over fields (see \cite{Murthy} or the introduction to \cite{AFpi3a3minus0} for some discussion of the history of this problem).  For a precise problem in this direction, M.P. Murthy suggested that over an algebraically closed field $k$, the only obstruction to splitting off a trivial rank $1$ summand from a rank $d$ bundle $E$ on a smooth affine variety of dimension $d+1$, $d \geq 1$, was the vanishing of $c_d(E) \in CH^d(X)$ \cite[p. 173]{Murthy}.  That the preceding statement was true came to be known as Murthy's splitting conjecture \cite[Conjecture 1]{AFpi3a3minus0}. 

\begin{thmintro}[See Theorem~\ref{thm:murthy}]
	\label{thmintro:murthy}
	Assume $k$ is an algebraically closed field having characteristic $0$.  If $X$ is a smooth affine $k$-scheme of dimension $d+1$, $d \geq 1$, and $E$ is a rank $d$ vector bundle on $X$, then $E$ splits off a trivial rank $1$ summand if and only if $0 = c_{d}(E) \in CH^d(X)$.  
\end{thmintro}

The theorem is immediate from definitions in the case $d = 1$ and as such is only interesting when $d \geq 2$.  In that case, the connection between Theorem~\ref{thmintro:freudenthal} and Theorem~\ref{thmintro:murthy} was laid out in the introduction of \cite{AFpi3a3minus0}.  In brief, the question of splitting a trivial rank $1$ summand off a vector bundle can be phrased as an $\aone$-homotopical lifting problem by appeal to the representability theorem for vector bundles \cite[Theorem 1]{AHW}.  In turn, the relevant lifting problem can be analyzed by techniques of obstruction theory.  

In the case of Murthy's conjecture, the potentially non-vanishing obstructions rely on knowledge of the sheaves $\bpi_{d-1}({\mathbb A}^d \setminus 0)$ and $\bpi_d({\mathbb A}^d \setminus 0)$, $d \geq 2$.  The sheaves $\bpi_{d-1}({\mathbb A}^d \setminus 0)$ were computed by Morel and stabilize for $d \geq 2$.   In \cite{AFThreefolds} and \cite{AFpi3a3minus0}, the first author and J. Fasel computed the sheaves $\bpi_d({\mathbb A}^d \setminus 0)$ for $d = 2,3$, at least over fields having characteristic not equal to $2$.  

Building on some ideas of Morel, the first author and J. Fasel formulated a general conjecture for the form of $\bpi_d({\mathbb A}^d \setminus 0)$ for $d \geq 4$ in \cite[Conjecture 7]{AFOBW} from which Murthy's conjecture would follow.  In fact, to establish Murthy's conjecture one only needs the computation of $\bpi_3({\mathbb A}^3 \setminus 0)$ mentioned above and a surjectivity statement that follows from Theorem~\ref{thmintro:freudenthal}, i.e., less than the complete computation of $\bpi_d({\mathbb A}^d \setminus 0)$.  This is the path we follow to deduce Murthy's conjecture.  On the other hand, Theorem~\ref{thmintro:freudenthal} gives much more.  Indeed, the entire structure of $\bpi_d({\mathbb A}^d \setminus 0)$ for $d \geq 4$ may be deduced from a corresponding $\pone$-stable computation which appears in \cite{RSO}, at least over fields having characteristic $0$ (see Theorem~\ref{thm:pid}). 

\subsection*{Overview of the proof}
Freudenthal's original proof of the suspension theorem in the case of spheres was of a rather geometric nature \cite{Freudenthal}.  Blakers and Massey employed a more formal and conceptually clearer technique via homotopy excision \cite{BlakersMassey}, which is essentially the standard proof now.  Our proof, by necessity, does not follow any of the classical proofs so we give here a brief outline for the argument.  

First, there is the key geometric ingredient arising in an analysis of motivic Eilenberg--Mac Lane spaces made possible by Voevodsky's motivic Dold--Thom theorem, which gives a model of $H\Z \wedge \Sigma^{\infty}_{\pone} X_+$ for $X$ a smooth scheme \cite{VoeMEM} in terms of colimits of suitable symmetric powers.  By direct analysis of the geometry of symmetric powers, we can give a weak cellular estimate on the cofiber of the bonding maps in the spectrum above.  Indeed, the cofiber of the bonding maps can be described geometrically as a Thom complex of a suitable equivariant vector bundle and the cellular estimate in question can be obtained by analyzing the stabilizer type stratification of the symmetric groups; this kind of analysis goes back to the work of Nakaoka \cite{Nakaoka} and was codified in \cite{AroneDwyer}.   In a motivic context, this kind of analysis also appears in Voevodsky's work on the zero slice of the sphere spectrum in characteristic $0$ \cite{VZeroSlice}.  In the context of this paper, the required result in characteristic $0$ is Theorem~\ref{thm:cofiberofmotivicEMassemblycellularityestimate} and a modification in positive characteristic situations can be found in Theorem~\ref{thm:cofiberofmotivicEMassemblycellularityestimatecharp}.

Second, we reduce from $H\Z$-modules to $\pone$-spectra.  This uses a number of ingredients, but perhaps the most important is an adaptation of the main idea in Levine's argument for the convergence of Voeovodsky's slice tower \cite{LevineConvergence} as detailed in \cite{BEO}.  Levine's argument treats the case of fields of finite cohomological dimension, but in conjunction with some information about real \'etale homotopy theory, one can treat the case of arbitrary characteristic $0$-fields.  This reduction is handled in the proof of Theorem~\ref{thm:freudenthal}, but see Lemmas~\ref{lem:torsioncaseconnectivityestimate} and~\ref{lem:therholocalcase}.  This convergence result implicitly relies on the celebrated norm residue isomorphism theorem, a.k.a. the Bloch--Kato conjecture, established by Voevodsky--Rost (originally \cite{VMod2,VModl}, but see \cite{HWBK} and references therein for a consolidated treatment). 

Third, we reduce from $\pone$-spectra to spaces.  This reduction uses a refinement of the classical Whitehead tower of a space taking weak-cellular classes into account.  Building this refined tower and analyzing its properties is the content of Section~\ref{s:cellularWhiteheadtower}, but see Amplification~\ref{amplification:2-eff-whitehead} for a precise statement.  The key inputs to this statement are results of Bachmann--Yakerson \cite{BachmannYakerson} and Bachmann \cite{BachmannConservativity} on conservativity of stabilization with respect to $\Omega^{1,1}$ which we review here as Theorem~\ref{thm:conservativityofgmstabilization}.   Implicitly, this conservativity result relies on the motivic infinite loop space technology of \cite{EHKSYdeloop}.  

These tools are stitched together using the notion of weak cellular classes that we develop in Section~\ref{s:weakcellularityandnullity}.  This section contains the basic definitions and develops a calculus of cellular classes including relationship between cellularity and fiber and cofiber sequences, loop spaces of various sorts, etc.; these results rely on a host of tools, including Levine's analysis of the homotopy coniveau tower \cite{Levineconiveautower}.

\begin{remintro}
Given the bevy of tools mentioned above, it is natural to ask whether there might be a simpler proof of Theorem~\ref{thmintro:freudenthal}.  We make some elliptical comments in this direction.  Our proof of the motivic Freudenthal suspension theorem in essence uses the classical Freudenthal suspension theorem, so does not give a new proof of this result (though it does suggest a new way to organize the proof of equivariant Freudenthal suspension theorems).  Because motivic localization fails to be a left-exact localization, most of the tools of, e.g., \cite{ABFJ} analyzing very general Blakers--Massey style excision results in $\infty$-topoi fail to apply in this context.  Indeed, Theorem~\ref{thmintro:freudenthal} is known to be false if $q = 0$, e.g., for $S^{p,0}$ immediately from Morel's computations of homotopy sheaves.  It is possible to reorganize some parts of the argument using more homotopy-theoretic tools, but we save this for future work.  
\end{remintro}


\subsection*{Notation and conventions}
Much of the notation to be used in this paper will be introduced in Sections~\ref{s:preliminaries} and \ref{s:weakcellularityandnullity}; we try to present an overview here.  Throughout, assume $k$ is a field; further hypotheses will be imposed on $k$ as necessary in the body of the paper: frequently $k$ will be assumed perfect, and sometimes furthermore infinite.  Write $\Sm_k$ for the category of separated, smooth $k$-schemes and $\Sch_k$ for the category of finite-type separated $k$-schemes; we will also write $\Sch_k^{qp}$ for the category of quasi-projective $k$-schemes.

If $\mathrm{S}_k$ is one of the categories of schemes described above, then we write $\mathrm{P}(\mathrm{S}_k)$ for the $\infty$-topos of presheaves of spaces on $\Sm_k$.  If $\tau$ is a Grothendieck topology on $\mathrm{S}_k$, then we write  $\Shv_{\tau}(\mathrm{S}_k) \subset \mathrm{P}(\mathrm{S}_k)$ for full subcategory of $\tau$-sheaves of spaces on $\mathrm{S}_k$, which is once again an $\infty$-topos.  There are corresponding pointed versions of all these constructions: we write $\mathrm{P}(\mathrm{S}_k)_*$ for the $\infty$-topos of presheaves of pointed spaces and $\Shv_{\tau}(\mathrm{S}_k)_{*}$ for the subcategory of $\tau$-sheaves of pointed spaces.

We write $S^i$ for the simplicial $i$-circle, and we set $S^{p,q} := S^{p-q} \wedge \gm{\sma q}$.  We write $\Sigma^i$ for the operation of smashing with $S^i$, and $\Sigma^{p,q}$ for the operation of smashing with $S^{p,q}$.  We write $\iMap$ for the internal mapping space, i.e., given $\mathscr{X},\mathscr{Y} \in \Spc(k)$, $\iMap(\mathscr{X},\mathscr{Y})$ represents $U \mapsto \hom_{\Shv_{Nis}(\Sm_k)}(U \times \mathscr{X},\mathscr{Y})$.  Likewise, if $\mathscr{X}$ and $\mathscr{Y}$ are pointed spaces, then we write $\iMap_{*}(\mathscr{X},\mathscr{Y})$ for the pointed internal mapping space.  We write $\Omega^i$ for the $i$-fold loops functor adjoint to $\Sigma^i$ and, more generally, $\Omega^{p,q}\mathscr{X}$ for $\iMap_*(S^{p,q},\mathscr{X})$.



Given a cofiber sequence
\[
\mathscr{X} \longrightarrow \mathscr{Y} \longrightarrow \mathscr{Z}
\]
we will refer to $\mathscr{Y}$ as a cofiber extension of $\mathscr{Z}$ by $\mathscr{X}$.  We write $\SHS(k)$ for the $S^1$-stable motivic homotopy category, and $\SH(k)$ for the $\pone$-stable motivic homotopy category; the first category is constructed from $\Spc(k)_*$ by inverting $S^{1}$, while the latter is constructed from $\Spc(k)_*$ by inverting $S^{2,1}$ or from $\SHS(k)$ by inverting $\Sigma^{\infty} S^{1,1}$.  We refer the reader to Section~\ref{ss:stablehomotopyandeffectivity}, for more precise descriptions of these constructions, but collect here notation for the associated loop-suspension adjunctions that arise (see Paragraph~\ref{par:s1vsp1} for more details):
\[ 
\begin{split}
\Sigma^\infty_{S^1}: \Spc(k)_* & \adj \SHS(k): \Omega^\infty_{S^1},  \\
\Sigma^\infty: \Spc(k)_* & \adj \SH(k): \Omega^\infty, \\
\sigma^\infty: \SHS(k) & \adj \SH(k): \omega^\infty .
\end{split}
\]
Further notation will be introduced in the main body of the text.







\subsubsection*{Acknowledgements}
The authors would like to thank M. Levine for questions about the relation between the structure of our proof of the Freudenthal suspension theorem and the classical proof.  We would also like to thank J. Fasel for discussion around Murthy's conjecture and J. Ayoub, M. Hoyois and J. Rognes for helpful comments.

\section{Preliminaries on unstable and stable motivic homotopy theory}
\label{s:preliminaries}
In this section, we recall the basic definitions and properties of unstable and stable motivic homotopy categories in various settings.  Notation in the literature can vary quite a bit, but in an effort to keep our presentation as self-contained as possible, we have attempted to summarize everything we need here.   


\subsection{Motivic localization and the unstable connectivity theorem}
\label{ss:motiviclocalization}
In this section, we introduce the category of motivic spaces and a host of variants: the original story begins with smooth schemes and the Nisnevich topology, but we use variants with different categories of schemes and different Grothendieck topologies (see \ref{par:motivicspaces}-\ref{par:essentiallysmoothbasechange}).  We recall Morel's unstable connectivity theorem and foundational facts about $\aone$-invariant sheaves of groups of various sorts (see \ref{thm:unstableconnectivity}-\ref{cor:sifted-colim-strongly-A1-inv}).  We recall facts about (weakly) nilpotent motivic spaces, associated (functorial) Whitehead and Postnikov towers, and localization of such spaces of with respect to subrings of $\Q$ (see \ref{par:lowercentralseries}-\ref{thm:rlocalizationlocalizeshomotopysheaves}).  Finally, we make a careful study of motivic localization, giving better understanding of when the motivic localization functor preserves limits and colimits of various sorts, extending the results of \cite[\S 6]{MField} (see \ref{par:kanloopgroup}-\ref{prop:simplicialloopssiftedcolimits}).

\begin{entry}[Motivic spaces]
	\label{par:motivicspaces}
A presheaf $\mathscr{X} \in \mathrm{P}(\Sm_k)$ is $\aone$-invariant if for every $U \in \Sm_k$ the projection $U \times \aone \to U$ induces an equivalence $\mathscr{X}(U) \to \mathscr{X}(U \times \aone_k)$.  The $\infty$-category of motivic spaces $\ho{k}$ is the full subcategory of $\Shv_{\Nis}(\Sm_k)$ spanned by $\aone$-invariant Nisnevich sheaves of spaces.  Since homotopy invariance is defined by a small set of conditions, the inclusion $\ho{k} \subset \Shv_{\Nis}(\Sm_k)$ is an accessible localization \cite[\S 3.4]{HoyoisEquiv}.  The category $\ho{k}$ is a presentable $\infty$-category and the inclusion $\ho{k} \subset \mathrm{P}(\Sm_k)$ admits a left adjoint
\[
\mathrm{L}_{mot}: \mathrm{P}(\Sm_k) \longrightarrow \ho{k}
\]
that we call the motivic localization functor.  
\end{entry}

\begin{entry}[Variants]
	\label{par:motivicspacesvariants}
	There are two classes of variants of $\ho{k}$ that will appear in the body of the text: modifications of the category of schemes we consider (e.g., to permit singularities), and alternative Grothendieck topologies $\tau$ on such categories of schemes (typically finer than the Nisnevich topology).  In practice, we will consider Voevodsky's cdh-topology, Scheiderer's real-\'etale topology or Kelly's $\ell$dh topology.  The variants to be discussed will not reappear until Sections~\ref{ss:ret} and \ref{s:symmetricpowers}, and we will include more details and references there as necessary.
	
	Recall that we write $\Sch_k$ for the category of finite type, separated $k$-schemes.  By an admissible subcategory $\mathrm{S}_k$ of $\Sch_k$, we will mean a subcategory such that $X \in \mathrm{S}_k \Longrightarrow \aone_X \in \mathrm{S}_k$,  $X \in \mathrm{S}_k \Longrightarrow U \in \mathrm{S}_k$ for $f: U \to X$ any finite-type \'etale morphism, and $\mathrm{S}_k$ is closed under finite products and coproducts.  Examples include taking $\mathrm{S}_k$ to be the subcategory $\Sch_k^{qp}$ consisting of quasi-projective $k$-schemes or all of $\Sch_k$.  
	
	If $\tau$ is a Grothendieck topology on an admissible subcategory $\mathrm{S}_k$ of $\Sch_k$, then we write $\Spc_{\tau}(\mathrm{S}_k)$ for the full subcategory of the $\infty$-topos $\Shv_{\tau}(\mathrm{S}_k)$ spanned by $\aone$-invariant sheaves of spaces and $\Spc_{\tau}(k)$ for the full subcategory of the $\infty$-topos $\Shv_{\tau}(\Sm_k)$ spanned by $\aone$-invariant sheaves of spaces.  We reserve the notation $\mathrm{L}_{\tau}$ for the localization functor to $\tau$-sheaves of motivic spaces; see Paragraph~\ref{par:retmotivicspaces} for further discussion.
\end{entry}

\begin{entry}[Motivic localization]
	\label{par:motiviclocalization}
Write $\Delta^*_k$ for the cosimplicial affine space with 
\[
\Delta^n_k := \Spec k[x_0,\ldots,x_n]/ \langle \sum_{i=0}^n x_i = 1 \rangle.
\]  
One defines the singular functor by the formula:
\[
\Singaone \mathscr{X} := \colim_{n \in \Delta^{\opcat}} \mathscr{X}(\Delta^n_k \times -).
\]
This functor is left adjoint to the inclusion of $\aone$-invariant presheaves of spaces into presheaves of spaces.  Morel and Voevodsky construct $\mathrm{L}_{mot}$ explicitly as an infinite composition of $\Singaone$ and Nisnevich sheafification \cite[\S2 Lemma 3.20]{MV}.  
\end{entry}

\begin{entry}[Properties of $\mathrm{L}_{mot}$]
	\label{par:lmotproperties}
The simplex category is sifted \cite[Lemma 5.5.8.4]{HTT}, and one deduces that $\mathrm{L}_{mot}$ preserves finite products \cite[C.6]{Hoyois}.  Moreover, the functor $\mathrm{L}_{mot}$ is locally cartesian in the sense that for any diagram $\mathscr{X}_1 \to \mathscr{X}_0 \leftarrow \mathscr{X}_2 \in \mathrm{P}(\Sm_k)$ with $\mathscr X_1, \mathscr X_0 \in \Spc(k)$ the canonical map
\[
\mathrm{L}_{mot}(\mathscr{X}_1 \times_{{\mathscr{X}_0}} \mathscr{X}_2) \longrightarrow \mathscr{X}_1 \times_{{\mathscr{X}_0}} \mathrm{L}_{mot}(\mathscr{X}_2) 
\]
is an equivalence.  In particular, it follows that colimits in $\ho{k}$ are universal (i.e., colimits are stable by pullback) \cite[Proposition 3.15]{HoyoisEquiv}.  While the localization $\mathrm{L}_{mot}$ is known not to be left exact (i.e., $\ho{k}$ is not an $\infty$-topos), universality of colimits in $\ho{k}$ implies that $\ho{k}$ is a semi-topos in the sense of \cite[Definition 6.2.3.1]{HTT}.  Additionally, if $\mathscr{X}$ is a pointed motivic local space, then the connected component of the base-point is again a motivic local space. 
\end{entry}

\begin{entry}[Essentially smooth base change]
	\label{par:essentiallysmoothbasechange}
	We will call a morphism $f: X \to Y$ of schemes {\em essentially smooth} if it can be written as a cofiltered limit $\lim_{\alpha} X_{\alpha} \to Y$ of smooth, affine transition morphisms\NB{this is not quite the same as ``essentially smooth'' in Fabien's book...}.  If $k$ is a perfect field and $L/k$ is a field extension then $f: \Spec L \to \Spec k$ is essentially smooth (see, e.g., \cite[Lemma A.2]{HoyoisAlgCob}).  In that case, the functors $f^*: \mathrm{P}(\Sm_k) \to \mathrm{P}(\Sm_L)$ preserve $\aone$- and Nisnevich local objects \cite[Lemma A.4]{HoyoisAlgCob}.  In this context, suppose we are given a finite diagram $d: I \to \Sm_L$ which is a cofiltered limit of finite diagrams $d_{\alpha}: I \to \Sm_{L_{\alpha}}$.  In that case, the evident map
	\[
	\colim_{\alpha} \Map(\colim {d_{\alpha}},f_{\alpha}^*\mathscr{X}) \longrightarrow \Map(\colim {d},f^*\mathscr{X})
	\]
	is an equivalence by \cite[Lemma A.5]{HoyoisAlgCob}.  We will typically appeal to such results by stating that some construction that we want to perform is ``compatible with essentially smooth base change".  
\end{entry}

Morel's foundational work \cite{MField} studied the interaction between $\mathrm{L}_{mot}$ and connectivity. Unlike the results above which hold over rather general base schemes, the following result requires strong hypotheses on $k$ and is the main reason we assume $k$ is a field (frequently perfect).  

\begin{theorem}[Morel]
	\label{thm:unstableconnectivity}
	Suppose $n \geq -2$ is an integer and assume $k$ is a field and $\mathscr{X} \in \Spc(k)$. Assume either $n \le 0$ or $\mathscr X$ is pulled back from a perfect subfield of $k$.  If $\mathscr{X}$ is $n$-connected, then $\mathrm{L}_{mot}\mathscr{X}$ is also $n$-connected.
\end{theorem}

\begin{proof}
	The statement is trivial for $n < 0$, and for $n = 0$ follows essentially from the construction of the motivic localization functor \cite[\S2 Corollary 3.22]{MV}.  Assuming $k$ is perfect, for $n > 0$ it is \cite[Theorem 6.38]{MField} (though see \cite[Theorem 2.2.12]{AWW} for a detailed proof).  The assertion that the result holds if $\mathscr{X}$ is pulled back from a perfect subfield then follows by compatibility with essentially smooth base change (see Paragraph~\ref{par:essentiallysmoothbasechange}).
\end{proof}

Morel deduces the above results by analyzing the Postnikov tower of connected motivic local spaces.  Key to deducing the above results are the notion of strictly and strongly $\aone$-invariant sheaves of groups, properties of which we now review.

\begin{entry}[Strong and strict $\aone$-invariance]
	\label{par:strongstrict}
	A Nisnevich sheaf of groups $\mathbf{G}$ on $\Sm_k$ is {\em strongly $\aone$-invariant} if the cohomology presheaves $H^i(-,\mathbf{G})$ are $\aone$-invariant for $i = 0,1$ and {\em very strongly $\aone$-invariant} if any $\aone$-invariant quotient sheaf of groups is strongly $\aone$-invariant.  We write $\Grp^{\aone}_k$ for the category of strongly $\aone$-invariant sheaves of groups.  A Nisnevich sheaf of abelian groups is {\em strictly $\aone$-invariant} if the cohomology presheaves $H^i(-,\mathbf{A})$ are $\aone$-invariant for all $i \geq 0$; we write $\Ab^{\aone}_k$ for the category of strongly $\aone$-invariant sheaves of groups.  Over a perfect field $k$, any strongly $\aone$-invariant sheaf of  abelian groups is very strongly $\aone$-invariant \cite[Proposition 2.8]{ABHWhitehead}.  We will also write $\Set^{\aone}_k$ for the category of $\aone$-invariant sheaves of sets.  It will be useful to also consider the categories $\Set_k, \Grp_k, \Ab_k$ of Nisnevich sheaves of sets, groups or abelian groups, with the evident forgetful functors between these categories.
\end{entry}

\begin{rem}
    Morel established that if $k$ is a perfect field, then strongly $\aone$-invariant sheaves of abelian groups are strictly $\aone$-invariant \cite[Corollary 5.45]{MField}, i.e., the evident inclusion
    \[
    \Ab^{\aone}_k \longrightarrow \Grp^{\aone}_k \cap \Ab_k,
    \]
    is an equivalence.  Thus, in practice we will only speak of strong $\aone$-invariance of Nisnevich sheaves of groups.  The category of strictly $\aone$-invariant sheaves $\Ab^{\aone}_k$ is abelian (see Paragraph~\ref{par:homotopytstructureS1} for further discussion of this point).
\end{rem}

\begin{entry}
	\label{par:aonelocalityofspaces}
A pointed connected space $\mathscr{X} \in \Shv_{Nis}(\Sm_k)$ is $\aone$-invariant if and only if the homotopy sheaves $\bpi_i(\mathscr{X})$ are strongly $\aone$-invariant for all $i \geq 1$ \cite[Lemma 2.2.11]{AWW}.  Taken together, the above results also imply that the layers of the Postnikov tower of any connected space $\mathscr{X}$ are themselves $\aone$-invariant. 
\end{entry}

\begin{entry}[Contraction]
	\label{par:contraction}
	Recall that if $\mathbf{G}$ is a pointed presheaf of groups on $\Sm_k$, then its contraction $\mathbf{G}_{-1}$ is the presheaf defined by
	\[
	1 \longrightarrow \mathbf{G}_{-1}(U) \longrightarrow \mathbf{G}(U \times \gm{}) \longrightarrow \mathbf{G}(U),
	\]
	where the second map is induced by $id \times 1: U \to U \times \gm{}$ \cite[p. 33]{MField}.  Morel showed that contraction is an endofunctor of the category of $\Grp^{\aone}_k$ \cite[Lemma 2.32]{MField} that preserves exact sequences \cite[Lemma 7.33]{MField}.  The key result about contractions that we need is: if $\mathscr{X}$ is a pointed connected motivic space, then 
	\[
	\bpi_i(\Omega^{1,1} \mathscr{X}) \weq \bpi_i(\mathscr{X})_{-1},
	\]
	which is established in \cite[Theorem 6.13]{MField}.  
\end{entry}
%
\begin{lem} \label{lem:lmot-hspace-strongly-a1-inv}
Let $k$ be a perfect field and assume that $\mathscr X \in \Shv_\Nis(\Sm_k)$ is an $h$-space (grouplike monoid in the homotopy category). If the sheaf of groups $\bpi_0 \mathscr X$ receives a surjection from a very strongly $\A^1$-invariant sheaf, then $\bpi_0 \Lmot \mathscr X$ is very strongly $\A^1$-invariant.
\end{lem}

\begin{proof}
We know that $\Lmot \mathscr X$ is still an $h$-space (since $\Lmot$ preserves finite products, including the shearing equivalence $\mathscr X \times \mathscr X \to \mathscr X \times \mathscr X$, $(x,y) \mapsto (xy, y)$), and hence $\bpi_0 \Lmot \mathscr X$ is $\A^1$-invariant \cite[Theorem 4.18]{Choudhuryhspace}.
We also know that $\bpi_0 \mathscr X \to \bpi_0 \Lmot \mathscr X$ is surjective \cite[\S 2 Corollary 3.22]{MV}.  Since $\bpi_0 \mathscr X$ receives a surjection from a very strongly $\A^1$-invariant sheaf hence so does the $\A^1$-invariant sheaf $\bpi_0 \Lmot \mathscr X$, which is thus strongly $\A^1$-invariant by the definition of very strongly $\aone$-invariant sheaf of groups (Paragraph~\ref{par:strongstrict}).  In that case, we conclude that $\bpi_0 \Lmot \mathscr X$ is again very strongly $\aone$-invariant as well \cite[Lemma 2.9]{ABHWhitehead}.  
\end{proof}

\begin{lem} 
	\label{lem:sifted-colim-strongly-A1-inv}
Let $k$ be a perfect field.  The forgetful functor 
\[ 
\Grp^{\aone}_k \longrightarrow \Set^{\aone}_k \cap \Grp_k
\] 
preserves geometric realizations of simplicial diagrams $\mathbf G_\bullet$ with $\mathbf G_0$ very strongly $\A^1$-invariant.
\end{lem}

\begin{proof}
Let $|\mathbf G_\bullet|^\Nis$ denote the colimit of the diagram in the statement in the category of Nisnevich sheaves of spaces.  In that case, $\bpi_0 |\mathbf G_\bullet|^\Nis$ is a quotient of the very strongly $\A^1$-invariant sheaf $\mathbf G_0$. In that case, Lemma~\ref{lem:lmot-hspace-strongly-a1-inv} implies that $\mathbf C := \bpi_0 \Lmot |\mathbf G_\bullet|^\Nis$ is grouplike and very strongly $\A^1$-invariant.  We conclude that $\mathbf C$ is the common colimit of $\mathbf G_\bullet$ in the two categories.
\end{proof}

\begin{rem}
In \cite[Lemma 2.9(3)]{ABHWhitehead}, we established that the forgetful functor of Lemma~\ref{lem:sifted-colim-strongly-A1-inv} preserves filtered colimits as well.  There are, however, examples of categories that possess both filtered colimits and reflexive coequalizers (the $1$-categorical version of geometric realizations) yet fail to possess all sifted colimits \cite[Example 1.4]{ARV}.  If we knew that the category of very strongly $\aone$-invariant sheaves had all colimits, then we could appeal to \cite[Theorem 2.1]{ARV} or \cite[Corollary 5.5.8.17]{HTT} to conclude that the forgetful functor preserved sifted colimits. 
\end{rem}

\begin{cor} \label{cor:sifted-colim-strongly-A1-inv}\NB{could go even weaker than $\mathscr E_1$ here...}\NB{had to adapt this because of sifted stuff}
Let $k$ be a perfect field.  
Let $\mathscr X_\bullet$ be a simplicial diagram in $\Mon(\Spc(k))$ such that $\bpi_0(\mathscr{X}_n)$ is a group and $\bpi_0(\mathscr{X}_0)$ is very strongly $\aone$-invariant.
Then
$\bpi_0 \colim \mathscr X$ is grouplike and very strongly $\A^1$-invariant.
\end{cor}

\begin{proof}
As in the proof of Lemma~\ref{lem:sifted-colim-strongly-A1-inv}, $\bpi_0 \colim \mathscr{X}$ is grouplike and $\A^1$-invariant by appeal to \cite[Theorem 4.18]{Choudhuryhspace}.  It follows that $\bpi_0 \colim \mathscr X = \colim \bpi_0 \mathscr X$ in the category of $\A^1$-invariant sheaves of groups.
This colimit is very strongly $\A^1$-invariant by Lemma \ref{lem:sifted-colim-strongly-A1-inv}.
\end{proof}

\subsubsection*{Motivic localization, colimits and limits}
\begin{entry}
	\label{par:kanloopgroup}
	Let $\mathscr X \in \Shv_\Nis(\Sm_k)_*$.
	Then $\Omega \mathscr X$ has the structure of a group in the sense of homotopy theory.
	In fact the induced functor from pointed connected Nisnevich sheaves to groups is an equivalence \cite[Lemma 7.2.2.11]{HTT} with inverse denoted $B$.
	Furthermore given $f: \mathscr E \to \mathscr B$, $\fib(f)$ carries a canonical action by $\Omega \mathscr B$.  Sending $f: \mathscr{E} \to \mathscr{B}$ to $\fib(f)$ with its $\Omega \mathscr{B}$-action defines a functor from the category of spaces over $\mathscr B$ to spaces with $\Omega \mathscr B$-action.  In fact, if $\mathscr B$ is connected then this functor is an equivalence, with inverse obtained by sending a space $\mathscr{F}$ with action of $\Omega \mathscr{B}$ to the map $\mathscr{F} \sslash \Omega \mathscr{B} \to \ast \sslash \Omega \mathscr{B}$ where $\mathscr{F} \sslash \Omega \mathscr{B}$ denotes the (homotopy) quotient, a.k.a. bar construction, i.e., the geometric realization of 
	\[ 
	\xymatrix{
		\mathscr{F} & \ar@<.05 em>[l]\ar@<-.25 em>[l] \mathscr{F} \times \Omega \mathscr{B} & \ar@<.25 em>[l]\ar@<-.05 em>[l]\ar@<-.35 em>[l] \cdots
	}
	\]
	These results carry over to the motivic situation, with one small modification.
\end{entry}

\begin{lem} 
	\label{lem:loopgroup-basics}
	Let $k$ be a perfect field.
	\begin{enumerate}[noitemsep,topsep=1pt]
		\item The functor $\Omega: \Spc(k)_* \to \Mon(\Spc(k))$ induces an equivalence between pointed connected motivic spaces and grouplike monoids in motivic spaces with the additional condition that $\bpi_0$ be strongly $\A^1$-invariant; its inverse is the classifying space functor $B$.
		\item Let $\mathscr B \in \Spc(k)_*$ be connected. There is an equivalence of categories between $\mathscr E \in \Spc(k)_{/\mathscr B}$ and the category of $\mathscr F \in \Spc(k)$ with an action by $\Omega \mathscr B$. It sends $\mathscr E$ to $\fib(\mathscr E \to \mathscr B)$ with the canonical action, and it sends $\mathscr F$ to $\mathscr F \sslash \Omega \mathscr B$.
	\end{enumerate}
\end{lem}
\begin{proof}
	Point (1) is an immediate consequences of the analogous results in $\Shv_\Nis(\Sm_k)$, using the characterization of pointed, connected motivic local spaces in terms of homotopy sheaves (Paragraph \ref{par:aonelocalityofspaces}). For (2), it is clear that the functor is fully faithful. To prove essential surjectivity it suffices to show that if $\mathscr F \in \Shv_\Nis(\Sm_k)_*$ carries an action by $\Omega \mathscr B$ with $\mathscr F, \mathscr B$ $\A^1$-invariant, then also $\mathscr F \sslash \Omega \mathscr B$ is $\A^1$-invariant. This follows (via \cite[Lemma 2.2.10]{AWW}) by considering the fiber sequence 
	\[ 
	\mathscr F \longrightarrow \mathscr F \sslash \Omega \mathscr B \longrightarrow \mathscr B \in \Shv_\Nis(\Sm_k), 
	\] 
	in which by assumption $\mathscr B$ is connected and $\mathscr F, \mathscr B$ are both $\A^1$-invariant.
\end{proof}

\begin{theorem}
	\label{thm:lmotandpullbacks}
	Assume $k$ is a field and suppose we have a cartesian diagram of pointed objects in $\Shv_{\Nis}(\Sm_k)$:
	\[
	\xymatrix{
		\mathscr{X}_{00} \ar[r]^{g_0}\ar[d]^{f_0} & \mathscr{X}_{10} \ar[d]^{f_1} \\
		\mathscr{X}_{01} \ar[r]^{g_1} & \mathscr{X}_{11}
	}
	\]
	that is pulled back from a perfect subfield.  If $\mathscr{X}_{11}$ is connected and $\bpi_0(\Lmot \Omega \mathscr{X}_{11})$ is strongly $\aone$-invariant (in particular, if $\bpi_1(\mathscr{X}_{11})$ is strongly $\aone$-invariant, e.g., trivial), then the diagram
	\[
	\xymatrix{
		\mathrm{L}_{mot}\mathscr{X}_{00} \ar[r]^{g_0}\ar[d]^{f_0} & \mathrm{L}_{mot}\mathscr{X}_{10} \ar[d]^{f_1} \\
		\mathrm{L}_{mot}\mathscr{X}_{01} \ar[r]^{g_1} & \mathrm{L}_{mot}\mathscr{X}_{11}
	}
	\]
	is cartesian as well.
\end{theorem}

\begin{proof}
	We can view $\mathscr X_{00}, \mathscr X_{01}, \mathscr X_{10} \in \Shv_\Nis(\Sm_k)_{/\mathscr X_{11}}$.
	Put $\mathscr G = \Omega \mathscr X_{11}$, $\mathscr F := \fib(\mathscr X_{01} \to \mathscr X_{11})$, $\mathscr H := \fib(\mathscr X_{10} \to \mathscr X_{11})$.
	Under the equivalence of $\Shv_\Nis(\Sm_k)_{/\mathscr X_{11}}$ with spaces with an action by $\mathscr G$ (Lemma~\ref{lem:loopgroup-basics}), $\mathscr X_{01}, \mathscr X_{10}$ correspond respectively to $\mathscr F, \mathscr H$.
	Noting that $\mathscr X_{00}$ is the product of $\mathscr X_{01}, \mathscr X_{10} \in \Shv_\Nis(\Sm_k)_{/\mathscr X_{11}}$ and applying the inverse equivalence, we deduce that \[ \mathscr X_{00} \weq (\mathscr F \times \mathscr H) \sslash \mathscr G. \]
	The same logic applies to $\Lmot \mathscr X_{10} \times_{\Lmot \mathscr X_{11}} \Lmot \mathscr X_{01}$.
	It thus suffices to prove that $\Lmot$ preserves the fiber sequences $\mathscr G \to * \to \mathscr X_{11}$, $\mathscr F \to \mathscr X_{01} \to \mathscr X_{11}$ and $\mathscr H \to \mathscr X_{10} \to \mathscr X_{11}$.
	This is proved in \cite[Theorem 2.3.3]{AWW}.
\end{proof}

\begin{proposition}[Realization fibrations]
	\label{prop:realizationfibrations}
	Assume $k$ is a perfect field and $\mathbf{I}$ is a sifted category.
        Assume $\mathscr{X}^{rs}$, $r,s = 0,1$ are $\mathbf I$-diagrams in $\ho{k}$ fitting into a cartesian square of the form:
	\[
	\xymatrix{
		\mathscr{X}^{00} \ar[r]\ar[d] & \mathscr{X}^{01} \ar[d] \\
		\mathscr{X}^{10} \ar[r] & \mathscr{X}^{11}.
	}
	\]
	Assume that either (a) $\mathscr{X}^{11}$ is objectwise simply connected, or (b) $\mathscr{X}^{11}$ is objectwise connected, $\mathbf I = \Delta^{op}$ and $\bpi_0(\Omega \mathscr{X}^{11}_0)$ is very strongly $\aone$-invariant.
	Then, there is a cartesian square of motivic spaces of the form
	\[
	\xymatrix{
		|\mathscr{X}^{00}| \ar[r]\ar[d] & |\mathscr{X}^{01}| \ar[d] \\
		|\mathscr{X}^{10}| \ar[r] & |\mathscr{X}^{11}|,
	}
	\]
	where $|(\ph)|$ means $\colim_{\mathbf I}$ if $\mathbf I \ne \Delta^{op}$.
\end{proposition}

\begin{proof}
	Write $|-|^{\Nis}$ for the geometric realization computed in $\Shv_{\Nis}(\Sm_k)$.  In that case, observe that the hypothesis that $\mathscr{X}^{11}$ is objectwise connected implies that there is a cartesian square of the form
	\[
	\xymatrix{
		|\mathscr{X}^{00}|^{\Nis} \ar[r]\ar[d] & |\mathscr{X}^{01}|^{\Nis} \ar[d] \\
		|\mathscr{X}^{10}|^{\Nis} \ar[r] & |\mathscr{X}^{11}|^{\Nis};
	}
	\]
	this result was mentioned by Rezk \cite[Theorem 4.4]{Rezkpistar}, see also \cite[Lemma 5.5.6.17]{HA}.  In the special case where $\mathscr{X}^{01} = \mathscr{X}^{10} = \ast$ and $\mathscr{X}^{00} = \Omega \mathscr{X}^{11}$ the same argument yields an equivalence $\Omega |\mathscr{X}^{11}|^{\Nis} \weq |\Omega \mathscr{X}^{11}|^{\Nis}$.
	In case (a), $\bpi_0 \Lmot |\Omega \mathscr{X}^{11}|^{\Nis} = 0$ and so strongly $\A^1$-invariant, whereas in case (b), Corollary \ref{cor:sifted-colim-strongly-A1-inv} shows that $\bpi_0 \Lmot |\Omega \mathscr{X}^{11}|^{\Nis}$ is strongly $\A^1$-invariant.
	The result thus follows by appeal to Theorem~\ref{thm:lmotandpullbacks}.
\end{proof}

\begin{proposition}
	\label{prop:simplicialloopssiftedcolimits}
	If $k$ is a perfect field, then the functor $\Omega$ preserves sifted colimits of simply connected motivic spaces, and geometric realizations of simplicial diagrams $\mathscr X_\bullet$ of connected motivic spaces with $\bpi_1(\mathscr{X}_0)$ very strongly $\aone$-invariant. 
\end{proposition}

\begin{proof}
        There is a cartesian square of the form
	\[
	\xymatrix{
		\Omega \mathscr{X} \ar[r] \ar[d] & \ast \ar[d] \\
		\ast \ar[r]& \mathscr{X}.
	}
	\]  
	The result then follows immediately from Proposition~\ref{prop:realizationfibrations}.
\end{proof}

\begin{rem}
	Corresponding results about preservation of sifted colimits by $\Omega^{1,1}$ and $\Omega^{2,1}$ will be established in Proposition~\ref{prop:tateloopssiftedcolimits}, but the proofs of these results are more indirect.  
\end{rem}

\subsubsection*{Nilpotence and functorial towers}
\begin{entry}[Lower central series]
	\label{par:lowercentralseries}
	Assume $\bpi$ is a strongly $\aone$-invariant sheaf of groups acting on a very strongly $\aone$-invariant sheaf of groups $\mathbf{G}$.  In that case, there exists \cite[Construction 4.14]{ABHWhitehead} an initial strongly $\aone$-invariant quotient of $\mathbf{G}$ on which $\bpi$ acts trivially; we write $\Gamma^{2}_{\bpi}\mathbf{G} \subset \mathbf{G}$ for the (strongly $\aone$-invariant) kernel of the quotient map.  In the case where $\mathbf{G}$ acts on itself by conjugation, we write $[\mathbf{G},\mathbf{G}]_{\aone}$ for this subsheaf and refer to it as the $\aone$-commutator subsheaf ({a priori} it could differ from commutator subgroup sheaf).  Inductively, one then defines the {\em $\aone$-lower central series} $\Gamma^i_{\bpi}\mathbf{G}$ for the action of $\bpi$ on $\mathbf{G}$ in the usual way: $\Gamma^{1}_{\bpi}\mathbf{G} = \mathbf{G}$ and $\Gamma^{i+1}_{\bpi}\mathbf{G} = \Gamma^2_{\bpi}\Gamma^{i}_{\bpi}\mathbf{G}$; when $\bpi = \mathbf{G}$ acting by conjugation we simply drop $\bpi$ from the notation.  An action of $\bpi$ on $\mathbf{G}$ is called {\em $\aone$-nilpotent} if and only if $\Gamma^i_{\bpi}\mathbf{G} = 1$ after finitely many steps, i.e., the $\aone$-lower central series is finite.
\end{entry}



\begin{proposition}
	\label{prop:functorialwhiteheadtowers}
	If $\mathscr{X} \in \Spc(k)_{*}$ is connected and $\bpi_1(\mathscr{X})$ is nilpotent, then $\mathscr{X}$ admits a functorial Whitehead tower.  More precisely, there exist a weakly increasing sequence of integers $1 \leq n_1 \leq n_2 \leq \cdots$, strictly $\aone$-invariant sheaves $\mathbf{A}_i(\mathscr{X})$, and spaces $\mathscr{X} \langle i \rangle$, $i \geq 0$, fitting into fiber sequences of the form
	\[
	\mathscr{X}\langle i+1 \rangle \longrightarrow \mathscr{X} \langle i \rangle \longrightarrow K(\mathbf{A}_i(\mathscr{X}),n_i)
	\]
	such that: 
	\begin{enumerate}[noitemsep,topsep=1pt]
		\item $\mathscr{X}\langle 0 \rangle = \mathscr{X}$,
		\item $\mathscr{X}\langle i \rangle$ is $n_i$-connective, and 
		\item $n_i \to \infty$ as $i \to \infty$.
	\end{enumerate}
\end{proposition}

\begin{proof}
	Set $\mathscr{X}\langle 0 \rangle = \mathscr{X}$.  We define $\mathscr{X} \langle i \rangle$ inductively as follows.  Since $\mathscr{X} \langle i \rangle$ is $n_i$-connective, its first non-trivial homotopy sheaf appears in degree $n_i$.  If $n_i = 1$, then we define $\mathscr{X} \langle i+1 \rangle$ to be the fiber of the composite map:
	\[
	\mathscr{X} \langle i \rangle \longrightarrow B\bpi_1(\mathscr{X}\langle i \rangle) \longrightarrow B\bpi_1(\mathscr{X})^{ab}_{\A^1};
	\]
	Since $\bpi_1(\mathscr{X})$ is nilpotent, it is very strongly $\aone$-invariant by \cite[Proposition 2.11]{ABHWhitehead}.  In that case, the kernel of the map $\bpi_1(\mathscr{X}) \to \bpi_1(\mathscr{X})^{ab}_{\aone}$ is again nilpotent and thus itself very strongly $\aone$-invariant.  Moreover, the resulting $\aone$-derived series necessarily terminates after finitely many steps, again because of the definition of nilpotence.  If $n_i > 1$, then we define $\mathscr{X} \langle i+1 \rangle$ to be the fiber of the corresponding $k$-invariant in the Postnikov tower.  The assumption that $\mathscr{X}$ has $\aone$-nilpotent fundamental sheaf of groups guarantees that $n_i$ is eventually greater than $1$ and we conclude.
\end{proof}

\begin{entry}[Nilpotent motivic spaces]
	\label{par:simpleandnilpotent}
	A pointed, connected motivic space $\mathscr{X}$ is called {\em simple} if $\bpi_1(\mathscr{X})$ is abelian and the action of $\bpi_1(\mathscr{X})$ on higher homotopy sheaves is trivial.  Any motivic $h$-space is simple.  A pointed connected motivic space is called (weakly) {\em nilpotent} if $\bpi_1(\mathscr{X})$ is $\aone$-nilpotent and acts (locally) $\A^1$-nilpotently on higher homotopy sheaves \cite[Definition 3.3.1]{AFHLocalization}.  If $\mathscr{X}$ is a nilpotent motivic space, then the usual Postnikov tower of $\mathscr{X}$ admits a functorial principal refinement \cite[Theorem 5.1]{ABHWhitehead}.  For later use, and to make the presentation more self-contained we summarize this in the following statement.
\end{entry}

\begin{proposition}
	\label{prop:functorialprincipalrefinement}
	If $\mathscr{X}$ is a nilpotent motivic space, then the Postnikov tower admits a functorial principal refinement.  More precisely, there exist a weakly increasing sequence of integers $n_i \geq 2$, strictly $\aone$-invariant sheaves $\mathbf{A}_i(\mathscr{X})$, morphisms $\mathscr{X} \to \mathscr{X}_i$, and fiber sequences of the form
	\[
	\mathscr{X}_{i+1} \longrightarrow \mathscr{X}_i \longrightarrow K(\mathbf{A}_i(\mathscr{X}),n_i) 
	\]
	such that: 
	\begin{enumerate}[noitemsep,topsep=1pt]
		\item $\mathscr{X}_0 = \ast$
		\item the induced map $\mathscr{X} \to \lim_i \mathscr{X}_i$ is an equivalence
		\item $n_i \to \infty$ as $i \to \infty$.
	\end{enumerate}
\end{proposition}

\subsubsection*{Localization at primes}
Let $R \subset \mathbb{Q}$ be a subring. We will need to consider unstable $R$-localization of $\ho{k}$ as discussed in \cite[Definition 4.3.1]{AFHLocalization}; we write $\mathscr{X}_R$ or $\mathscr{X} \tensor R$ for what we called $\mathrm{L}_R\mathscr{X}$ in \cite{AFHLocalization}.  In particular, we recall the following fact. 

\begin{theorem}[{\cite[Theorem 4.3.9]{AFHLocalization}}]
	\label{thm:rlocalizationlocalizeshomotopysheaves}
	Assume $k$ is a field.  If $\mathscr{X} \in \ho{k}_*$ is a weakly nilpotent motivic space, then $\mathscr{X}_R$ is connected and weakly nilpotent and the map
	\[
	\bpi_i(\mathscr{X}) \tensor R \longrightarrow \bpi_i(\mathscr{X}_R)
	\]
	is an isomorphism for every $i \geq 1$ (where for $i = 1$, the tensor product is interpreted as in \textup{\cite[Theorem 4.3.7]{AFHLocalization}})
\end{theorem}

\subsection{Stable motivic homotopy theory and effectivity}
\label{ss:stablehomotopyandeffectivity}
In this section, we begin by quickly reviewing definitions of various motivic stable homotopy categories and adjunctions between them; we recall some facts about the relationship between $S^1$-spectra and $S^1$-infinite loop spaces, a.k.a. commutative monoids in spaces, and we include results here about when the right adjoints preserve colimits of various types (see \ref{par:presentablysymmetricmonoidalcategories}-\ref{rem:mappingspacesandsiftedcolimits}).  We then introduce notation about module spectra, slices and (very) effective spectra of various types in the associated categores (see \ref{par:modulespectra}-\ref{par:veryeffective}).  Then, we recall various facts about localization with respect to subrings of $\Q$ in motivic stable categories (see \ref{prop:dualizablesgenerate}-\ref{prop:pdivisibility}).  Next, we recall useful facts about generation of $t$-structures and use these facts to establish some strong generation results (see \ref{par:homotopytstructureS1}-\ref{prop:generationundercolimits}).  We recall results about rationalization of motivic stable homtopy categories and associated decompositions (see \ref{par:stablerationalization}-\ref{thm:rationalizedpluspart}).  Finally, this section closes with a discussion of conservativity of $\gm{}$-stabilization (see \ref{par:gmstabilization}-\ref{thm:conservativityofgmstabilization}).

\begin{entry}[Presentably symmetric monoidal categories]
	\label{par:presentablysymmetricmonoidalcategories}
	Write $\mathrm{Pr}^L$ for the $\infty$-category whose objects are presentable $\infty$-categories and morphisms are colimit preserving functors \cite[Definition 5.5.3.1]{HTT}.  The category $\mathrm{Pr}^L$ admits limits \cite[Proposition 5.5.3.13]{HTT} and colimits \cite[Theorem 5.5.3.18]{HTT}.  The category $\mathrm{Pr}^L$ can be equipped with a symmetric monoidal structure $\mathrm{Pr}^{L,\tensor}$ \cite[Proposition 4.8.1.15]{HA}.  By a {\em presentably symmetric monoidal category} we will mean a commutative algebra in $\mathrm{Pr}^{L,\tensor}$ and we write $\mathrm{CAlg}(\mathrm{Pr}^{L,\tensor})$ for the category of commutative algebra objects. Recall that a presentably symmetric monoidal $\infty$-category is the same thing as a symmetric monoidal $\infty$-category $\mathscr C$ such $\mathscr{C}$ is presentable and the tensor product $\mathscr C \times \mathscr C \to \mathscr C$ preserves colimits in each variable separately \cite[Notation 4.8.1.2(iv)]{HA}.
\end{entry}

\begin{entry}[Stable homotopy categories]
	\label{par:stablehomotopycategories}
	Given any presentably symmetric monoidal $\infty$-category $\mathcal{C}$ and a set of objects $X$ in $\mathcal{C}$, we can invert the tensor product with the objects $X$.  More precisely there exists a new presentably symmetric monoidal infinity category $\mathcal{C}[X^{-1}]$ with the following universal property in the infinity category $\mathrm{CAlg}(\mathrm{Pr}^{L,\tensor})$: given any functor $\mathcal{C} \to \mathcal{D}$ such that $f(x)$ is invertible for every $x \in X$, it factors uniquely though $\mathcal{C}[X^{-1}]$; we refer the reader to \cite[\S 2.1]{Robalo} for further details of this construction.
\end{entry}

\begin{entry}[$S^1$ and $\pone$-spectra]
	\label{par:s1vsp1}
	The category of pointed motivic spaces $\ho{k}_*$ is a presentably symmetric monoidal infinity category via the smash product of pointed spaces.  We define
	\[
	\begin{split}
	\SH^{S^1}(k) &:= \ho{k}_*[(S^1)^{-1}] \\
	\SH(k) &:= \ho{k}_*[({\pone})^{-1}].
	\end{split}
	\]
	These categories come equipped with canonical adjunctions
	\[
	\begin{split}
	&\xymatrix{
	\ho{k}_* \ar@<.5ex>[r]^{\Sigma^{\infty}_{S^1}} & \ar@<.5ex>[l]^{\Omega^{\infty}_{S^1}} \SH^{S^1}(k)
	} \\
	&\xymatrix{
	\ho{k}_* \ar@<.5ex>[r]^{\Sigma^{\infty}} & \ar@<.5ex>[l]^{\Omega^{\infty}} \SH(k)
	}
	\end{split}
	\]	
	where $\Sigma^{\infty}_{S^1}$ and $\Sigma^{\infty}$ are symmetric monoidal.  The categories $\SH^{S^1}(k)$ and $\SH(k)$ are stable $\infty$-categories and are generated under sifted colimits by $S^1$ or $\pone$-desuspensions of suspension spectra of smooth schemes.  We refer the reader to \cite[Proposition 6.4]{HoyoisEquiv} for further details.  For the moment, we observe only that   the passage from $S^1$-spectra to $\pone$-spectra also fits into an adjunction.
	\[
	\xymatrix{
		\SH^{S^1}(k) \ar@<.5ex>[r]^{\sigma^{\infty}} & \ar@<.5ex>[l]^{\omega^{\infty}} \SH(k)
	} 
	\]
	having properties analogous to those mentioned above. 
	(Indeed since $\mathbb P^1 \weq S^1 \wedge \Gm$ we have $\SH(k) \weq \SH^{S^1}(k)[\Gm^{-1}]$.)
\end{entry}

\begin{entry}
	\label{par:emsheaves}
	If $\mathbf{A}$ is a strictly $\aone$-invariant sheaf of groups, we will write $H\mathbf{A}$ for the Eilenberg--Mac Lane $S^1$-spectrum determined by $\mathbf{A}$.
\end{entry}

\subsubsection*{From $S^1$-spectra to commutative monoids and back}
Given an $S^1$-spectrum, we can consider the corresponding zero space, which is a connective commutative monoid.  When the source is motivic local, the associated commutative monoid has a strong $\aone$-invariance property.  The following result records the basic facts about the corresponding functor of taking $S^1$-infinite loops, especially as regards preservation of colimits.  

\begin{theorem} \label{thm:SHS-CMon-pres-colim}
	Let $k$ be a perfect field.  
	\begin{enumerate}[noitemsep,topsep=1pt]
		\item The functor
		\[ 
		\Omega^{\infty}_{S^1}: \SH^{S^1}(k)_{\ge 0} \longrightarrow \mathrm{CMon}(\Spc(k)) 
		\] 
		is fully-faithful, with image consisting of those grouplike commutative monoids such that $\bpi_0$ is strongly $\A^1$-invariant.
		\item The image subcategory  in the previous point is closed under formation of colimits in $\mathrm{CMon}(\Spc(k))$.  In particular $\SH^{S^1}(k)_{\ge 0} \to \mathrm{CMon}(\Spc(k))$ preserves colimits.
	\end{enumerate}
\end{theorem}

\begin{proof}
	The assertion that $\Omega^{\infty}_{S^1}$ is fully faithful as well as the characterization of its essential image is the content of \cite[Proposition 3.1.13 and Corollary 3.1.15]{EHKSYdeloop}.
	
	To establish the second point, it suffices to prove closure under finite coproducts, filtered colimits and geometric realizations \cite[Lemma 5.5.8.13]{HTT}\NB{better ref?}.  We must show that such colimits in $\CMon(\Spc(k))$ preserve the subcategory of commutative monoids with $\bpi_0$ grouplike and strongly $\A^1$-invariant.
	The case of finite coproducts (i.e., finite sums) is clear, and the same holds for filtered colimits.  For closure under geometric realizations, begin by noting that $\bpi_0$ is an abelian strongly $\aone$-invariant sheaf of groups and thus very strongly $\aone$-invariant by \cite[Proposition 2.8]{ABHWhitehead}.  In that case, we may appeal to Corollary \ref{cor:sifted-colim-strongly-A1-inv} to conclude.
\end{proof}

\begin{cor} \label{cor:CMon-pushout-mixed}
	Let $k$ be a perfect field.  By appeal to ~\textup{Corollary \ref{thm:SHS-CMon-pres-colim}} the functor $\Omega^{\infty}_{S^1}$ is part of an adjunction:
	\[ 
	\xymatrix{
		B^{\infty}_{mot}: \mathrm{CMon}(\Spc(k)) \ar@<.5ex>[r] & \ar@<.5ex>[l] \SH^{S^1}(k)_{\ge 0}: \Omega^{\infty}_{S^1}. 
	}
	\]
	Consider a pushout square in $\mathrm{CMon}(\Spc(k))$:
	\begin{equation*}
		\begin{CD}
			A @>>> B \\
			@VVV  @VVV \\
			C @>>> D.
		\end{CD}
	\end{equation*}
	If $\bpi_0(B), \bpi_0(C)$ are and strongly $\A^1$-invariant, then 
	\[ 
	D \weq \Omega^{\infty}_{S^1}(B^{\infty}_{mot}C \amalg_{B^{\infty}_{mot}A} B^{\infty}_{mot}B). \]
\end{cor}

\begin{proof}
	Since $\Omega^{\infty}_{S^1}$ preserves colimits by appeal to Corollary \ref{thm:SHS-CMon-pres-colim}, the assertion is just that $\bpi_0(D)$ is grouplike and strongly $\A^1$-invariant.
	Writing the pushout as a geometric realization, Corollary \ref{cor:sifted-colim-strongly-A1-inv} reduces this to showing that $\bpi_0(C \oplus D)$ is commutative and strongly $\A^1$-invariant, which is clear.
\end{proof}

\begin{proposition}
	\label{prop:tateloopssiftedcolimits}
	Assume $k$ is a field.
	\begin{enumerate}[noitemsep,topsep=1pt]
		\item The functor $\Omega^{\infty}_{S^1}: \SH^{S^1}(k)_{\geq 0} \longrightarrow \ho{k}_*$ preserves sifted colimits.
		\item If $k$ is furthermore perfect, then the functor $\Omega^{1,1}$ on $\ho{k}_*$ preserve sifted colimits of $1$-connected spaces.
		\item If $k$ is perfect, then the functor $\Omega^{2,1}$ preserves sifted colimits of $1$-connected spaces, while the functor $\Omega^{2,1}\Sigma^{2,1}$ preserves sifted colimits of connected spaces.
	\end{enumerate}
\end{proposition}

\begin{proof}
	{\bf Point 1}.  Begin by observing that $\aone$-invariant Nisnevich sheaves of spectra are closed under colimits in the larger category of Nisnevich sheaves of spectra.  It therefore suffices to establish the result for Nisnevich sheaves of spectra. To establish the result for Nisnevich sheaves of spectra, it suffices to check stalkwise, in which case we may reduce to the case of ordinary spectra, which may be found in \cite[Proposition 1.4.3.9]{HA}.
	
	\noindent {\bf Point 2}. For the second statement, suppose $\mathscr{X}: \mathbf I \to \ho{k}_*$ is a sifted diagram of pointed $1$-connected motivic spaces.  Recall that $\Omega^{1,1}$ preserves connectivity (see Paragraph~\ref{par:contraction}).  In that case, we can consider the $j$-connective covers $\mathscr{X} \langle j \rangle$, which fit into principal fiber sequences of the form:
	\[
	\mathscr{X}\langle j+1 \rangle \longrightarrow \mathscr{X}\langle j \rangle \longrightarrow K(\bpi_j \mathscr{X},j)
	\]
	where $j \geq 2$ by assumption.  The above is a $\mathbf I$-diagram of fiber sequences with $1$-connected base, so it follows from Proposition~\ref{prop:realizationfibrations} (which uses the assumption $k$ is perfect because of the implicit appeal to the unstable connectivity theorem) that 
	\[
	|\mathscr{X} \langle i+1 \rangle| \longrightarrow |\mathscr{X} \langle i \rangle| \longrightarrow |K(\bpi_i\mathscr{X},i)|
	\]
	is a fiber sequence, where $|(\ph)|$ means $\colim_{\mathbf I}(\ph)$.  
	
	We first establish the result for $|K(\bpi_j\mathscr{X},j)|$ by comparison with the $S^1$-stable case.  Indeed, each $K(\bpi_j\mathscr{X},j)$ can be identified with $\Omega^{\infty}_{S^1} \Sigma^{j}H\bpi_j\mathscr{X}$ for $j \geq 2$, and these identifications give rise to a diagram of Eilenberg--MacLane spectra (e.g. use \cite[Proposition 1.4.3.9]{HA} again).  By the first point, the functor $\Omega^{\infty}_{S^1}$ preserves geometric realizations.  Granted that observation, we conclude that:
	\[
	\Omega^{1,1} |K(\bpi_j\mathscr{X},j)| \weq \Omega^{1,1} |\Omega^{\infty}_{S^1}  \Sigma^{j}H\bpi_j\mathscr{X}| \weq \Omega^{1,1}\Omega^{\infty}_{S^1} | \Sigma^{j}H\bpi_j\mathscr{X}| \weq \Omega^{\infty}_{S^1} \Omega^{1,1} |\Sigma^{j}H\bpi_j\mathscr{X}|,
	\]
	where the second equivalence follows from the fact established in the first point that $\Omega^{\infty}_{S^1}$ preserves sifted colimits.  Now, $\Omega^{1,1}$ commutes with colimits of spectra and therefore, we conclude that
	\[
	\Omega^{\infty}_{S^1} \Omega^{1,1} |\Sigma^{i}H\bpi_i\mathscr{X}| \weq \Omega^{\infty}_{S^1} |\Omega^{1,1}\Sigma^{i}H\bpi_i\mathscr{X}| \weq |\Omega^{1,1}K(\bpi_i\mathscr{X},i)|.
	\]
	where the last equivalence follows by repeating in the opposite order the equivalences described two displays above.
	
	Since applying $\Omega^{1,1}$ preserves fiber sequences and connectivity, another application of Proposition~\ref{prop:realizationfibrations} implies that there is a fiber sequence of the form
	\[
	|\Omega^{1,1} \mathscr{X} \langle i+1 \rangle| \longrightarrow |\Omega^{1,1} \mathscr{X} \langle i \rangle| \longrightarrow |\Omega^{1,1}K(\bpi_i\mathscr{X},i)|.
	\]
	In that case, there is a commutative diagram of fiber sequences of the form
	\[
	\xymatrix{
		|\Omega^{1,1} \mathscr{X} \langle i+1 \rangle| \ar[r]\ar[d] &  |\Omega^{1,1} \mathscr{X} \langle i+1 \rangle| \ar[r]\ar[d] &  |\Omega^{1,1}K(\bpi_i\mathscr{X},i)| \ar[d]  \\
			\Omega^{1,1}|\mathscr{X} \langle i+1 \rangle| \ar[r] &  \Omega^{1,1}|\mathscr{X} \langle i \rangle| \ar[r] &  \Omega^{1,1}|K(\bpi_i\mathscr{X},i)|
	}
	\]
	The right vertical map is an equivalence by the preceding discussion.  Taking fibers horizontally and appealing to \cite[Proposition 3.1]{ABHWhitehead}, we conclude that the map from the fiber $\mathscr{F}\langle i+1\rangle$ of the left vertical map to the fiber $\mathscr{F} \langle i \rangle$ of the middle vertical map is an equivalence.  Note that $\mathscr{F}\langle i\rangle$ is at least $i$-connective, so the equivalence above followed by induction shows that $\mathscr{F} \langle i \rangle$ is $\infty$-connective and thus contractible, which is what we wanted to show.
	
	\noindent {\bf Point 3}.  Writing $\Omega^{2,1} = \Omega\Omega^{1,1}$ and combining the conclusion of Point (2) with Proposition~\ref{prop:simplicialloopssiftedcolimits}, we see that $\Omega^{2,1}$ preserves sifted colimits of $1$-connected spaces.  Since $\Sigma^{2,1}$ preserves colimits, and since $\Sigma^{2,1}$ sends $0$-connected spaces to $1$-connected spaces by the unstable connectivity theorem \ref{thm:unstableconnectivity}, we also conclude that $\Omega^{2,1}\Sigma^{2,1}(-)$ also preserves sifted colimits of $0$-connected spaces.
\end{proof}

\begin{rem}
	\label{rem:mappingspacesandsiftedcolimits}
	Assume $k$ is a field and $Y$ is a pointed open subscheme of $\aone$.  The same style of proof can be used to establish that the internal mapping space $\iMap_*(Y,-)$ preserves sifted colimits.  Indeed, by \cite[Remark 2.33]{MField}, given a strongly $\aone$-invariant sheaf $\mathbf{G}$, we may define
	\[
	\mathbf{G}^{(Y)}(U) := \ker(\mathbf{G}(Y \times U) \to \mathbf{G}(U)).
	\]
	This formula defines a sheaf of groups, and the same argument as in \cite[Lemma 2.32]{MField} shows that $B\mathbf{G}^{(Y)}$ can be identified with the connected component of the identity map in $\iMap_*(Y,B\mathbf{G})$ and thus $\mathbf{G}^{(Y)}$ is again strongly $\aone$-invariant.  In that case, the evident analog of \cite[Lemma 6.14]{MField} holds using homotopy purity replacing $\gm{}$ by $Y$ and thus the proof of \cite[Theorem 6.13]{MField} goes through with $\gm{}$ replaced by $Y$.  In particular, $\iMap_*(Y,-)$ preserves connectivity so we can repeat the argument above.
\end{rem}

\subsubsection*{Slices and (very) effective motivic spectra}
\begin{entry}
	\label{par:modulespectra}
	If $E \in \mathrm{CAlg}(\SH(k))$, we write $\Mod_{E}(k)$ for the category of $E$-module spectra.  Of particular interest will be the case where $E = H\Z$, Voevodsky's motivic Eilenberg--Mac Lane spectrum or $H\Z[1/p]$ where $p$ is the characteristic exponent of $k$ (note: a minor conflict arises between this notation and that for the Eilenberg--Mac Lane  $S^1$-spectrum attached to a strictly $\aone$-invariant sheaf; i.e., Paragraph~\ref{par:emsheaves}---we trust that this will not lead to confusion).  By \cite[Theorem 1]{RondigsOestvaermodules}, we know that if $k$ has characteristic $0$, then $\Mod_{H\Z}(k)$ is equivalent to Voevodsky's (unbounded) derived category of motives $\DM(k)$.  This result was extended to an equivalence of $\Mod_{H\Z[1/p]}(k)$ and Voevodsky's (unbounded) derived category of motives with $\Z[1/p]$-coefficients $\DM(k,\Z[1/p])$ in \cite[Corollary 5.3]{ElmantoKolderup}.
\end{entry}

\begin{entry}[Essentially smooth base change revisited] 
	\label{par:essentiallysmoothbasechangeII}
	The results of Paragraph~\ref{par:essentiallysmoothbasechange} carry over to $\SH^{S^1}(k)$ or $\SH(k)$; see \cite[Lemma A.7]{HoyoisAlgCob} for details.  These results also carry over to $\Mod_{H\Z[1/p]}$ by the same result and the discussion of Paragraph~\ref{par:modulespectra}.
\end{entry}

\begin{entry}[Effective spectra]
	\label{par:effective}
	We write $\SH^{S^1}(k)(d)$ for the localizing subcategory generated by $\Sigma^{d,d} \Sigma^{\infty}_{S^1} X_+$ for $X \in \Sm_k$; we refer to this category as the subcategory of $d$-effective $S^1$-spectra.  We also write $\SH(k)^{\eff}$ for the localizing subcategory of $\SH(k)$ generated by spectra of the form $\Sigma^{\infty}_{\pone} X_+$ (i.e., no $\gm{}$-desuspensions) and we write $\SH(k)^{\eff}(n)$ for the category $\Sigma^{2n,n}\SH(k)^{\eff}$; this category can be identified with $\Sigma^{n,n}\SH(k)^{\eff}$ because it is triangulated.    
\end{entry}

\begin{entry}[Slices]
	\label{par:slices}
	The inclusion $i_n: \SH(k)^{\eff}(n) \hookrightarrow \SH(k)^{\eff}$ admits a right adjoint $r_n$ as a consequence of Neeman's Brown representability theorem \cite[Theorem 8.4.4]{Neeman}.  We set $f_n := i_n \circ r_n$.  There are natural transformations $f_n \to f_m$ whenever $m < n$ and a tower of the form
	\[
		\dots \longrightarrow f_n E \longrightarrow f_{n-1} E \longrightarrow \cdots \longrightarrow f_0 E \longrightarrow f_{-1} E \longrightarrow \dots \longrightarrow E.
	\]
	This tower is called the {\em slice tower}; the layers $f_nE$ are {\em effective covers} of $E$ and the {\em slices} $s_n E$ are defined as the cofibers of $f_{n+1} E \to f_n E$.  
	
	Voevodsky showed \cite[Theorem 6.6]{VZeroSlice} that $s_0(\1_k) \weq H\Z$ when $k$ is a field having characteristic $0$ and Levine generalized this result (by a completely different argument) to $k$ an arbitrary field \cite[Theorem 10.5.1]{Levineconiveautower}.   Furthermore, the slice filtration is multiplicative, so the slices $s_iE$ have the structure of a module over $s_0 \1$  \cite[Theorem 3.6.13(6)]{Pelaez}.  
\end{entry}

\begin{entry}[Very effective spectra]
	\label{par:veryeffective}
	Following Spitzweck--{\O}stv{\ae}r \cite{SpitzweckOstvaer}, define $\SH(k)^{\veff}$ to be the smallest subcategory closed under colimits and extensions and containing objects of the form $\Sigma^{n,0} \Sigma^{\infty}_{\pone} X_+$, $n \geq 0$.  More generally, if $E \in \CAlg(\SH(k))$, we write $\Mod_E(k)^{\veff}$ for the subcategory of $\Mod_E(k)$ generated under colimits and extensions by $E \wedge \Sigma^{\infty}_{\pone}X_+$, $X \in \Sm_k$.
\end{entry}

\subsubsection*{Recollections on inverting primes, strong dualizability}
\begin{entry}
	\label{par:stablerationalization}
Given a subring $R \subset \Q$, we can consider variants of all the stable categories above where with $R$-coefficients, e.g., $\SH^{S^1}(k)_R$, $\SH(k)_R$, etc.  Of particular interest will be the cases $R = \Q$, $R = \Z[\frac{1}{2}]$ and $R = \Z[\frac{1}{p}]$ where $p$ is the characteristic exponent of $k$.
\end{entry}

\begin{proposition}
	\label{prop:dualizablesgenerate}
	Assume $k$ is a field with characteristic exponent $p$.
	\begin{enumerate}[noitemsep,topsep=1pt]
		\item For any $U \in \Sm_k$, $\Sigma^{\infty}U_+$ is strongly dualizable in $\SH(k)[\frac{1}{p}]$.
		\item The category $\SH(k)[\frac{1}{p}]$ is generated under colimits by $\Sigma^{\infty}X_+$ for $X \in \Sm_k$ projective or, more generally, by strongly dualizable objects.  
	\end{enumerate}
\end{proposition}

\begin{proof}
	The case where $k$ has characteristic $0$ is \cite[Th{\'e}or{\'e}me 1.4,2.2]{Riou} appealing to resolution of singularities.  By appeal to Gabber's refined alternations, Riou \cite[B.1-2]{LYZR} extended this argument to the case where $k$ was perfect.  The assumption $k$ perfect is removed in \cite[Theorem 3.2.1]{ElmantoKhan}.
\end{proof}

We will also need some comparisons between unstable and stable localizations.  By checking stalkwise, the following result follows immediately from the corresponding classical statement (which can be proved by, e.g., observing that $\mathscr X$ below has vanishing $p$-completion).

\begin{proposition} 
	\label{prop:pdivisibility-unstable-implies-stable}
	Assume $k$ is a field and $\mathscr{X}$ is a pointed $1$-connected motivic space.  If for all $i$, $\bpi_i(\mathscr{X})$ is uniquely $p$-divisible (for some $p \in \Z$), then $\Sigma^{\infty}_{S^1} \mathscr{X} \weq \Sigma^{\infty}_{S^1}\mathscr{X}[\frac{1}{p}]$. 
\end{proposition}

\begin{proposition}
	\label{prop:pdivisibility}
	Assume $k$ is a field and $\mathscr{X}$ is a pointed $1$-connected motivic space.  If $\Sigma^{\infty}_{S^1}\mathscr{X} \weq \Sigma^{\infty}_{S^1}\mathscr{X}[\frac{1}{p}]$, then $\bpi_i(\mathscr{X})$ is uniquely $p$-divisible for each $i \geq 1$.
\end{proposition}

\begin{proof}
	Consider the $\Z[\frac{1}{p}]$-localization described in Theorem~\ref{thm:rlocalizationlocalizeshomotopysheaves}, which has the property that $\bpi_i(\mathscr{X}_{\Z[1/p]})$ tensors $\bpi_i(\mathscr{X})$ with $\Z[1/p]$.  The map $\Sigma^{\infty}_{S^1} \mathscr{X} \to \Sigma^{\infty}_{S^1}(\mathscr{X}_{\Z[1/p]})$ is a $\Z[1/p]$-equivalence of $p$-periodic spectra (by assumption for the first and by the preceding proposition for the second statement), whence an equivalence.  The Whitehead theorem \cite[Theorem 5.2]{ABHWhitehead} implies that $\mathscr{X} \to \mathscr{X}_{\Z[1/p]}$ is an equivalence as well.
\end{proof}

\subsubsection*{Motivic spectra and $t$-structures}
We review here some facts about generation of $t$-structures, which seems to have been independently discovered by several authors including in unpublished work of Morel \cite[{\S 2.1.3} p. 280]{AyoubI} and \cite{AJS}.  Lurie has also given a treatment of these ideas in the context of stable presentable $\infty$-categories.  

\begin{entry}[Homotopy $t$-structure on $S^1$-spectra]
	\label{par:homotopytstructureS1}
	The category $\SH^{S^1}(k)$ can be equipped with the homotopy $t$-structure: $\SH^{S^1}(k)_{\leq 0}$ consists of those spectra that are coconnective (i.e. $\bpi_i = 0$ for $i>0$), and $\SH^{S^1}(k)_{\geq 0}$ consists of those spectra that are connective (i.e. $\bpi_i = 0$ for $i<0$) \cite[Lemma 6.2.11]{MStable}.  Equivalently, the category $\SH^{S^1}(k)_{\geq 0}$ is generated under colimits and extensions by $\Sigma^{\infty} X_+$, $X \in \Sm_k$.  The abelian category $\SH^{S^1}(k)^{\heartsuit}$ can be identified with the category of strictly $\aone$-invariant sheaves \cite[Lemma 6.2.13]{MStable}.
\end{entry}
	
\begin{entry}[Homotopy $t$-structure on $\SH(k)$ and homotopy modules]
	  \label{par:homotopytstructureP1}
	  The category $\SH(k)$ can be equipped with a homotopy $t$-structure \cite[Theorem 5.2.3]{MIntro}.  In this case, the non-negative part of the $t$-structure is generated under colimits and extensions by $\Sigma^{p,q}\Sigma^{\infty}_{\pone} X_+$ with $p-q \geq 0$ (see \cite[\S 2.1]{HoyoisAlgCob}).  The category $\SH(k)^{\heartsuit}$ can be described in terms of {\em homotopy modules} via \cite[Theorem 5.2.6]{MIntro}.  Recall that a homotopy module consists of a $\Z$-graded strictly $\aone$-invariant sheaf $\mathbf{M}_*$ together with isomorphisms $\mu_n: \mathbf{M}_n \weq (\mathbf{M}_{n+1})_{-1}$ (see \ref{par:contraction}).  
\end{entry}

\begin{entry}
	\label{par:milnorwittktheory}
	The homotopy module $\bpi_0(\1_k)$ has been determined by Morel \cite[Theorem 6.4.1]{MIntro} and \cite[Theorem 6.40]{MField}: it is the graded unramified Milnor--Witt K-theory sheaf $\K^{MW}_{\ast}$.  Our conventions for Milnor--Witt K-theory follow Morel's book \cite[3.1]{MField}.  For an extension $K/k$, the sections of $\bpi_0(\1_k)(K)$ are generated by the motivic Hopf map $\eta$ of degree $-1$, and symbols $[u]$, for $u \in K^{\times}$, satisfying relations.  For a unit $u$, we write $\langle u \rangle = 1 + \eta [u]$, and the map $u \to \langle u \rangle$ factors through an isomorphism $GW(K) \to \bpi_0(\1_k)(K)_0$.  We write $\rho$ for the symbol $[-1]$, $h$ for $1 + \langle -1 \rangle$ and $\varepsilon$ for $- \langle -1 \rangle$.
\end{entry}

\begin{entry}[Homotopy modules ctd.]
	\label{par:homotopymodulesbis}
	Every homotopy module has the structure of a module over $\bpi_0(\1_k)$.  The category of homotopy modules also admits several equivalent descriptions that are quite useful.  Using motivic infinite loop space recognition machinery \cite[Theorem 3.5.14]{EHKSYdeloop}, the degree $0$ part of a homotopy module can be described in terms of strictly $\aone$-invariant sheaves admitting ``framed transfers" \cite[Theorem 5.14]{BachmannYakerson}.  Alternatively, there is a ``field-theoretic" description a la Voevodsky's theory of sheaves with transfers, the so-called category of sheaves with Milnor--Witt transfers \cite[Chapter 2]{MWMotives}.  Finally there is an approach by Morel using so-called ``generalized transfers" \cite{FeldMWCycle}.  In all cases, the sheaves $\mathbf{M}_*$ are determined by their sections over finitely generated extensions of the base field together with suitable residue and transfer morphisms that are compatible with action of $\bpi_0(\1_k)$.  Moreover, all the descriptions above are equivalent \cite[Theorem 1.3]{AnanyevskiyNeshitov} or \cite[Corollary 6.2.4]{FeldHomotopyModules}.
\end{entry}

\begin{entry}
	\label{par:modules}
	If $E \in \mathrm{CAlg}(\SH(k))$, then $\Mod_E(k)$ admits a homotopy $t$-structure as well.  The heart of this $t$-structure has been described for various $E$.  Of particular interest to us is the case $E = H\Z[1/p]$, i.e., Voevodsky's motivic Eilenberg--Mac Lane spectrum, in which case $\Mod_{H\Z[1/p]}(k)^{\heartsuit}$ coincides with Voevodsky's category of homotopy modules of $\Z[1/p]$-modules with transfers.  Indeed, Voevodsky showed that the heart of a corresponding $t$-structure on $\DM(k)$ could be described in terms of strictly $\aone$-invariant Nisnevich sheaves with transfers \cite[\S 3.1 p. 205]{VoeTriCat} and then one may transfer this equivalence to module spectra in conjunction with the discussion of Paragraph~\ref{par:modulespectra}.
\end{entry}

\begin{entry}
	\label{par:effectivetstructures}
	The category $\SH(k)^{\eff}$ also has a $t$-structure.  Indeed, we define the non-negative part to be $\SH(k)^{\eff} \cap \SH(k)_{\geq 0}$. By \cite[Proposition 4]{Bachmanngenslices} it follows that this yields a $t$-structure on $\SH(k)^{\eff}$ that is called the effective homotopy $t$-structure; $E \in \SH(k)^{\eff}_{\le 0}$ if and only if $\bpi_i(E)_0 = 0$ for $i > 0$.  If follows from the discussion after \cite[Proposition 4]{Bachmanngenslices} that the non-negative part of the homotopy $t$-structure on $\SH(k)^{\eff}$ is generated under colimits and extensions by $\Sigma^{\infty}_{\pone} X_+$ and thus coincides with $\SH(k)^{\veff}$ (see Paragraph~\ref{par:veryeffective}); but we recapitulate the argument in the next result.
\end{entry}

\begin{proposition}
	\label{prop:generationundercolimits}
	The categories $\SH^{S^1}(k)_{\geq 0}$, $\SH(k)^{\veff}$ and $\Mod_{E}(k)^{\veff}$ for any $E \in \mathrm{CAlg}(\SH(k)^{\veff})$ are generated under sifted colimits by the image of $\Sm_k$.
\end{proposition}

\begin{proof}
	We first show that the categories are generated by the objects claimed under colimits.  The categories $\SH^{S^1}(k)_{\geq 0}$ and $\SH(k)^{\veff}$ are the non-negative parts of presentably stable $\infty$-categories with $t$-structure by the discussion of Paragraph~\ref{par:homotopytstructureS1} and \ref{par:homotopytstructureP1}. In these stable $\infty$-categories, the $\infty$-connective objects are zero (this follows from hypercompleteness of the Nisnevich $\infty$-topos).  Moreover, for every non-negative object, we can always find a family of maps from shifted suspension spectra of smooth schemes that induce an epimorphism after taking $\pi_0$ (see ``Details on killing cells" \cite[p. 1130]{Bachmanngenslices}); the result follows from the discussion there, though see \cite[Corollary C.2.1.7]{SAG} for a more abstract statement.  The case of $\Mod_E(k)^{\veff}$ is similar and reduces to the case of $\SH(k)^{\veff}$ since the forgetful functor $\Mod_{E}(k)^{\heartsuit} \to \SH(k)^{\heartsuit}$ is exact and conservative \cite[Lemma 29 and its proof]{bachmann-tambara}.

	Now we reduce to sifted colimits.\NB{This is a general fact that the closure under sifted colimits of a subcategory closed under sums is closed under all colimits. Should be able to find a reference...} We treat $\SH^{S^1}(k)_{\ge 0}$ as an example. Let $\mathcal C \subset \SH^{S^1}(k)_{\ge 0}$ denote the subcategory generated under sifted colimits by $\Sm_k$. We claim $\mathcal C$ is closed under binary sums. To see this, first let $X \in \Sm_k$ and write $\mathcal C_X \subset \mathcal C$ for the subcategory of those objects $Y$ such that $\Sigma^\infty_+ X \oplus Y \in \mathcal C$. Then $\mathcal C_X$ contains $\Sigma^\infty_+ \Sm_k$ and is closed under sifted colimits (use that sifted categories are weakly contractible \cite[Proposition 5.5.8.7]{HTT}), hence $\mathcal C_X = \mathcal C$. New let $X \in \mathcal C$ be arbitrary. Then $\mathcal C_X$ contains $\Sigma^\infty_+ \Sm_k$ by what we just did, and is closed under sifted colimits as before; hence $\mathcal C_X = \mathcal C$. We have thus proved that $\mathcal C$ is closed under finite sums and sifted colimits, whence all colimits \cite[Lemma 5.5.8.13]{HTT}.
\end{proof}

\subsubsection*{Decomposing $\SH(k)_{\Q}$}
\begin{entry}
	\label{par:varepsilonconventions}
There exists an endomorphism $\varepsilon$ of the motivic sphere spectrum $\1_k \in \SH(k)$ such that $\varepsilon^2 = id$: $\varepsilon$ corresponds to permutation of the two factors in $\gm{} \wedge \gm{}$.  Morel computed that $\varepsilon$ corresponds to the element $- \langle -1 \rangle \in GW(k)$ under the identification of $\bpi_0(\1)$ with $\K^{MW}_0(k)$ from Paragraph \ref{par:milnorwittktheory}.  With that notation, the class of the hyperbolic form $h$ is $1 - \varepsilon$, and the relation $\eta h = 0$ in Milnor--Witt K-theory implies that $\eta$ is killed by $1 - \varepsilon$.  After inverting $2$, we obtain a pair of orthogonal idempotents $e_{\pm} := \frac{1}{2}(1 \mp \varepsilon)$ that decompose $\SH(k)_{\Z[1/2]}$ as a sum of eigen-subcategories.  The image of $e_+$, i.e., the {\em $-1$-eigenspace} of $\varepsilon$ is defined to be $\SH(k)^+_{\Z[1/2]}$ , while the image of $e_-$ is defined to be $\SH(k)_{\Z[1/2]}^{-}$, and there is a decomposition: 
\[
\SH(k)_{\Z[1/2]} \weq \SH(k)_{\Z[1/2]}^+ \times \SH(k)_{\Z[1/2]}^{-};
\]
(this unfortunate terminological issue was caused by a change in conventions as regards $\varepsilon$.)  If $k$ has finite $2$-cohomological dimension, then the $-$-part vanishes \cite[Lemma 6.8]{LevineConvergence}.  After rationalizing, the $+$-part admits a description originally due to F. Morel.
\end{entry}

\begin{theorem}[Morel]
	\label{thm:rationalizedpluspart}
	If $k$ is a field, then $\SH(k)_{\Q}^{+} \weq \Mod_{H\Q}(k)$.
\end{theorem}

\begin{proof}
	This result is obtained by composing a number of equivalences.  First, the usual Dold--Kan adjunction defines an equivalence between $\SH^{S^1}(k)_{\Q}$ and $\D^{\aone}(k)_{\Q}$ where the latter is Morel's $\aone$-derived category \cite[\S 5.3.35]{CisinskiDeglise}.    After inverting the Tate object, there is an induced equivalence $\SH(k)_\Q \weq \D^{\aone}(k)_{\Q}[\gm{{-1}}]$.  By \cite[Theorem 16.2.13]{CisinskiDeglise} we know that $\D^{\aone}(k)_{\Q}[\gm{{-1}}]^{+}$ is the category of ``Beilinson" motives \cite[\S 14]{CisinskiDeglise}.  Finally, under the assumptions on $k$, we conclude that the category of ``Beilinson" motives coincides with Voevodsky's rationalized derived category $\DM(k,\Q)$ \cite[Theorem 16.1.14]{CisinskiDeglise}.  The equivalence in the statement then follows from the discussion of Paragraph~\ref{par:modulespectra}.
\end{proof}

\subsubsection*{Conservativity of $\gm{}$-stabilization}
\begin{entry}
	\label{par:gmstabilization}
Consider the $\gm{}$-stabilization functor 
\[
\sigma^{\infty}: \SH^{S^1}(k) \longrightarrow \SH(k)
\] 
(see Paragraph~\ref{par:s1vsp1}).  By construction, this functor sends the $S^1$-suspension spectrum of a smooth scheme $X$ to its corresponding $\pone$-suspension spectrum; it therefore sends $d$-effective $S^1$-spectra to $d$-effective $\pone$-spectra and also sends $\SH^{S^1}(k)_{\geq 0}$ to $\SH(k)^{\veff}$ (see \cite[\S 6.1]{BachmannYakerson}).  
\end{entry}

\begin{entry}
	\label{par:strictlyaoneinvarianttohomotopymodules}
There is an induced functor
\[
\sigma^{\infty}: \SH^{S^1}(k)^{\heartsuit} \longrightarrow \SH(k)^{\heartsuit}.
\]  
If a strictly $\aone$-invariant sheaf $\mathbf{A}$ supports the structure of a homotopy module (see Paragraphs~\ref{par:homotopytstructureP1} and \ref{par:homotopymodulesbis}), then we must be able to supply $\gm{}$-deloopings and a compatible structure of a module over $\bpi_0(\1_k)$.  More precisely, we must be able to find a sequence of strictly $\aone$-invariant sheaves $\mathbf{A}_i$, $i \geq 0$, such that $\mathbf{A}_0 = \mathbf{A}$ and $(\mathbf{A}_{i+1})_{-1} = \mathbf{A}_i$; in particular, the sheaf $\mathbf{A}$ must itself be a contraction.  Additionally, we need to provide the result with a $\K^{MW}_{-\ast}$-module structure compatible with these $\gm{}$-deloopings.    
\end{entry}

\begin{ex}
	\label{ex:nonhomotopymodule}
	The functor $\sigma^{\infty}$ is not an equivalence in general.   Indeed, if $C$ is a smooth projective curve of genus $g > 0$ over a field $k$, then the zeroth $\aone$-homology sheaf $\H_0^{\aone}(C)$ is known to be the free sheaf of abelian groups on $C$ (see \cite[Example 2.6]{AsokHaesemeyer}).  A modification of the argument in \cite[\S 2]{Levineslices} shows that $\H_0^{\aone}(C)$ does not underlie a homotopy module.  
\end{ex}

Morel asked whether any strictly $\aone$-invariant sheaf that is a $1$-fold contraction should underlie a homotopy module (see \cite[Remark 5.10]{MFM}).  Variants of this question were pursued by the second author, partly in collaboration with M. Yakerson.  The following result summarizes the results we will use: loosely speaking, $S^1$-spectra that are ``sufficiently effective" can be infinitely $\gm{}$-delooped.

\begin{theorem}
	\label{thm:conservativityofgmstabilization}
Assume $k$ is a perfect field of exponential characteristic $e$.  There exist minimal integers $\epsilon(k) \ge \epsilon'(k) \geq 0$ such that the following statements hold.
\begin{enumerate}[noitemsep,topsep=1pt]
	\item For $E \in \SH^{S^1}(k)_{\geq 0} \cap \SH^{S^1}(k)(\epsilon(k))$, the canonical map
\[
\bpi_0(E) \longrightarrow \bpi_0(\omega^{\infty}\sigma^{\infty}E)
\]
is an isomorphism; the same holds for $E \in \SH^{S^1}(k)_{\geq 0} \cap \SH^{S^1}(k)(\epsilon'(k))[1/e]$.
	\item The functor
	\[
	\omega^{\infty}: \SH(k)^{\eff}(\epsilon(k))^{\heartsuit} \longrightarrow \SH^{S^1}(k)(\epsilon(k))
	\]
	is an equivalence of categories; the functor $\omega^{\infty}$ induces an equivalence $\SH(k)^{\eff}(\epsilon'(k))^{\heartsuit}[1/e] \weq \SH^{S^1}(k)(\epsilon'(k))[1/e]$.
\end{enumerate}
In fact $\epsilon(k) \leq 2$ and, assuming an unpublished result of Morel, $\epsilon'(k) = 1$.
\end{theorem}	

\begin{proof}
We first prove the claims about $\epsilon(k)$.
For any strictly $\aone$-invariant sheaf, F. Morel constructed ``canonical" or ``Bass--Tate" transfers on $\mathbf{M}_{-1}$ (see \cite[p. 99]{MField} or \cite[\S 2.1.8]{BachmannConservativity}) and, consequently, on $\mathbf{M}_{-i}$ for any $i$.  If the ``canonical transfers" on $\mathbf{M}_{-i}$ give rise to framed transfers, then (1) is the conclusion of \cite[Corollary 4.9]{BachmannConservativity}.  If $k$ has characteristic $0$, the condition that Morel's canonical transfers on $2$-fold contractions give rise to framed transfers was established by \cite[Theorem 5.19]{BachmannYakerson}; if $k$ has positive characteristic, then the same statement follows from \cite[Theorem 6.1.5(iii,iv)]{FeldGenFramedTransfers}.  The second statement is simply \cite[Theorem 4.6]{BachmannConservativity} in light of the observations just made.

To prove the results about $\epsilon'(k)$, the only thing we need to prove is that if $\mathbf{M}$ is strictly $\A^1$-invariant and $e$-periodic, then the canonical transfers on $\mathbf{M}_{-1}$ extend to framed transfers.
By \cite[Theorem 6.1.5(i,iii)]{FeldGenFramedTransfers}, this follows from Corollary \ref{cor:morel-improvement} below.
\end{proof}

We are referring to the following unpublished result of Morel.
\begin{theorem}[Morel] \label{thm:morel-new}
Let $k$ be a field and $l/k$ a finite separable extension of $k$.
Then the canonical map \[ [\Sigma^\infty_{S^1} \mathbb{P}^1, \Sigma^\infty_{S^1} \mathbb{P}^1 \wedge \Spec(l)]_{\SH^{S^1}(k)} \longrightarrow [\Sigma^\infty \mathbb{P}^1, \Sigma^\infty \mathbb{P}^1 \wedge \Spec(l)]_{\SH(k)} \] is a bijection.
\end{theorem}

\begin{cor} \label{cor:morel-improvement}
Suppose $k$ is a perfect field of exponential characteristic $e$.  If $\mathbf{M}$ is a strictly $\A^1$-invariant sheaf over $k$ that is $e$-periodic, then the canonical transfers on ${\mathbf M}_{-1}$ are well-defined.
\end{cor}
\begin{proof}
Let us recall the issue.  Let $L/K$ be a finite extension of finitely generated fields over $k$.  Suppose $x \in L$ is a generator of $L/K$.  The composite 
\[ 
c_x: \Spec(\mathbb{P}^1) \wedge \Spec(L)_+ \longrightarrow \Spec(\PP^1_L)/\Spec(\PP^1_L) \setminus x \weq Th(\omega_{L/K}) \in \mathrm{Pro}(\SH^{S^1}(k)) 
\] 
induces a map 
\[ 
c_x^*: \mathbf{M}_{-1}(K) \longrightarrow \mathbf{M}_{-1}(L, \omega_{L/K}). 
\]
It follows from \cite[Lemma 4.34]{MField} that $c_x^*$ is $GW(K)$-linear and hence may be twisted by line bundles on $\Spec(K)$.  Abusing notation, we denote the twisted versions again by $c_x^*$.  Now let $L/K$ again be finite, and suppose that $(x_1, \dots, x_n) \in L$ is a collection of generators.  We obtain a sequence of intermediate field extensions $K_r := K(x_1, \dots, x_r)$ and transfer maps $c_{x_{r+1}}^*: \mathbf{M}_{-1}(K_{r+1}, \omega_{K_{r+1}/K_r}) \to \mathbf{M}_{-1}(K_r)$.
Composing them all, after appropriate twists, yields a map 
\[ 
c_{x_1, \dots, x_r}^*: \mathbf{M}_{-1}(L, \omega_{L/K}) \longrightarrow \mathbf{M}_{-1}(K). 
\]
We must thus prove that this map is independent of the choice of the sequence of generators $(x_1, \dots, x_r)$.

Let us begin with the simplest case, in which $L/K$ is finite separable, which is, of course, the only case when $k$ has characteristic $0$.  Since the only reason to assume $k$ perfect was to deal with inseparable extensions, we may then assume $k = K$.  We are thus given a finite separable extension $l/k$ and must show that the corresponding transfer on $\mathbf{M}_{-1}$ is independent of choices.  Note that for a given sequence of generators $(x_1, \dots, x_r)$, the co-transfer 
\[ 
c_{x_1, \dots, x_r}: \Sigma^\infty_{S^1} \PP^1 \longrightarrow \Sigma^\infty_{S^1} Th(\omega_{l/k}) \quad (\weq \PP^1 \wedge \Spec(l)_+) 
\] 
is an actual map in $\SH^{S^1}(k)$, and not of pro-objects (the last equivalence being non-canonical).  Moreover, by Morel's Theorem \ref{thm:morel-new}, the set of homotopy classes of maps with the same source and target is canonically isomorphic to $[\PP^1, Th(\omega_{l/k})]_{\SH(k)}$, which is a torsor under $GW(l)$.
In other words, the transfers corresponding to two different choices of generators coincide for all $\mathbf{M}_{-1}$ if and only if they coincide for $\mathbf{M}=\mathbf{K}_1^{MW}$ (so that $\mathbf{M}_{-1} = \mathbf{GW}$).
But this case is well-known (since $\mathbf{GW}$ is a double contraction).

It remains to treat the case where $L/K$ is an inseparable extension.  We employ the following two reductions.
\begin{enumerate}
\item If $L/K$ is a finite extension, $x \in L$, $K'/K$ a purely inseparable extension, and $L'/K'$ the composite extension generated by $x$, then there exists $n_{L/K} \in GW(k)$ such that the following diagram commutes
\begin{equation*}
\begin{CD}
M_{-1}(L, \omega_{L/K}) @>{c_x^*}>> M_{-1}(K) \\
@VVV                                   @VVV \\
M_{-1}(L', \omega_{L'/K'}) @>{c_{x'}^*}>> M_{-1}(K').
\end{CD}
\end{equation*}
  Here $x'$ denotes the image of $x$ in $L'$.
\item If $L/K$ is a finite, purely inseparable extension, then $\mathbf{M}_{-1}(K) \to \mathbf{M}_{-1}(L)$ is injective.
\end{enumerate}
Now let $L/K$ be a finite but possibly inseparable extension.  There exists a purely inseparable extension $K'/K$ such that $L'/K'$ is separable.
Using (1) and (2) repeatedly, we see that well-definedness of the transfer for $L/K$ reduces to the corresponding situation for $L'/K'$, which we have already treated.

It remains to establish (1) and (2).  The content of (1) can be found in \cite[Lemma 4.1.7]{FeldGenFramedTransfers}.
For (2), we argue as follows.  We may as well assume that $L/K$ is generated by a single element $x$ with $x^p \in K$.  Observe as well that the square
\begin{equation*}
\begin{CD}
\Spec(L) @>x>> \A^1_K \\
@VVV          @V{f}VV \\
\Spec(K) @>0>> \A^1_K
\end{CD}
\end{equation*}
is cartesian, where $f(T) = T^p - x^p$.  Hence $f: (\A^1_K, \Spec(L)) \to (\A^1_K, 0_K)$ is a morphism of smooth closed pairs.  Since the purity equivalence is functorial in morphisms of smooth closed pairs \cite[Theorem 3.23]{HoyoisEquiv}, the induced map 
\[ 
\PP^1 \wedge \Spec(L)_+ \weq \A^1_K/\A^1_K \setminus x \stackrel{f}{\longrightarrow} \A^1_K/\A^1_K \setminus 0 \weq \PP^1 \wedge \Spec(K)_+ 
\] 
is equivalent to the $\PP^1$-suspension of $\Spec(L)_+ \to \Spec(K)_+$.  This last observation is useful to us because the map $f:\A^1_K/\A^1_K \setminus x \to \A^1_K/\A^1 \setminus 0$ lifts to $\Spc(K)$ (along the pro-(left adjoint) to $\Spc(k) \to \Spc(K)$)\footnote{Note, however, that the equivalence $\A^1_K/\A^1_K \setminus x \weq \PP^1 \wedge \Spec(L)_+$ does not lift.}
We shall prove that the composition with $c_x: \PP^1 \to \PP^1/\PP^1 \setminus x \weq \A^1/\A^1 \setminus x$ yields a map $f \circ c_x \in [\PP^1, \PP^1]_{\SH^{S^1}(K)}$ which becomes an isomorphism after inverting $e$.
As established by Morel, we have $[\PP^1, \PP^1]_{\SH^{S^1}(K)} \weq [\PP^1, \PP^1]_{\SH(K)} \weq GW(K)$.
In particular to establish that $f \circ c_x[1/e]$ is an equivalence, we may work $\PP^1$-stably.  It is clear that $f \circ c_x$ corresponds to an element of $GW(K)$ of rank $p$.  Any such element defines a unit in $GW(K)[1/e]$, $K$ having positive characteristic, and thus determines an equivalence as required.
\end{proof}

\begin{rem}
	\label{rem:morelimprovement}
Due to the currently unpublished nature of Theorem \ref{thm:morel-new}, in the body of the text we largely do not utilize it.
In other words, we write as if $\epsilon(k) = 2 = \epsilon'(k)$ is the best known bound.
\end{rem}
	
\subsection{Real-\'etale motivic homotopy theory}
\label{ss:ret}
In this section, we establish some unstable analogs of results from \cite{Bachmannret}; much stronger results will be established in forthcoming work.  We begin by recalling some facts about the real-\'etale topology and associated categories of motivic spaces (see \ref{par:realetaletopology}-\ref{prop:unstableretloc-basics}).  If $\rho: S^0 \to \Gm$ is the map sending the non-basepoint to $-1$, we recall some links between real-\'etale localization and $\rho$-localization (see \ref{par:rho}-\ref{prop:unstable-rholoc}).  We close with some complements about the other part of the rationalized stable homotopy category and real realizations (see \ref{thm:rationalminuspartviaret}-\ref{rem:MVrealrealization}).

\subsubsection*{Unstable real-\'etale localization}
\begin{entry}
	\label{par:realetaletopology}
	We write $\ret$ for the real-\'etale topology of Scheiderer \cite[Definition 1.2.1]{Scheiderer}.  This topology is finer than the Nisnevich topology and stalks \cite[\S 3.7]{Scheiderer} are the {\em real Henselian local schemes}, i.e., Henselian local schemes with real closed residue fields.  Note that the Artin--Schreier theorem \cite[Theorem 1.6.1]{KnebuschScheiderer} implies that if $k$ is a real closed field, then $vcd_2(k) = 0$. 
\end{entry}

\begin{entry}
	\label{par:retmotivicspaces}
	The $\infty$-topos $\Shv_{\ret}(\Sm_k)$ is hypercomplete \cite[Theorem B.13]{ElmantoShah}.  We write $\L_\ret^\naive: \Shv_\Nis(\Sm_k) \to \Shv_\Nis(\Sm_k)$ for the real étale sheafification functor (recall the conventions from  Paragraph~\ref{par:motivicspacesvariants}).  We write $\Spc_{\ret}(k)$ for the $\infty$-category of real étale motivic spaces.  The $\infty$-category $\Spc_{\ret}(k)$ is the intersection $\Spc(k) \cap \Shv_{\ret}(\Sm_k)$ in $\mathrm{P}(\Sm_k)$.  
\end{entry}

\begin{proposition} 
	\label{prop:unstableretloc-basics}
Let $k$ be a perfect field and $\mathscr X \in \Shv_\Nis(\Sm_k)$ be connected.  
\begin{enumerate}[noitemsep,topsep=1pt]
\item If for $i \le n$, $\bpi_i \mathscr X$ is a strictly $\A^1$-invariant (in particular abelian) real étale sheaf, then for $i \le n$, $\bpi_i \L_\ret \mathscr X \weq \bpi_i \mathscr X$.
\item If $\mathscr X$ is connected and $\bpi_1 \mathscr X$ is abelian, then $\mathscr X$ is a real étale sheaf if and only if all its (Nisnevich) homotopy sheaves are real étale sheaves.
\item If $\mathscr X$ is $n$-connective, then so is $\L_\ret \mathscr X$.
\end{enumerate}
\end{proposition}

\begin{proof}
We begin with some preliminary observations.
(a) The functor $\Lret$ can be written as a transfinite composite of the functors $\Lretna$ and $\Lmot$.
(This is the case because $\Shv_{\ret}(\Sm_k) \to \Shv_\Nis(\Sm_k)$ and $\Spc(k) \to \Shv_\Nis(\Sm_k)$ preserve $\kappa$-filtered colimits for some $\kappa$, the localizations being accessible.
In fact $\kappa = \omega$ works, but we do not need this stronger assertion here.)
(b) For a real étale sheaf $A$ of abelian groups, we have $K_\ret(A, n) \weq K_\Nis(A, n)$ \cite[Proposition 19.2.1]{Scheiderer}.
(c) For $\mathscr X \in \Shv_\Nis(\Sm_k)$ $n$-connective ($n \ge 1$, with $\bpi_1$ abelian if $n=1$) also $\Lretna \mathscr X \in \Shv_\Nis(\Sm_k)$ is $n$-connective.
Indeed we can check this on stalks, and hence using (b) this follows from the fact that any affine $U \in \Sm_k$ has finite real étale homotopy dimension \cite[Proof of Theorem B.13]{ElmantoShah}.

(2) If all the homotopy sheaves are real étale sheaves, then it follows from (b) by working up the Postnikov tower that $\mathscr X$ is a real étale sheaf.
Conversely, if $\mathscr X$ is a real étale sheaf, then it follows from (b) that the real étale postnikov tower of $\mathscr X$ coincides with the Nisnevich postnikov tower, which implies the claim.

(1) Using (a) it suffices to prove the following: if $\mathscr Y \in \Shv_\Nis(\Sm_k)$ is connected and $\bpi_i \mathscr Y$ is both strongly $\A^1$-invariant and a real étale sheaf for $i \le n$, then the maps (1a) $\mathscr Y \to \L_\ret^\naive \mathscr Y$ and (1b) $\mathscr Y \to \Lmot \mathscr Y$ both induce isomorphisms on $\bpi_i$ for $i \le n$.
Consider the fiber sequence $\mathscr Y_{>n} \to \mathscr Y \to \mathscr Y_{\le n}$. The functor $\L_{\ret}^{\naive}$ preserves this fiber sequence.
Since $\Lretna \mathscr Y_{\le n} \weq \mathscr Y_{\le n}$ by (2) and $\Lretna \mathscr Y_{>n}$ is $n$-connected by (c), we conclude that (1a) holds.
The argument for (1b) is the same, using that the fiber sequence is preserved by $\Lmot$ by the unstable connectivity theorem \ref{thm:unstableconnectivity}.
\end{proof}

\begin{entry}
	\label{par:rho}
	If $k$ is a field, we write $\rho: S^0 \to \Gm$ for the map sending the non base-point to $-1$.  Abusing notation, if $\mathscr{X}$ is a pointed motivic space, we also write $\rho$ for the induced map $\mathscr{X} \to \mathscr{X} \wedge \Gm$.  Stabilizing, the map $\rho$ defines an element of $\K^{MW}_1(k)$ which was mentioned in Paragraph~\ref{par:milnorwittktheory}.  Note that in \cite{Bachmannret} the stable map corresponding to $\rho$ is denoted $\rho'$ and it is frequently useful stably to consider $-\rho$.
\end{entry}

\begin{lem}
	\label{lem:rhoperiodization}
	If $k$ is a perfect field and $\mathscr{X} \in \Spc(k)_*$ is simply connected, then consider the $\rho$-periodization 
	\[ 
	\mathscr X[\rho^{-1}] := \colim(\mathscr X \xrightarrow{\rho} \mathscr X \wedge \Gm \xrightarrow{\rho} \mathscr X \wedge \Gm^{\wedge 2} \xrightarrow{\rho} \dots ). 
	\]
	The following statements hold. 
	\begin{enumerate}[noitemsep,topsep=1pt]
	\item The space $\mathscr{X}[\rho^{-1}]$ is again simply connected,  
	\item the maps
	\[
	\Lret \mathscr{X} \longrightarrow (\Lret \mathscr{X})[\rho^{-1}] \to \Lret(\mathscr X[\rho^{-1}])
	\]
	are equivalences, and 
	\item the map 
	\[
	\rho: \mathscr{X}[\rho^{-1}] \longrightarrow \mathscr{X}[\rho^{-1}] \sma \Gm
	\]  
	is an equivalence.
	\end{enumerate}
\end{lem}

\begin{proof}
(1) is clear.

(2, beginning)
We now prove that $f: \Lret \mathscr{X} \to \Lret(\mathscr X[\rho^{-1}])$ is an equivalence.
Proposition~\ref{prop:unstableretloc-basics}(3) implies that $\Lret \mathscr{X}$ and $\Lret(\mathscr X[\rho^{-1}])$ are again simply connected.
Thus, to establish the claim, it suffices to check after mapping into a simply connected $\ret$-local space.  Proposition~\ref{prop:unstableretloc-basics}(2) implies that any simply connected $\ret$-local space admits a Postnikov tower with layers of the form $K(\mathbf{A},i)$ with $\mathbf{A}$ a real \'etale sheaf. Then the Eilenberg--Mac Lane spectrum $H\mathbf A \in \SH^{S^1}(k)$ is real étale local (e.g. by  Proposition~\ref{prop:unstableretloc-basics}(2)). Thus it suffices to show that $f$ becomes an equivalence in $\Lret \SH^{S^1}(k)$. For this it is enough to show that $\rho$ becomes an equivalence in the latter category, which is the content of \cite[Theorem 4.2]{bachmann-SHet2}.

(2, conclusion) We must show that $\Lret \mathscr{X} \longrightarrow (\Lret \mathscr{X})[\rho^{-1}]$ is an equivalence.
Since this is a map between simply connected spaces, by the Whitehead theorem \cite[Theorem 5.2]{ABHWhitehead}, for this we may check $S^1$-stably.
Note that $\Sigma^\infty_{S^1} (\Lret \mathscr{X})[\rho^{-1}]$ is arbitrarily highly effective, and hence so is its retract $\Sigma^\infty_{S^1} \Lret \mathscr{X}$.
Thus in conjunction with Theorem~\ref{thm:conservativityofgmstabilization} it suffices to check that our map is an equivalence $\mathbb P^1$-stably.
But now $\Sigma^\infty \Lret \mathscr{X}$ is a retract of the $\rho$-periodic object $\Sigma^\infty (\Lret \mathscr{X})[\rho^{-1}]$, whence $\rho$-periodic, concluding the proof.

(3) The argument is similar but easier.
The map is between simply connected, arbitrarily effective objects, so as before we may check it is an equivalence $\mathbb P^1$-stably.
This is immediate.
\end{proof}

\begin{proposition} 
	\label{prop:unstable-rholoc}
Let $k$ be a perfect field. For $\mathscr X \in \Spc(k)_*$ simply connected, the following are equivalent:
\begin{enumerate}[noitemsep,topsep=1pt]
	\item the space $\mathscr X$ satisfies real étale descent;
	\item the map $\rho: \mathscr X \to \mathscr X \wedge \Gm$ is an equivalence;
\end{enumerate}
Moreover, if either of the above equivalent statements is true, then $\rho^*: \Omega_\Gm \mathscr X \to \mathscr X$ is an equivalence.
\end{proposition}

\begin{proof} 
That (1) implies (2) follows from Lemma \ref{lem:rhoperiodization}.
We prove the converse, so suppose that $\mathscr X$ satisfies (2).  Suppose that $\bpi_i \mathscr X$ is a real étale sheaf, for $i \le n$, e.g., $n=1$ since $\mathscr{X}$ is simply connected.  Write $\mathscr F$ and $\mathscr C$ for the fiber and cofiber of $\mathscr X \to \mathscr X_{\le n}$.  It follows from Proposition \ref{prop:unstableretloc-basics}(2) that $\mathscr X_{\le n}$ satisfies real étale descent, and hence satisfies (2) by what we have already done.  Consequently $\rho_{\mathscr C}: \mathscr C \to \mathscr{C} \sma \Gm$ is an equivalence as well (being the cofiber of the map of equivalences $\rho_{\mathscr X} \to \rho_{\mathscr X_{\le n}}$).  Now, by the homotopy excision theorem \cite[Proposition 3.5(2)]{ABHWhitehead} the canonical map $\Sigma \mathscr{F} \to \mathscr{C}$ induces an isomorphism $\bpi_{n+2} \Sigma \mathscr F \weq \bpi_{n+2} \mathscr C$, 
while the simplicial suspension theorem \cite[Theorem 3.7]{ABHWhitehead} yields an isomorphism $\bpi_{n+1} \mathscr{F} \cong \bpi_{n+2}\Sigma \mathscr{F}$.  On the other hand $\bpi_{n+2} \mathscr C \weq \bpi_{n+2} \Sigma^\infty \mathscr C$ by \cite[Corollary 4.9]{BachmannConservativity}, which is a real étale sheaf by \cite[Theorem 25]{Bachmannret}.  Repeating the argument, it follows that all homotopy sheaves of $\mathscr X$ are real étale sheaves, and so $\mathscr X \weq \L_\ret \mathscr X$ by Proposition \ref{prop:unstableretloc-basics}.

To prove the claim about $\rho^*$, we may work our way up the Postnikov tower.
As in the proof of Lemma \ref{lem:rhoperiodization}, the layers come from $\L_\ret \SH^{S^1}(k)$, and in the latter category $\rho$ is an equivalence.
\end{proof}

\subsubsection*{The rationalized $-$ part}
Morel initially proposed a description of the rationalized $-$-part complementing the description of Theorem~\ref{thm:rationalizedpluspart} in terms of modules over the Witt sheaf.  Ananyevskiy--Levine--Panin established Morel's description \cite[Theorem 4.2]{ALP}.  We will need the following alternative description of the rationalized $-$-part due to the second author.  The key fact that will be used later is that objects in the rationalized $-$-part have real-\'etale descent.

\begin{theorem}[{\cite[Proposition 41]{Bachmannret}}]
	\label{thm:rationalminuspartviaret}
	There is an equivalence 
	\[
	\SH(k)_{\Q}^{-} \weq D((\Spec k)_{\ret},\Q),
	\] 
	in particular, the minus part of a rationalized motivic spectrum has real-\'etale descent.
\end{theorem}

\subsubsection*{Real realizations}
\begin{entry}
	\label{par:realclosures}
If $k$ is a field and $\sigma$ is an ordering of $k$, then there exists a real closed field $k_{\sigma}$ together with an order preserving inclusion $k \hookrightarrow k_{\sigma}$ of $k$; this field is unique up to order-preserving isomorphism \cite[Corollary 1.11.3]{KnebuschScheiderer}.  By \cite[Theorem 35]{Bachmannret}, there is an equivalence $\SH(k_{\sigma})[\rho^{-1}] \weq \SH$, where $\SH$ is the classical stable homotopy category.  
\end{entry}

\begin{defn} \label{def:rsigma}
If $k$ is a field and $\sigma$ is an ordering of $k$, then the composite functor 
\[ 
r_\sigma: \SH(k) \longrightarrow \SH(k_\sigma)[\rho^{-1}] \weq \SH 
\] 
will be called the {\em real realization associated with $\sigma$}.
\end{defn}

\begin{rem}
	\label{rem:MVrealrealization}
When $k = \real$, this realization functor is canonically isomorphic to the intuitive functor obtained by Kan extending the functor sending a smooth variety $X$ to $X(\real)$ equipped with its classical topology (see \cite[\S 10]{Bachmannret} for further discussion of this point).
\end{rem}

\subsubsection*{Conservativity results}
\begin{entry}
We recall here some extensions of reslts from \cite{bachmann-hurewicz}.  Indeed in \cite{bachmann-hurewicz} it is shown that if $k$ is a field of finite virtual \'etale-$2$-cohomological dimension and if $E \in \SH(k)$ is connective and slice connective, then if $E \wedge H\Z$ is trivial, then triviality of $E$ can be guaranteed by imposing conditions on the real realizations attached to all orderings of $k$.  The following is a generalization of \cite[Theorem 16, Lemma 19, Theorem 33]{bachmann-hurewicz}: roughly the vanishing condition on real realizations alluded to above can be replaced by a connectivity condition when appropriately weakening the conclusion.
\end{entry}

\begin{proposition} \label{prop:hurewicz-improved}
	Let $k$ be a perfect field of exponential characteristic $e$, and $E \in \SH(k)^{\veff}[1/e]$.  Suppose that $E[1/\rho] \in \SH(k)_{\ge p}$ and $E \wedge H\Z \in \SH(k)_{\ge p}$.  If $k$ has finite virtual $2$-étale cohomological dimension, then $E \in \SH(k)_{\ge p}$.
\end{proposition}

\begin{proof}
	Suppose $E$ is a spectrum satisfying the hypotheses.  By suspending or desuspending $E$ as necessary, we may assume that $p = 1$.  In that case, by assumption $\bpi_0(E)_*[1/\rho] = 0$ and the assumption that $E \wedge H\Z \in \SH(k)_{\geq 1}$ implies that $\bpi_0(E)_*/\eta = 0$ (indeed $H\Z \in \SH(k)_{\ge 0}$ and $\bpi_0(H\Z)_* \weq \bpi_0(\1)_*/\eta$).  We want to show that $\bpi_0(E)_* = 0$.  Since the sheaves $\bpi_0(E)_*$ are unramified, it suffices to show that their sections over finitely generated extensions of the base-field are trivial.  
	
	Thus, let $x \in \bpi_0(E)_*(K)$ for some field $K$ finitely generated over $k$.  We have $\rho^n x = 0$ for some $n$, and we can write $x = \eta^n y$.  Recalling the conventions from Paragraph \ref{par:varepsilonconventions}, the following identities hold:
	\[ 
	0 = \eta^{2n} \rho^n x = (\eta^2 \rho)^n y = (\eta(h-2))^n y = (-2\eta)^n y = \pm 2^n x. 
	\]
	In other words, we conclude that every element of $\bpi_0(E)_*$ is annihilated by a power of $2$.  
	
  In fact, we claim that there exists $N = N(K)$ such that for $d>N$, $\rho: \bpi_0(E)_d(K) \to \bpi_0(E)_{d+1}(K)$ is an isomorphism; we establish this at the end of the argument.  Assuming this, we deduce that if $x \in \bpi_0(E)_i(K)$ then $\rho^s x = 0$, where $s = \max(N-i,0)$.
	Hence by the argument from before we see that $2^s x = 0$.
	Since $s$ only depends on $i$ (and $K$), it will be enough to show that every element of $\bpi_0(E)_i$ is divisible by $2$.  If $i > N$ this is trivial.  If $i \le N$ the composite 
	\[ 
	\bpi_0(E)_{N+1} \stackrel{\rho}{\longrightarrow} \bpi_0(E)_{N+2} \xrightarrow{\eta^{2 + (N-i)}} \bpi_0(E)_i 
	\] 
	is surjective. Again as before $\eta^{2 + (N-i)}\rho$ is divisible by $2$.
	
	It remains to establish the claim.  Let $M = \pi_0^{\eff}(E) \in \SH(k)^{\eff\heartsuit}$.  Then $\bpi_0(E)_* \weq \bpi_0(M)_*$, by \cite[Proposition 5(1)]{Bachmanngenslices}.
	It will suffice to show that for $i \in \{0,1\}$ and $d > N = N(K)$ we have $\bpi_i(\cof(M \xrightarrow{\rho} M))_d(K) = 0$.
	Since $\bpi_0(\ph)_0$ is an exact conservative functor on $\SH(k)^{\eff\heartsuit}$ \cite[Proposition 5(3)]{Bachmanngenslices}, $M \in \SH(k)^{\eff\heartsuit}$ is $2$-primary torsion (i.e., $M \weq \colim_n M[2^n]$) and hence so are $\bpi_*(M)_*$.  It follows that we may apply \cite[Lemma 5.23]{BEO} to $\cof(M \xrightarrow{\rho} M)$ to obtain the required isomorphism.  
\end{proof}

\section{Weak cellularity and nullity}
\label{s:weakcellularityandnullity}
The classical Postnikov tower (say in the category of simplicial sets) can be effected by localization with respect to the maps $S^n \to \ast$ (or $\ast \to S^n$), i.e., nullification \cite[Example A.6]{Farjoun}.  More generally, one can build the Postnikov tower in any $\infty$-topos by means of nullifications \cite[Lemma 3.3.3]{ABFJ}.

In motivic homotopy theory, for $p \geq q$ there is a bi-graded family of spheres $S^{p,q} := S^{p-q} \wedge \gm{\sma q}$.  Analogous to the classical situation we will consider nullifications of these spheres.  While all maps $S^i \to S^j$ are null-homotopic if $i < j$, this is no longer the case for $\gm{\sma i} \to \gm{\sma j}$.  As such, one cannot expect that the nullifications with respect to bi-graded spheres will behave in exactly the same way as the classical Postnikov truncations.  We will use the terminology {\em weakly $S^{p,q}$-cellular} for spaces that are killed by nullification.  

Here, we analyze these nullification functors and associated ``weakly cellular" objects in both stable and unstable settings and establish various properties.  Section~\ref{ss:cellularitybasics} gives the basic definitions and properties of these nullification functors, but then focuses on the unstable setting.  Section~\ref{ss:spqcellularityviatstructures} shows that in the stable setting the nullification functors can be described in terms of truncation with respect to a $t$-structure to amplify various properties.  Section~\ref{ss:cellularexamples} contains a number of examples illustrating the geometry of ``weakly cellular" spaces.




\subsection{Weak cellularity and nullifications: basic properties}
\label{ss:cellularitybasics}
In this section, we study nullifications at a {\em pointed, compact, motivic space $A$}.  We then introduce the class $O(A)$ of weakly $A$-cellular spaces, which are spaces that admit no non-trivial maps to $A$-null spaces (see \ref{par:Anullification}-\ref{ex:classicalzerotruncation}).  We then analyze the behavior of weakly $A$-cellular spaces with respect to colimits and $A$-null spaces with respect to limits, and the interaction of the nullification functors with monoidal structures and base change (see \ref{lem:leftbousfieldcolimits}-\ref{rem:A-null-expl}).  Next, upon restricting to pointed connected spaces, we give an explcit construction of $A$-nullification and deduce a host of consequences, e.g., behavior of weak cellularity with respect to smash products and connectivity (see \ref{const:smallobject}-\ref{thm:cellularityofsmashproducts}).  

Of greatest interest to us is the analysis of weakly $S^{p,q}$-cellular and $S^{p,q}$-null spaces; we give some useful characterizations of such spaces (see \ref{cor:truncatednessviahomotopysheaves}-\ref{cor:cellularityofcompositesoffibers}).  Afterwards, we study the interaction of loop spaces and $S^{p,q}$-nullification and use this to analyze the relationship between truncations and weak cellularity, and the behavior of weak cellularity in fiber sequences.  We conclude that nilpotent motivic spaces are preserved by $S^{p,q}$-nullification (see \ref{lem:Lpq-loop-B}-\ref{cor:pqtruncationpreservesnilpotence}).  

\subsubsection*{Basic definitions and constructions}
\label{ss:gmconnectivitybasics}  
\begin{entry}[$A$-equivalences]
	\label{par:Anullification}
	Assume $A$ is a pointed, compact, motivic space.  Consider the set of morphisms $\{A \times U \to U\}$, $U \in \Sm_k$; we will call this set of morphisms in $\ho{k}$ the {\em generating} $A$-equivalences.  The left Bousfield localization at (the smallest strongly saturated class of $\ho{k}$ generated by) the set of generating $A$-equivalences exists and is again a presentable $\infty$-category \cite[Proposition 5.5.4.15]{HTT}.  We write 
\[
\mathrm{L}_{A}: \ho{k} \longrightarrow \ho{k}
\] 
for the corresponding localization endofunctor. 
\end{entry}

\begin{entry}
	\label{par:localizingundercategories}
If $\mathcal{C}$ is a presentable $\infty$-category equipped with a functor $\Sm_k \to \mathcal{C}$, we use similar notation for the localization of $\mathscr{C}$ generated by the image of the above maps.  We have in mind $\mathcal{C} = \ho{k}_*$, $\SH^{S^1}(k)$, $\SH^{S^1}(k)_{\geq 0}$, $\SH(k)$, $\SH(k)^{\veff}$. 
\todo{Discuss singular varieties.}
\end{entry}

As usual, the localization functors identify the left Bousfield localization of $\ho{k}$ with the subcategory consisting of local objects.  The $A$-local spaces may be characterized in terms of the internal mapping objects: a space $\mathscr{X}$ is $A$-local if and only if the map $\mathscr{X} \weq \iMap(\ast,\mathscr{X}) \to \iMap(A,\mathscr{X})$ is an equivalence.  Likewise, $A$-equivalences can be detected by mapping into $A$-local spaces, i.e., a map $f: \mathscr{X} \to \mathscr{Y}$ is an $A$-equivalence if and only if for every $A$-local space $\mathscr{W}$, the map $\iMap(\mathscr{Y},\mathscr{W}) \to \iMap(\mathscr{X},\mathscr{W})$ is an equivalence.  In the case of interest, we make the following definition.

\begin{defn}
	Assume $A$ is a pointed, compact motivic space.  Suppose $\mathscr{X},\mathscr{Y}$ are motivic spaces and $f: \mathscr{X} \to \mathscr{Y}$ is a map.   
	\begin{enumerate}[noitemsep,topsep=1pt]
		\item The space $\mathscr{X}$ is {\em $A$-null} if $\mathscr{X} \to \mathrm{L}_A\mathscr{X}$ is an equivalence (i.e., $\mathscr{X}$ is $A$-local).
		\item The space $\mathscr X$ is {\em weakly $A$-cellular} and we write $\mathscr{X} \in O(A)$, if $\mathrm{L}_A \mathscr{X} = \ast$; equivalently the map $\mathscr{W} \to \iMap(\mathscr{X},\mathscr{W})$ is an equivalence for every $A$-null space $\mathscr{W}$.  
		\item The map $f$ is a {\em $A$-equivalence} if $\mathrm{L}_A(f)$ is an equivalence.
	\end{enumerate}
\end{defn}

\begin{notation}
	\label{rem:terminologicalremarkGm}
	We will sometimes write $\stackrel{A}{\weq}$ for $A$-equivalences.  Of particular interest to us will be the case $A = S^{p,q}$, in which case we will write $\mathrm{L}^{p,q}$ for $\mathrm{L}_A$.  We refer the reader to Section~\ref{ss:cellularexamples} for a number of examples that help to give a sense of how we will use these notions.
\end{notation}

\begin{ex}
	\label{ex:classicalzerotruncation}
	In the case $A = S^n_k$, the generating $A$-equivalences are the maps $S^n_k \times X \to X$.  In any $\infty$-topos, the corresponding localization is $n-1$-truncation \cite[Lemma 3.3.3]{ABFJ}.  In the case where $n = 0$, any $-1$-truncated sheaf is $\aone$-invariant, it follows that $\L^{0,0} = \mathrm{L}_A = \tau_{\leq -1}^{\Nis}$.  
\end{ex}

We now establish basic properties of these localization functors; our treatment follows that of \cite{Farjoun}.  

\subsubsection*{Weak cellularity, first properties}
We record the compatibility of localization with colimits in the following result.

\begin{lem}
	\label{lem:leftbousfieldcolimits}
	The localizations $\mathrm{L}_A$, viewed as functors from motivic spaces to $A$-null motivic spaces preserve colimits.  In particular, given a cofiber sequence, 
	\[
	\mathscr{X} \stackrel{f}{\longrightarrow} \mathscr{Y} \stackrel{g}{\longrightarrow} \mathscr{Z},
	\]
	the following statements hold.
	\begin{enumerate}[noitemsep,topsep=1pt]
	\item If $\mathscr{X}$ and $\mathscr{Z}$ are in $O(A)$, then so is $\mathscr{Y}$.
	\item If $\mathscr{X} \in O(A)$, then $g$ is an $A$-equivalence, and 
	\item If $f$ is a $A$-equivalence, then $\mathscr{Z} \in O(A)$.
	\end{enumerate}
\end{lem}

\begin{proof}
	All these statements follow from the fact that $\mathrm{L}_A$ is a homotopical left adjoint.
\end{proof}

\begin{amplification}
	\label{amplification:leftbousfieldcolimitsstablecases}
	If $\mathcal{C}$ is a presentable $\infty$-category under $\Sm_k$ (see \textup{Paragraph~\ref{par:localizingundercategories}}), then the conclusions of \textup{Lemma~\ref{lem:leftbousfieldcolimits}} hold in $\mathcal{C}$.
\end{amplification}

\begin{proof}
	The left Bousfield localization $\mathrm{L}_{A}$ of $\mathcal{C}$ is still a homotopical left adjoint.
\end{proof}

\begin{lem} \label{lemm:LA-products}
The functor $\L_A$ preserves finite products.
\end{lem}

\begin{proof}
We need only check that if $f$ is a generating $A$-equivalence and $X \in \Sm_k$ then $f \times \id_X$ is an $A$-equivalence \cite[Proposition 2.2.1.9]{HTT}, which is clear.

\end{proof}

We amplify this point.
Write $\mathrm{CMon}(\ho{k})$ for the category of commutative monoid objects on $\ho{k}$.  This category admits forgetful functors to $\ho{k}_*$ and $\ho{k}$.
Applying \ref{par:localizingundercategories} to the left adjoints, we transport $\L_A$ to these categories as well.

\begin{amplification}
	\label{amplification:localizationpreservesmonoids}
	The forgetful functors 
	\[
	\mathrm{CMon}(\ho{k}) \longrightarrow \Mon(\Spc(k)) \longrightarrow \ho{k}_* \longrightarrow \ho{k}
	\]
	commute with $\mathrm{L}_A$.
	(In fact this holds for the forgetful functor from algebras in $\Spc(k)$ under any $\infty$-operad.)
\end{amplification}
\begin{proof}
	By Lemma~\ref{lemm:LA-products} we know that $\mathrm{L}_A$ is compatible with the symmetric monoidal structure on $\ho{k}$.  By definition the commutative monoid objects in $\ho{k}$ are algebras over an $\infty$-operad, and the same holds for the pointed objects. Hence it suffices to prove the parenthetical statement.

Now let $\mathcal O$ be any $\infty$-operad.
By \cite[Proposition 2.2.1.9]{HA}, $\L_A: \Spc(k) \to \L_A\Spc(k)$ upgrades to an $\mathcal O$-monoidal functor, with the inclusion a lax $\mathcal O$-monoidal right adjoint.
Then \cite[Remark 7.3.2.13]{HA} shows that we obtain an adjunction 
\[ 
\xymatrix{
L': \Alg_{\mathcal O}(\Spc(k))  \ar@<.5ex>[r] & \ar@<.5ex>[l] \Alg_{\mathcal O}(\L_A \Spc(k)): \iota 
}
\] compatible with forgetting the algebra structures.
The functor $L'$ is a Bousfield localization and has the same local objects as $\L_A$; consequently the two must coincide.
Since $L'$ commutes with the forgetful functor, thus so must $\L_A$.
\end{proof}
\begin{rem}
Here is a slightly more down to earth proof of Amplification \ref{amplification:localizationpreservesmonoids}.
Since the forgetful functors preserve $A$-null objects and sifted colimits, it suffices to show that the ``free object'' functors send $A$-equivalences to $A$-equivalences \cite[Lemma 2.10]{bachmann-norms}. Using the explicit form of the free functors (namely $\mathscr X \mapsto \mathscr X_+$ respectively $\mathscr X \mapsto \amalg_i \mathscr X^i_{h\Sigma_i}$), this is clear.
\end{rem}

Before deducing some further consequences about $A$-null spaces from the above result, we amplify some of the consequences about $A$-localizations preserving symmetric monoidal structures.

\begin{amplification}
	\label{amplification:symmetricmonoidal}
Suppose $\mathcal{C}$ is a presentably symmetric monoidal $\infty$-category under $\Sm_k$ such that 
\begin{enumerate}[noitemsep,topsep=1pt]
\item the functor $\Sm_k \to \mathcal{C}$ is symmetric monoidal, and 
\item the functor $\Sm_k \to \mathcal{C}$ has dense image, i.e., the category $\mathcal{C}$ is generated under colimits by the image of the source.
\end{enumerate}
The localization functor $\mathrm{L}_A$ on $\mathcal{C}$ is monoidal.  
\end{amplification}
\begin{proof}
The condition that the functor $\Sm_k \to \mathcal{C}$ is symmetric monoidal guarantees that if $\alpha$ is a generating equivalence, then $\alpha \tensor \id_{X}$ is as well for any $X \in \Sm_k$.
Combined with the density condition, this implies that $\L_A: \mathscr C \to \L_A \mathscr C$ is symmetric monoidal \cite[Proposition 2.2.1.9]{HTT}.
\end{proof}

\begin{cor}
	\label{cor:stablelocalizationsymmetricmonoidal}
	The conclusion of \textup{Amplification~\ref{amplification:symmetricmonoidal}} holds for $\mathcal{C} = \ho{k}_*$, $\SH^{S^1}(k)_{\geq 0}$, $\SH(k)^{\veff}$ or $\Mod_{H\Z}^{\veff}$.
\end{cor}
\begin{proof}
Immediate by combining Amplification \ref{amplification:symmetricmonoidal} with Proposition \ref{prop:generationundercolimits}.
\end{proof}

\begin{ex}
	\label{ex:basicexample}
	If $\mathscr{X} \in \ho{k}$, then $A \sma \mathscr{X}_+$ is weakly $A$-cellular.
	Indeed by Corollary \ref{cor:stablelocalizationsymmetricmonoidal}, it suffices to prove that $A$ is $A$-cellular as a pointed space; this follows from Amplification \ref{amplification:localizationpreservesmonoids}.
\end{ex}

Let us also mention the following.
\begin{lem} \label{lem:cellularity-ess-sm-bc}
The functor $\L_A$ commutes with essentially smooth change of base field.
In particular, if $f: \Spec(k') \to \Spec(k)$ is an essentially smooth field extension, then $f^*$ sends $A$-equivalences to $f^*A$-equivalences, $O(A)$ to $O(f^*A)$ and $A$-null spaces to $f^*A$-null spaces.
\end{lem}
\begin{proof}
Since $f^*$ preserves colimits and finite products, it sends $A$-equivalences to $f^*A$-equivalences.
It hence remains to prove that it sends $A$-null spaces to $f^*A$-null spaces.
Since $A$ is compact, this follows immediately from essentially smooth base change (see Paragraph \ref{par:essentiallysmoothbasechange}).
\end{proof}

\subsubsection*{Null spaces and limits}
Just as weakly cellular objects and weakly cellular maps interact well with colimits, local objects interact well with limits.  We record the necessary facts in the following result (these facts are standard and amount to notational modifications of \cite[1.A.8]{Farjoun})

\begin{lem}
	\label{lem:localizationmappingspaces}
	The following statements hold.
	\begin{enumerate}[noitemsep,topsep=1pt]
	\item If $\mathscr{X}$ is $A$-null, then for any space $\mathscr{Y}$, the space $\iMap(\mathscr{Y},\mathscr{X})$ is again $A$-null.
	\item Any limit of $A$-null spaces is again $A$-null.  
	\item If $\mathscr{F} \to \mathscr{E} \to \mathscr{B}$ is a fiber sequence with $\mathscr{F}$ and $\mathscr{B}$ $A$-null spaces, and $\mathscr{B}$ connected, then $\mathscr{E}$ is $A$-null as well. 
	\item If $\mathscr{X}$ is a pointed connected space, then $\mathscr{X}$ is $A$-null if and only if $\iMap_*(A,\mathscr{X}) = \ast$.
	\item If $\mathscr{X}$ is a pointed, simply connected space, then $\mathscr{X}$ is $\Sigma A$-null if and only if $\Omega \mathscr{X}$ is $A$-null.
	\end{enumerate}
\end{lem}
\begin{proof}
	\noindent {\em Point (1)}.  To say that $\mathscr{X}$ is $A$-null is equivalent to saying that the induced map
	\[
	\mathscr{X} = \iMap(\ast,\mathscr{X}) \to \iMap(A,\mathscr{X})
	\]
	is an equivalence.  Mapping $\mathscr{Y}$ into this equivalence and using the exponential law shows that 
	\[
	\iMap(\mathscr{Y},\mathscr{X}) \to \iMap(\mathscr{Y},\iMap(A,\mathscr{X})) \weq \iMap(A,\iMap(\mathscr{Y},\mathscr{X}))
	\]
	is an equivalence as well, which is what we wanted to show.\newline
	
	\noindent {\em Point (2)}. Let $\mathbf{I}$ be a small diagram and let $\mathscr{X}: \mathbf{I} \to \ho{k}$, be an $\mathbf{I}$-diagram of motivic spaces.  There is an equivalence 
	\[
	\mathrm{Map}(-,\lim_{\mathbf{I}} \mathscr{X}) \weq \lim_{\mathbf{I}} \mathrm{Map}(-,\mathscr{X})
	\]
	by the definition of mapping spaces.  Now, for any $i \in \mathbf{I}$, the map $\mathscr{X} \to \iMap(A,\mathscr{X})$ is an equivalence by assumption and we conclude that the induced map 
	\[
	\lim_{\mathbf{I}} \mathscr{X} \longrightarrow \lim_{\mathbf{I}} \iMap(A,\mathscr{X})
	\]
	is an equivalence as well.  \newline

	\noindent {\em Point (3)}.  Suppose $\mathscr{F} \to \mathscr{E} \to \mathscr{B}$ is a fiber sequence with $\mathscr{B}$ connected.  In that case, we have the following morphism of fiber sequences
	\[
	\xymatrix{
	\mathscr{F}\ar[d] \ar[r] & \mathscr{E} \ar[d]\ar[r] & \mathscr{B} \ar[d] \\
	\iMap(A,\mathscr{F}) \ar[r] & 	\iMap(A,\mathscr{E}) \ar[r] & 	\iMap(A,\mathscr{B})	
	}
	\]
	where the outer two maps are equivalences and $\mathscr B$ is connected. It follows that the middle map is an equivalence \cite[Lemma 6.2.3.16]{HTT}.\newline
	
	\noindent {\em Point (4)}. We know that $A \to \ast$ is an $A$-equivalence by Example~\ref{ex:basicexample}.  If $\mathscr{X}$ is pointed and connected and $A$-null, it follows immediately that $\iMap_*(A,\mathscr{X})$ is $\ast$.  Conversely, suppose $\iMap_*(A,\mathscr{X}) = \ast$ and $\mathscr{X}$ is connected.  There is a fiber square of the form
	\[
	\xymatrix{
		\iMap_*(A,\mathscr{X}) \ar[r]\ar[d] & \iMap(A,\mathscr{X}) \ar[d]^{ev} \\
		\ast \ar[r] & \mathscr{X} 
	}
	\]  
	linking pointed and free mapping spaces.  Since $\mathscr{X}$ is assumed connected, $\ast \to \mathscr{X}$ is an effective epimorphism.  Since $\ho{k} \subset \Shv_\Nis(\Sm_k)$ is closed under limits, \cite[Lemma 6.2.3.16]{HTT} implies that the evaluation map $ev$ in the above diagram is an equivalence. By $2$-out-of-$3$, it follows that also $\iMap(A, \mathscr X) \to \iMap(*, \mathscr X)$ is an equivalence, as needed.\newline
	
	\noindent {Point (5).}  This statement follows by adjunction of loops and suspension from the preceding point:  $\mathscr{X}$ is $\Sigma A$-null if and only if $\ast = \iMap_*(\Sigma A,\mathscr{X}) \weq \iMap_*(A,\Omega \mathscr{X})$.
\end{proof}

\begin{rem} \label{rem:A-null-expl}
Let $X \in \Sm_k$ and $\mathscr X \in \Spc(k)_*$.
Note that $\pi_0 \iMap_*(A, \mathscr X)(X) \weq [A \wedge X_+, \mathscr X]_*$ and $\pi_i(\iMap_*(A, \mathscr X)(X), 0) \weq [\Sigma^i A \wedge X_+, \mathscr X]_*$.
In particular, by Lemma \ref{lem:localizationmappingspaces}(4), if $\mathscr X$ is connected then $\mathscr X$ is $A$-null if and only if \[ [\Sigma^i A \wedge X_+, \mathscr X]_* = 0 \] for all $i \ge 0$ and $X \in \Sm_k$.
\end{rem}

\subsection*{Construction of the cellularization}
One may construct $\mathrm{L}_A$ explicitly by means of a small object argument; we summarize this observation in the next construction. 

\begin{construction}
	\label{const:smallobject}
	Let $\mathscr X \in \Spc(k)_*$ be connected. Define spaces $\mathscr{X}_n$ inductively as follows.  Set $\mathscr{X}_0 = \mathscr{X}$.  Assuming $\mathscr{X}_n$ is defined, we define $\mathscr{X}_{n+1}$ as the pushout of the diagram
	\[
	\ast \longleftarrow \bigvee_{\{f: \Sigma^i A \wedge Y_+ \to \mathscr{X}_n\}}  \Sigma^i A \wedge Y_+ \longrightarrow \mathscr{X}_n,
	\]
	where $Y \in \Sm_k$ and $i$ is an integer $\geq 0$.  We then set $\mathscr{X}_{\infty} = \colim_n \mathscr{X}_n$.  The space $\mathscr{X}_{\infty}$ is a model for $\mathrm{L}_A\mathscr{X}$.
\end{construction}
\begin{proof}
First note that each $\mathscr X_i$ is connected (see the proof of Lemma \ref{lem:X-oo-conn}, which contains a more general assertion), and hence so is $\mathscr X_\infty$.  Since $\Sigma^i A \wedge Y_+ \stackrel{A}{\weq} *$ we know that $\mathscr X \to \mathscr X_\infty$ is an $A$-equivalence; hence it remains to verify that $\mathscr X_\infty$ is $A$-null.  By Remark \ref{rem:A-null-expl}, for this we need to check that for every $Y \in \Sm_k$ and $i \ge 0$, every map $\Sigma^i A \wedge Y_+ \to \mathscr X_\infty$ is null homotopic.  Since the source is compact in $\Spc(k)$\NB{ref?}, this holds by construction.
\end{proof}

\begin{lem}
	\label{lem:X-oo-conn}
Assume $n$ is an integer $\geq 0$ and $k$ is a field, assumed to be perfect if $n > 0$ and $\mathscr{X}$ is an $n$-connected pointed motivic space.  If $A$ is an $n$-connective pointed, compact, motivic space, then $\L_A \mathscr{X}$ is again $n$-connected. 
\end{lem}
	
\begin{proof}
Notation as in Construction~\ref{const:smallobject}, since $n$-connective morphisms are stable under pushouts \cite[Corollary 6.5.1.17]{HTT}, one deduces from the unstable connectivity theorem~\ref{thm:unstableconnectivity} that $\mathscr{X}_{\infty}$ and thus $\L_A \mathscr{X}$ is again $n$-connected (for $n > 0$, the reference requires $k$ perfect).
\end{proof}


The next result is a consequence of Lemma~\ref{lem:X-oo-conn} in the case where $A = S^{p,q}$.

\begin{lem} \label{lem:Lpq-map-conn}
If $\mathscr X$ is a pointed connected motivic space, then $\mathscr X_{\le p-q-1} \weq (\L^{p,q} \mathscr X)_{\le p-q-1}$.  If additionally $\mathscr X \in O(S^{p,q})$, then $\mathscr X$ is $(p-q-1)$-connected.
\end{lem}

\begin{proof}
The space $S^{p,q}$ is $(p-q)-1$-connected.
It follows that $\mathscr X_n \to \mathscr X_{n+1}$ induces an equivalence on $(p-q-1)$-truncations.
The same is true for the colimit $\mathscr X \to \mathscr X_\infty \weq \L^{p,q} \mathscr X$.
\end{proof}

We can deduce the following.
\begin{theorem}\NB{can the connectivity assumptions be avoided?}
	\label{thm:cellularityofsmashproducts}
	If $\mathscr{X}$ is pointed, connected and weakly $A$-cellular, and $\mathscr{Y}$ is pointed, connected and weakly-$B$-cellular, then $\mathscr{X} \wedge \mathscr{Y}$ is weakly $A \wedge B$-cellular.
\end{theorem}

\begin{proof}
We continue to follow the notation of Construction~\ref{const:smallobject} and first show that $\mathscr X \wedge B \in O(A \wedge B)$.
To see this, note that $\mathscr X_n \wedge B \to \mathscr X_{n+1} \wedge B$ is the pushout of a map of the form $(\ph) \wedge A \wedge B \to *$, and hence an $A \wedge B$-equivalence.
Thus 
\[ 
\mathscr X \wedge B \stackrel{A \wedge B}{\weq} \mathscr X_\infty \wedge B = *, \] 
as claimed.

Next note that $\mathscr X \wedge \mathscr Y_n \to \mathscr X \wedge \mathscr Y_{n+1}$ is the pushout of a map $(\ph) \wedge \mathscr X \wedge B \to *$, and hence, repeating what we just observed, also an $A \wedge B$-equivalence.
Consequently \[ \mathscr X \wedge \mathscr Y \stackrel{A \wedge B}{\weq} \mathscr X \wedge \mathscr Y_\infty = *, \] as required.
\end{proof}

\subsubsection*{Characterizing $S^{p,q}$-null and weakly $S^{p,q}$-cellular spaces}
The characterization of $S^{p,q}$-null spaces from Lemma~\ref{lem:localizationmappingspaces}(4) admits an interpretation in terms of homotopy sheaves using the notion of contraction (see Paragraph \ref{par:contraction}).  

\begin{cor}
	\label{cor:truncatednessviahomotopysheaves}
	If $\mathscr{X}$ is a pointed connected space, then $\mathscr{X}$ is $S^{p,q}$-null if and only if $\bpi_i(\mathscr{X})_{-q} = 0$ for all $i \ge p-q$.  
\end{cor}

\begin{proof}
	Lemma~\ref{lem:localizationmappingspaces}(4) gives a necessary and sufficient condition for $S^{p,q}$-nullity of a pointed connected space in terms of contractibility of the $(p,q)$-fold loop space of $\mathscr{X}$.  We may write 
	\[
	\Omega^{p,q}\mathscr{X} \weq \Omega^{p-q}_{S^1}\Omega^{q,q}\mathscr{X},
	\]
	and then the statement follows by appeal to the discussion of Paragraph~\ref{par:contraction} (in particular, \cite[Theorem 6.13]{MField}).
\end{proof}

\begin{lem}
	\label{lem:truncationandpostnikovlayers}
	Suppose $\mathscr{W}$ is a pointed connected space.  The following statements are equivalent.
	\begin{enumerate}[noitemsep,topsep=1pt] 
		\item The space $\mathscr{W}$ is $S^{p,q}$-null.
		\item For every integer $i \geq 1$, the spaces  $K(\bpi_i(\mathscr{W}),i)$ are $S^{p,q}$-null.   
		\item For every integer $i$, the Postnikov layers $\tau_{\leq i} \mathscr{W}$ are $S^{p,q}$-null.
	\end{enumerate}
\end{lem}

\begin{proof}
	The implication (1) $\Longrightarrow$ (2) is an immediate consequence of Corollary~\ref{cor:truncatednessviahomotopysheaves}.  The implication (2) $\Longrightarrow$ (3) follows by induction using the fiber sequences $K(\bpi_i(\mathscr{X}),i) \to \tau_{\leq i}\mathscr{W} \to \tau_{\leq i-1}\mathscr{W}$
	and Lemma~\ref{lem:localizationmappingspaces}(3).  The implication (3) $\Longrightarrow$ (1) follows from Lemma~\ref{lem:localizationmappingspaces}(2) and the fact that $\mathscr{W} = \lim_i \tau_{\leq i}\mathscr{W}$.
\end{proof}

\begin{proposition}
	\label{prop:pqconnectivityandfibersequencesI}
	Assume $\mathscr{F} \to \mathscr{E} \stackrel{f}{\to} \mathscr{B}$ is a fiber sequence in $\ho{k}$ with $\mathscr{B}$ connected.  
	\begin{enumerate}[noitemsep,topsep=1pt]
		\item The morphism $f$ is a universal $A$-equivalence (that is, any base change of $f$ in $\ho{k}$ is an $A$-equivalence) if and only if $\mathscr{F} \in O(A)$.
		\item If $\mathscr{F}$ and $\mathscr{B}$ lie in $O(A)$, then $\mathscr{E} \in O(A)$ as well.
	\end{enumerate}
\end{proposition}

\begin{proof}
	The second statement is a special case of the first, so it remains to establish the first statement.  If $f$ is a universal $A$-equivalence, then $\mathscr{F} \to \ast$ is as well by universality, so $\mathscr{F} \in O(A)$ by definition.  Conversely, suppose that $\mathscr F \in O(A)$ and $\mathscr{B}' \to \mathscr{B}$ is a morphism, with fiber denoted $\mathscr S$.
Since $\mathscr B$ is connected, $* \to \mathscr B$ is an effective epimorphism and $\mathscr B \weq * \sslash \Omega \mathscr B$.
Similarly $\mathscr X \weq \mathscr F \sslash \Omega \mathscr B$ and $\mathscr B' \weq \mathscr S \sslash \Omega \mathscr B$ (see Paragraph \ref{par:kanloopgroup} for the notation).
By universality of colimits, the base change of $f$ along $\mathscr B' \to \mathscr B$ is given by \[ (\mathscr S \times \mathscr F) \sslash \Omega \mathscr B \to \mathscr S \sslash \Omega \mathscr B. \]
Since $\L_A$ preserves finite products (Lemma~\ref{lemm:LA-products}), the map displayed above is a colimit of $A$-equivalences, and hence an $A$-equivalence itself.
\end{proof}

\begin{cor}
	\label{cor:cellularityofcompositesoffibers}
	If $f: \mathscr{X} \to \mathscr{Y}$ and $g: \mathscr{Y} \to \mathscr{Z}$ are maps such that $\fib(f), \fib(g) \in O(A)$ and $\fib(g)$ is connected, then $\fib(g \circ f)$ lies in $O(A)$ as well.
\end{cor}

\begin{proof}
	By \cite[Proposition 3.1]{ABHWhitehead}, there is a fiber sequence of the form
	\[
	\fib(f) \longrightarrow \fib(g \circ f) \longrightarrow \fib(g).
	\]
	We deduce that $\fib(g \circ f) \in O(A)$ by appeal to Proposition~\ref{prop:pqconnectivityandfibersequencesI}(2).
\end{proof}

\subsubsection*{Truncation and cellularity}
\begin{proposition}
	\label{prop:pqtruncationpreservesconnectivity-v2} 
	Suppose $\mathscr{X}$ is a pointed connected motivic space.
Let $q \ge 0$ and $p \le d$ such that $\Omega^{q,q} \mathscr X$ is $d$-connective.
Then $\L^{p+q,q} \mathscr X \weq \L^{d+q,q} \mathscr X$.
\end{proposition}
\begin{proof}
Since the map $\mathscr X \to \L^{d+q,q} \mathscr X$ is an $S^{p+q,q}$-equivalence (as $d \ge p$), it suffices to prove that $\L^{d+q,q} \mathscr X$ is $S^{p+q,q}$-null.
Since $\L^{d+q,q}\mathscr X$ is connected (Lemma \ref{lem:X-oo-conn}), for this it is enough to show that $\bpi_n(\L^{d+q,q}\mathscr X)_{-q} = 0$ for all $n \ge p$ (Corollary \ref{cor:truncatednessviahomotopysheaves}).
This is true for $n \ge d$ by construction.
By Lemma \ref{lem:Lpq-map-conn} we have $\bpi_n(\L^{d+q,q}\mathscr X) \weq \bpi_n(\mathscr X)$ for $n < d$, whence in this case the vanishing holds by assumption.
This concludes the proof.
\end{proof}

\begin{cor}\label{cor:pqtruncationpreservesconnectivity}
Suppose $\mathscr X$ is a pointed $d$-connected motivic space, $d \ge 0$.
Then $\L^{p+q,q} \mathscr X$ is $d$-connected.
\end{cor}
\begin{proof}
If $p \ge d$ then $(\L^{p+q,q} \mathscr X)_{< d} \weq \mathscr X_{< d} = *$, by Lemma \ref{lem:Lpq-map-conn}.
Otherwise $\L^{p+q,q} \mathscr X \weq \L^{d+q,q} \mathscr X$ by Proposition \ref{prop:pqtruncationpreservesconnectivity-v2}, and now the same argument works.
\end{proof}

\begin{cor}
	\label{cor:separatingcellularity}
Let $\mathscr X \in O(S^{q,q})$ be pointed and connected.
The following are equivalent:
\begin{enumerate}[noitemsep,topsep=1pt]
	\item $\mathscr X \in O(S^{d,0})$,
	\item $\Omega^{q,q} \mathscr X \in O(S^{d,0})$,
	\item $\mathscr X \in O(S^{d+q,q})$.
\end{enumerate}
\end{cor}
\begin{proof}
It is clear that (1) implies (2) (see Paragraph \ref{par:contraction}), and (2) implies (3) by Proposition \ref{prop:pqtruncationpreservesconnectivity-v2}.
Finally (3) implies (1) by Lemma \ref{lem:Lpq-map-conn}.
\end{proof}

The following is occasionally useful.
\begin{cor} \label{cor:limit-of-Spq-equiv}
Let $f_i: \mathscr X_i \to \mathscr Y_i \in \Spc(k)_*$ be a cofiltered family of $S^{p,q}$-equivalences.
Let $f: \lim_i \mathscr X_i \to \lim_i \mathscr Y_i$ be the induced map.
Assume that the connectivities of $\mathscr X \to \mathscr X_i$ and $\mathscr Y \to \mathscr Y_i$ tend to infinity with $i$, and $\mathscr X, \mathscr Y$ connected.
Then $f$ is an $S^{p,q}$-equivalence.
\end{cor}
\begin{proof}
Both $\mathscr X$ and $\mathscr Y$ being connected, the same is true for their $S^{p,q}$-nullifications (Corollary \ref{cor:pqtruncationpreservesconnectivity}).
Hence by Yoneda and Lemma \ref{lem:truncationandpostnikovlayers}, it will suffice to show the following: if $\mathscr W$ is connected, $r$-truncated (for some $r$) and $S^{p,q}$-null, then $\Map(\mathscr X, \mathscr W) \weq \Map(\mathscr Y, \mathscr W)$.
By assumption for $i \gg 0$, both $\mathscr X \to \mathscr X_i$ and $\mathscr Y \to \mathscr Y_i$ are $r$-equivalences.
Hence \[ \Map(\mathscr X, \mathscr W) \weq \Map(\mathscr X_i, \mathscr W) \weq \Map(\L^{p,q} \mathscr X_i, \mathscr W) \weq \Map(\L^{p,q} \mathscr Y_i, \mathscr W) \weq \Map(\mathscr Y_i, \mathscr W) \weq \Map(\mathscr Y, \mathscr W). \]
This concludes the proof.
\end{proof}

\subsubsection*{Truncation, loop spaces and fiber sequences}
\begin{lem}\label{lem:Lpq-loop-B}
Let $k$ be a field.
\begin{enumerate}[noitemsep,topsep=1pt]
  \item Let $\mathscr G \in \Spc(k)$ be a connected $\mathscr E_1$-group. Then $B \L^{p,q} \mathscr G \weq \L^{p+1,q} B\mathscr G$.
  \item Let $\mathscr X \in \Spc(k)_*$ be simply connected. Then $\Omega \L^{p+1,q} \mathscr X \weq \L^{p,q} \Omega \mathscr X$.
\end{enumerate}
\end{lem}
\begin{proof}
Denote by $\mathcal C$ the category of simply-connected, pointed motivic spaces, and by $\mathcal C_0 \subset \mathcal C$ the full subcategory of $S^{p+1,q}$-null spaces.
By Corollary \ref{cor:pqtruncationpreservesconnectivity} and Amplification \ref{amplification:localizationpreservesmonoids}, the inclusion $\mathcal C_0 \to \mathcal C$ has a left adjoint, namely, $\L^{p+1,q}$.
Similarly denote by $\mathcal D$ the category of connected monoids in motivic spaces, and by $\mathcal D_0 \subset \mathcal D$ the full subcategory of $S^{p,q}$-null spaces.
As before, we see that $\L^{p,q}$ defines a left adjoint of $\mathcal D_0 \to \mathcal D$.
Lemma \ref{lem:loopgroup-basics} shows that 
\[ 
\xymatrix{
\mathcal D \ar@<.5ex>[r]^B & \ar@<.5ex>[l]^{\Omega}\mathcal C
}
\] 
is an adjoint equivalence, and Lemma~\ref{lem:localizationmappingspaces}(5) shows that the equivalence identifies the full subcategories $\mathcal C_0, \mathcal D_0$.
It follows that it must identify the left adjoints, as desired.
\end{proof}

Suppose $\mathscr{F} \to \mathscr{E} \to \mathscr{B}$ is a fiber sequence of pointed motivic spaces.  Lemma~\ref{lem:localizationmappingspaces}(2) implies that $\fib(\mathrm{L}^{p,q} \mathscr{E} \to \mathrm{L}^{p,q} \mathscr{B})$ is $S^{p,q}$-null and therefore there is an induced map $\mathrm{L}^{p,q}\mathscr{F} \to \fib(\mathrm{L}^{p,q} \mathscr{E} \to \mathrm{L}^{p,q} \mathscr{B})$.  The following result gives some conditions under which this map is an equivalence.
These conditions are not optimal, but will be sufficient for our future needs.

\begin{proposition}
	\label{prop:pqtruncationandfibersequences}
	Assume $k$ is a field and suppose $\mathscr{B}$ is a $\max\{p,1\}$-connected motivic space, $p \ge 0$.  If $\mathscr{F} \to \mathscr{E} \to \mathscr{B}$ is a fiber sequence, then
	\[
	\mathrm{L}^{p+q,q} \mathscr{F} \longrightarrow \mathrm{L}^{p+q,q} \mathscr{E} \longrightarrow \mathrm{L}^{p+q,q} \mathscr{B}
	\]
	is also a fiber sequence.
\end{proposition}
\begin{proof}
We shall freely use the fact that $\L^{p,q}$ preserves connectivity (Corollary \ref{cor:pqtruncationpreservesconnectivity}).  Since $\mathscr B$ is connected, $\mathscr E \weq \mathscr F \sslash \Omega \mathscr B$ (again, see Paragraph \ref{par:kanloopgroup}).  Set $\mathscr E' = \L^{p+q,q} \mathscr F \sslash \L^{p+q,q} \Omega \mathscr B$.
Via Lemma \ref{lem:loopgroup-basics} we obtain a fiber sequence
\[ 
\L^{p+q,q} \mathscr F \longrightarrow \mathscr E' \longrightarrow B \L^{p+q,q} \Omega \mathscr B. 
\]
Since $\mathscr B$ is $(p+1)$-connective, using Lemma \ref{lem:Lpq-loop-B} and Proposition \ref{prop:pqtruncationpreservesconnectivity-v2} we identify the base as \[ B \L^{p+q,q} \Omega \mathscr B \weq B \Omega \L^{p+q+1,q} \mathscr B \weq \L^{p+q+1,q} \mathscr B \weq \L^{p+q,q} \mathscr B. \]
It remains to show that $\mathscr E' \weq \L^{p+q,q} \mathscr E$.
By construction we have an $S^{p+q,q}$-equivalence $\mathscr E \to \mathscr E'$, so it suffices to show that $\mathscr E'$ is $S^{p+q,q}$-null.
This holds by Lemma \ref{lem:localizationmappingspaces}(3).
\end{proof}

\begin{rem}
The functor $\L^{p,q}$ cannot preserve fiber sequences in general.  Indeed, classical Postnikov truncations do not preserve fiber sequences in general.  For exapmle, if $\mathbf{A}$ is a strictly $\aone$-invariant sheaf of groups, then the fiber sequence 
\[ 
K(\mathbf{A}, p-1) \longrightarrow * \longrightarrow K(\mathbf{A}, p) 
\] 
is not preserved by $\L^{p,0} = \tau_{< p}$.
\end{rem}

\begin{cor}
	\label{cor:pqtruncationpreservesnilpotence}
	If $\mathscr{X}$ is a nilpotent motivic space, then $\mathrm{L}^{p+q,q}\mathscr{X}$ is again a nilpotent motivic space.  
	More precisely, suppose given a principal refinement of $\mathscr X \to \mathscr X_{<p}$, consisting of fiber sequences 
	\[ 
	\mathscr{X}_{i+1} \longrightarrow \mathscr{X}_i \longrightarrow \mathscr{B}_i \] 
	with $\mathscr X_0 = \mathscr X_{<p}$, $\mathscr B_i = K(\mathbf{A}_i, n_i)$ with $n_i \ge p+1$, and $\lim_i \mathscr X_i \weq \mathscr X$.
	Then $\L^{p+q,q} \mathscr X \weq \lim_i \L^{p+q,q} \mathscr X_i$, the induced sequences 
	\[ 
	\L^{p+q,q}\mathscr{X}_{i+1} \longrightarrow \L^{p+q,q} \mathscr{X}_i \longrightarrow \L^{p+q,q} \mathscr{B}_i 
	\] 
	are fiber sequences, and $\L^{p+q,q} \mathscr{B}_i$ is an $n_i$-connective, grouplike commutative monoid.
\end{cor}

\begin{proof}
	The fiber sequences are preserved by $\L^{p+q,q}$ by Proposition \ref{prop:pqtruncationandfibersequences}, since $\mathscr B_i$ is $p$-connected.
	Since $\L^{p+q,q}$ preserves finite products it preserves grouplike commutative monoids, whence $\L^{p+q,q} \mathscr B_i$ is one.
	In particular it is nilpotent, and $n_i$-connective by Corollary \ref{cor:pqtruncationpreservesconnectivity}.
	Also $\L^{p+q,q} \mathscr{X}_0 \weq \mathscr X_0$ (via Corollary \ref{cor:truncatednessviahomotopysheaves}: $\mathscr X_0$ is $p$-connected) is nilpotent, and hence by induction $\L^{p+q,q}\mathscr{X}_{i+1}$ is nilpotent.
	If we set $\mathscr X' = \lim_i \L^{p+q,q} \mathscr X_i$, then $\mathscr{X}'$ is nilpotent, since the connectivity of the $\L^{p+q,q} \mathscr B_i$ tends to infinity.
	It remains to prove that $\mathscr X' \weq \L^{p+q,q} \mathscr X$.  Since $\mathscr X'$ is $S^{p+q,q}$-null, for this it suffices to show that $\mathscr X \to \mathscr X'$ is an $S^{p+q,q}$-equivalence.
	This follows from Corollary \ref{cor:limit-of-Spq-equiv} (since the $\mathscr B_i$ have connectivity tending to $\infty$, and the same is true for the $\L^{p,q} \mathscr B_i$ by Corollary \ref{cor:pqtruncationpreservesconnectivity}, and hence so do $\mathscr X \to \mathscr X_i$ and $\mathscr X' \to \L^{p,q} \mathscr X_i$).
\end{proof}

\subsection{Stable weak $S^{p,q}$-cellularity via $t$-structures}
\label{ss:spqcellularityviatstructures}
We now analyze weak $S^{p,q}$-cellularity in the context of stable homotopy categories (see Paragraph~\ref{par:localizingundercategories}).  In the stable setting, $S^{p,q}$-truncation can be identified with truncation with respect to a $t$-structure (see \ref{par:conditionsonC}-\ref{lem:gmconnectivitystablehomotopysheaves}).  Note: our $t$-structures are homologically indexed, following topological conventions.  These observations allow us to deduce that various localizations preserve weak cellularity, and analyze how taking loop spaces changes weak cellular class (see \ref{cor:stablecellularityispreservedbyRlocalization}-\ref{prop:stablecellularityandlooping}); a key result here is that taking infinite loop spaces of various sorts preserves weak cellular classes (see \ref{prop:deloopingpreservesconnectivity}).  Finally the section closes with a refinement of the Whitehead theorem (see \ref{thm:pqequivsof1connectedspacesaredetectedstably}) which relies on a refinement of the stable Hurewicz theorem for motivic spectra (see \ref{prop:hurewicz-improved}).


\begin{entry}
	\label{par:conditionsonC}
Set $\mathcal{C} = \SH^{S^1}(k)$ or $\SH(k)^{\eff}$. The homotopy $t$-structures on each of these categories were described in Paragraphs~\ref{par:homotopytstructureS1} and ~\ref{par:effectivetstructures}.  We write $\mathcal{C}_{\geq 0}$ for the non-negative part of the relevant $t$-structure.  Given $X \in \Sm_k$, we abuse notation and write $X_+$ for the image of $X \in \Sm_k$ under the functor $\Sm_k \to \mathcal{C}$.  We now introduce some new $t$-structures on each of these categories which we will then link to weak $S^{p,q}$-cellularity.
\end{entry}

\begin{lem}
	\label{lem:stablepqtstructure}
	The subcategory of $\mathcal{C}$ generated under colimits and extensions by objects of the form $\Sigma^{p,q}X_+$ is the non-negative part of an accessible $t$-structure on $\mathcal{C}$.
\end{lem}

\begin{proof}
	This is an immediate consequence of \cite[Proposition 1.4.4.11]{HA}.
\end{proof}

\begin{defn}
	\label{def:pqtstructure}
	Write $\mathcal{C}_{\geq (p,q)}$ for the subcategory of $\mathcal{C}$ generated under colimits and extensions by objects of the form $\Sigma^{p,q} X_+$ (the non-negative part of an accessible $t$-structure by Lemma~\ref{lem:stablepqtstructure});  $\mathcal{C}_{\leq (p,q)}$ for the corresponding non-positive part and $\tau_{\leq (p,q)}$ or $\tau_{\geq (p,q)}$ for the corresponding truncation functors.  
\end{defn}

By Corollary~\ref{cor:stablelocalizationsymmetricmonoidal}, we know that the localization functor $\mathrm{L}^{p,q}$ on $\mathcal{C}_{\ge 0}$ is symmetric monoidal.  The next result identifies $\mathrm{L}^{p,q}$ with truncation for the $t$-structures of Definition~\ref{def:pqtstructure} and weakly $S^{p,q}$-cellular objects with the non-negative part of the $t$-structure.

\begin{proposition}
	\label{prop:pqlocalizationvstruncation}
Assume $k$ is a field and $\mathcal{C}$ is as in \textup{Paragraph~\ref{par:conditionsonC}}.
\begin{enumerate}[noitemsep,topsep=1pt]
	\item There is an identification $\mathcal{C}_{\geq (p,q)} = O(S^{p,q})$.
 	\item There is a canonical isomorphism of functors $\mathrm{L}^{p,q} \weq \tau_{< (p,q)}$ on $\mathcal{C}$. 
 \end{enumerate}
\end{proposition}

\begin{proof}
	That $\mathcal{C}_{\geq (p,q)}$ consists of weakly $S^{p,q}$-cellular objects holds because $\mathcal{C}_{\geq (p,q)}$ is generated under colimits and cofiber extensions by $\Sigma^{p,q} X_+$ and the latter is weakly $S^{p,q}$-cellular by Example~\ref{ex:basicexample}, the fact that $\mathrm{L}^{p,q}$ is symmetric monoidal (Corollary~\ref{cor:stablelocalizationsymmetricmonoidal}) and preserves colimits (Amplification~\ref{amplification:leftbousfieldcolimitsstablecases}).
	
	To conclude, it suffices to show that $\mathcal{C}_{< (p,q)}$ consists of $S^{p,q}$-null objects.  Indeed, assuming that statement is true, suppose $E$ is an arbitrary object of $\mathcal{C}$.  There is a cofiber sequence of the form
	\[
	\tau_{\geq (p,q)}E \longrightarrow E \longrightarrow \tau_{< (p,q)} E.
	\]
	Since $\tau_{\geq (p,q)}E$ is weakly $S^{p,q}$-cellular, then since $S^{p,q}$-localization preserves colimits, Amplification~\ref{amplification:leftbousfieldcolimitsstablecases} implies that $E \to \tau_{<(p,q)}E$ is an $S^{p,q}$-equivalence, i.e., that the map 
	\[
	\mathrm{L}^{p,q} E \longrightarrow \mathrm{L}^{p,q} \tau_{< (p,q)} E 
	\]
	is an equivalence.  On the other hand, since $\tau_{<(p,q)} E$ is $S^{p,q}$-null by assumption, we conclude that the natural map $\tau_{<(p,q) E} \to \mathrm{L}^{p,q}\tau_{\leq (p,q)}E$ is an equivalence as well.  These two observations yield the natural isomorphism we seek.
	
    To see that $\mathcal{C}_{< (p,q)}$ consists of $S^{p,q}$-null objects, we proceed as follows.
Generating $S^{p,q}$-equivalences are of the form $S^{p,q}_+ \wedge X_+ \to X_+$, and hence split as sums of maps of the form $S^{p,q} \wedge X_+ \to 0$ and $X_+ \xrightarrow{\id} X_+$.
Since any object is local with respect to the second class of maps, $E \in \mathcal C_{\ge 0}$ being $S^{p,q}$-local is equivalent to requiring that $\Map(S^{p,q} \wedge X_+, E) = 0$ for all $X \in \Sm_k$.
This is precisely the definition of $\mathcal{C}_{< (p,q)}$.
\end{proof}	

\begin{rem}
	Proposition~\ref{prop:pqlocalizationvstruncation} allows us to build a functorial weak $S^{p,q}$-cellularization of any spectrum.  Moreover, it highlights the fundamental difference between $S^{p,q}$-localization in the unstable and stable situations.  If $E$ is any object, we have the cofiber sequence
	\[
	\tau_{ \geq (p,q)} E \longrightarrow E \longrightarrow \tau_{< (p,q)} E,
	\]
	which is also a fiber sequence.  Thus $\fib(E \to \mathrm{L}^{p,q} E)$ is itself weakly $S^{p,q}$-cellular.  In contrast, working unstably, it is unclear that the homotopy fiber of the map $\mathscr{X} \to \mathrm{L}^{p,q}\mathscr{X}$ is itself weakly $S^{p,q}$-cellular.  We now use the description of localization in terms of truncation with respect to a $t$-structure to establish a number of properties. 
\end{rem}

\begin{entry}
	Up to shifts, the $t$-structures introduced above have been studied in several places including \cite[\S 4]{Bachmanngenslices} and \cite[\S 6]{BachmannYakerson}.  We may identify the positive and negative parts in terms of homotopy sheaves.  We summarize these observations in the following result.
\end{entry}

\begin{lem}
	\label{lem:gmconnectivitystablehomotopysheaves}
	The following statements hold for $\mathcal{C}$ as in \textup{Paragraph~\ref{par:conditionsonC}}. 
	\begin{enumerate}[noitemsep,topsep=1pt]
	\item The equality $O(S^{p,q}) = \mathcal{C}(q) \cap \mathcal{C}_{\geq p-q}$ holds, i.e., the weakly $S^{p,q}$-cellular objects of $\mathcal{C}$ are the $q$-effective and $(p-q)$-{\em connective} objects.
	\item If $E \in \mathcal{C}_{\geq 0}$, then the following statements are equivalent 
	\begin{enumerate}[noitemsep,topsep=1pt]
		\item $E \in \mathcal{C}_{< (p,q)}$, i.e., $E$ is $S^{p,q}$-null,
		\item $\bpi_i(E)_{-q} = 0$ for $i \geq p$.
	\end{enumerate}
	\item If $\alpha: E \to E' \in \mathcal{C}_{\geq 0}$ is such that $\fib(\alpha) \in O(S^{p,q})$ and $E'$ is connected, then $\alpha$ is a $S^{p,q}$-equivalence. 
	\item If $E \in O(S^{p,q})$ and $E' \in O(S^{p',q'})$, then $E \wedge E' \in O(S^{p+p',q+q'})$.  
	\end{enumerate}
\end{lem}

\begin{proof}
	The first and the second statements follow from \cite[Lemma 6.2(3)]{BachmannYakerson} and \cite[Proposition 4]{Bachmanngenslices} after shifting.  For the third statement, since $E'$ is connected we conclude that there is a cofiber sequence in $\mathcal{C}_{\geq 0}$ of the form
	\[
	\fib(\alpha) \longrightarrow E \longrightarrow E',
	\]
	and the result follows from Amplification~\ref{amplification:leftbousfieldcolimitsstablecases}.  The final statement follows from the first two points.
\end{proof}

\begin{rem}
	\label{rem:slicesvstruncation}
Note that $\tau_{\ge(n,n)}: \mathcal{C}_{\ge 0} \to \mathcal{C}_{\ge 0}$ coincides with the functor usually denoted $f_n$, i.e., the $n$-th slice cover (this follows from \cite[Lemma 6.2(2)]{BachmannYakerson}).
\end{rem}

One immediate consequence of the description of stable weak $S^{p,q}$-cellularity and $S^{p,q}$-nullity in terms of homotopy sheaves is the fact that cellularity is preserved by $R$-localization for $R \subset \Q$ a subring.  We record this consequence for ease of reference.  

\begin{cor}
	\label{cor:stablecellularityispreservedbyRlocalization}
	Suppose $R \subset \Q$ is a subring, and $E \in \mathcal{C}$ as in \textup{Paragraph~\ref{par:conditionsonC}}.  If $E$ is $S^{p,q}$-null or weakly $S^{p,q}$-cellular, then so is $E_{R}$.
\end{cor}

\subsubsection*{Detecting weak-cellularity stably}
We may also refine the Whitehead theorem for nilpotent motivic spaces \cite[Theorem 5.2]{ABHWhitehead} to a statement about stable detection of $S^{p,q}$-equivalences.  In particular, the result below implies that for a given nilpotent space, weak cellularity can be detected stably. To state the result, recall the real realization functor $r_\sigma$ attached to an ordering $\sigma$ of a field $k$ from Definition~\ref{def:rsigma}.

\begin{theorem}[Weakly cellular Whitehead theorem]
	\label{thm:pqequivsof1connectedspacesaredetectedstably}
	Let $k$ be a perfect field of exponential characteristic $e$ and $f: \mathscr{X} \to \mathscr{Y}$ be a map of nilpotent motivic spaces.
	Suppose one of the following conditions holds.
	\begin{enumerate}[noitemsep,topsep=1pt]
		\item $\Sigma_{S^1}^{\infty}f$ has fiber in $O(S^{p+q,q})$
		\item $\Sigma^\infty_{S^1} f$ has cofiber in $O(S^{2,2})$ and $\Sigma^\infty f$ has fiber in $O(S^{p+q,q})$
		\item $\Sigma^\infty_{S^1} f$ has cofiber in $O(S^{2,2})$, $\Sigma^\infty f$ has fiber in $O(S^{p,0})$ and $\Sigma^\infty f \wedge H\Z$ has cofiber in $O(S^{q,q})$
		\item $\Sigma^\infty_{S^1} f$ has cofiber in $O(S^{2,2})[1/e]$, $k$ has finite virtual $2$-étale cohomological dimension, $\Sigma^\infty f \wedge H\Z$ has fiber in $O(S^{p+q,q})$, and for every ordering $\sigma$ of $k$, $r_\sigma(f) \wedge H\Z$ has fiber in $O(S^p)$
	\end{enumerate}
	Then $f$ is an $S^{p+q,q}$-equivalence.
\end{theorem}

\begin{proof}
	(1) We want to show that $\mathrm{L}^{p+q,q}f$ is an equivalence.  The spaces $\mathrm{L}^{p+q,q}\mathscr{X}$ and $\mathrm{L}^{p+q,q} \mathscr{Y}$ are again nilpotent by appeal to Corollary~\ref{cor:pqtruncationpreservesnilpotence} (or simply Corollary~\ref{cor:pqtruncationpreservesconnectivity} if $p \geq 1$).  By the Yoneda lemma, it suffices to show that for any nilpotent $S^{p+q,q}$-null space $\mathscr{T}$, the map $\Map(\mathscr{Y},\mathscr{T}) \to \Map(\mathscr{X},\mathscr{T})$ is an equivalence.  \todo{move the following somewhere else?}Since $\mathscr{T}$ is nilpotent the map $\mathscr T \to \mathscr T_{<p}$ admits a principal refinement \cite[Theorem 5.1]{ABHWhitehead}.  In particular, there are a sequence of $S^{p+q,q}$-null spaces $\mathscr{T}_i$ and fiber sequences of the form
	\[
	\mathscr{T}_{i+1} \longrightarrow \mathscr{T}_i \longrightarrow K(\mathbf{A}_i,n_i)
	\]
	where $n_i \geq p+1$, $\mathscr{T} = \lim_i \mathscr{T}_i$, $\mathscr T_0 = \mathscr T_{<p}$ and $\mathbf{A}_i$ is a subquotient of $\bpi_j(\mathscr T)$ for some $j \ge p$\NB{exact ref?}. Since contraction is exact and $\bpi_j(\mathscr T)_{-q} = 0$ also $(\mathbf{A}_i)_{-q} = 0$. It now suffices to show that $\Map(\mathscr{Y},\mathscr W) \to \Map(\mathscr{X},\mathscr W)$ is an isomorphism, for  $\mathscr W = K(\mathbf{A}_i,n_i)$ and $\mathscr W = \mathscr T_0$. In the former case we use that since $(\mathbf{A}_i)_{-q} = 0$, $\Sigma^{n_i}H\mathbf{A}_i \in \SH^{S^1}(k)_{\geq 0}$ is $S^{p+q,q}$-null.  Since $\Sigma_{S^1}^{\infty}f$ is an $S^{p+q,q}$-equivalence, we conclude that $\Map(\mathscr{Y},K(\mathbf{A}_i,n_i)) \to \Map(\mathscr{X},K(\mathbf{A}_i,n_i))$ is an isomorphism for each $i \geq 0$ as required.  In the latter case we use that $\Sigma^\infty_{S^1} f$ has fiber in $O(S^{p+q,q})$, which is $(p-1)$-connected, and hence $f$ also has $(p-1)$-connected fiber \cite[Theorem 5.2]{ABHWhitehead}. Thus $\mathscr X_{<p} \weq \mathscr Y_{<p}$ and consequently \[ \Map(\mathscr X, \mathscr T_{<p}) \weq \Map(\mathscr X_{<p}, \mathscr T_{<p}) \weq \Map(\mathscr Y_{<p}, \mathscr T_{<p}) \weq \Map(\mathscr Y, \mathscr T_{<p}). \]
	
	(2) It will be enough to show: if $E \in \SH^{S^1}(k)(2)_{\ge 0}$ and $\sigma^\infty E \in \SH(k)(q)_{\ge p}$ then $E \in \SH^{S^1}(k)(q)_{\ge p}$ (apply this to $\fib(\Sigma^{\infty}_{S^1} f)$ and then use (1)). We may assume that $q \ge 2$ (by Lemma \ref{lem:gmconnectivitystablehomotopysheaves}(1)).
	Since $\SH^{S^1}(k)(q)^\heartsuit \weq \SH(k)(q)^\heartsuit$ (Theorem \ref{thm:conservativityofgmstabilization}), it suffices to prove that $E \in \SH^{S^1}(k)(q)$ as soon as $\sigma^\infty E \in \SH(k)(q)$ (and $E \in \SH^{S^1}(k)(2)_{\ge 0}$).
	Replacing $E$ by $E/f_qE$, we may assume that $\Omega^{q,q} E = 0$.
	Suppose that in addition $E \in \SH^{S^1}(k)_{\ge n}$ and consider the map $\iota: E \to \omega^\infty \sigma^\infty E$.
	Then $\bpi_n(\iota)$ is an isomorphism (Theorem \ref{thm:conservativityofgmstabilization}).
	It follows that $\bpi_n(\omega^\infty \sigma^\infty E)_{-q} = 0$, whence $\bpi_n(\omega^\infty \sigma^\infty E) = 0$ by \cite[Lemma 6.2(iii)]{BachmannYakerson}.
	Thus $E \in \SH^{S^1}(k)_{\ge n+1}$.
	Since $n$ was arbitrary, $E=0$ as needed.
	
	(3) By Lemma \ref{lem:gmconnectivitystablehomotopysheaves} it is enough to show for $E \in \SH(k)^{\veff}$ that if $E \wedge H\Z \in \SH(k)^{\eff}(q)$ then $E \in \SH(k)^{\eff}(q)$.
	For this see \cite[Theorem 3.1.1(II)]{bondarko-effectivity}.
	
	(4) Using (3) and the interaction between inverting $\rho$ and real realization established in \cite{Bachmannret}, this immediate from Proposition \ref{prop:hurewicz-improved}.
\end{proof}

\subsubsection*{Stable cellularity and looping}
In the stable setting, the characterizations of cellularity above allow us to establish weak cellular estimates for loop spaces.  The fact that $\Omega$ changes weak cellular class in the expected way is straightforward from the $t$-structure description.  On the other hand, the weak cellular estimate for $\Omega^{1,1}$ relies on Levine's analysis of the homotopy coniveau tower and the slice filtration \cite{Levineconiveautower} or another non-trivial ingredient like motivic infinite loop space theory (see Remark~\ref{rem:deloopingpreservescellularity}).  

\begin{proposition}
	\label{prop:stablecellularityandlooping}
	Assume $\mathcal{C}$ is as in \textup{Paragraph~\ref{par:conditionsonC}.  }Suppose $E \in \mathcal{C}_{\geq 0}$ and $p,q \geq 0$ are integers. 
	For (2) and (3) below, assume that either $k$ is infinite or $\mathcal C = \SH(k)^{\eff}$.\todo{Write this to explain why $2$-effective suffices?}
	\begin{enumerate}[noitemsep,topsep=1pt]
		\item If $E \in O(S^{p,q})$, then $\Omega E \in O(S^{p-1,q})$.
		\item If $E \in O(S^{p,q})$, then $\Omega^{1,1}E \in O(S^{p-1,q-1})$.
		\item If $E \in O(S^{p,q})$, then $\Omega^{2,1}E \in O(S^{p-2,q-1})$.
	\end{enumerate}
\end{proposition}

\begin{proof}
	For the first statement, we want to show that $E \in \Sigma \mathcal{C}_{\geq(p-1,q)}$, but this follows from Lemma~\ref{lem:gmconnectivitystablehomotopysheaves}(1).  For the second statement we first treat $\mathcal{C} = \SH(k)^{\eff}$.  In that case, since $\Omega^{1,1} \Sigma^{p,q} X_+ = \Sigma^{p-1,q-1}X_+$, it follows that $\Omega^{1,1}$ lowers the weight of the generators.  Since $\Omega^{1,1}$ preserves colimits, the result then follows by Proposition~\ref{prop:generationundercolimits}.  When $\mathcal{C} = \SH^{S^1}(k)_{\geq 0}$, it suffices to observe that\cite[Theorem 7.4.2]{Levineconiveautower} implies that $\Omega^{1,1}$ sends $\SH^{S^1}(k)(q+1)$ into $\SH^{S^1}(k)(q)$. The third statement follows by combining the first two.
\end{proof}



While the functors $\sigma^{\infty}$ and $\Sigma^{\infty}_{S^1}$ preserve weakly cellular objects since they preserve colimits and generators, the corresponding statements for the adjoint functors are not immediate.  The following result is key to our comparison of stable and unstable cellularity in the sequel. 

\begin{proposition}
	\label{prop:deloopingpreservesconnectivity}
	Let $k$ be a perfect field, and $p \in \Z, q \ge 0$.
The functors 
\[ 
\omega^{\infty}: \SH(k)^{\veff} \longrightarrow \SH^{S^1}(k)_{\geq 0}, \text{ and }
\] 
\[ 
\Omega^{\infty}_{S^1}: \SH^{S^1}(k)_{\geq 0} \longrightarrow \ho{k} 
\] 
preserve weakly $S^{p,q}$-cellular objects, $S^{p,q}$-equivalences and $S^{p,q}$-null objects. (For the second functor we assume in addition that $p-q \ge 0$.)  In particular, if $p-q \geq 0$, then the functor
\[
\Omega^{\infty}: \SH(k)^{\veff} \longrightarrow \Spc(k)
\]
preserves weakly $S^{p,q}$-cellular objects, $S^{p,q}$-equivalences and $S^{p,q}$-null objects.
\end{proposition}

\begin{proof}
It is clear that both functors preserve $S^{p,q}$-null objects.
To conclude, it will be enough to show that the functors send $S^{p,q}$-nullifications to $S^{p,q}$-equivalences.

For the first functor, let $E \in \SH(k)^{\veff}$.
We have the cofiber sequence $\tau_{\ge (p,q)} E \to E \to \tau_{< (p,q)} E \weq \L^{p,q} E$ where the last equivalence is a consequence of Proposition~\ref{prop:pqlocalizationvstruncation}.  The second map is thus an $S^{p,q}$-nullification, and so it will be enough to show that $\omega^\infty$ preserves weakly $S^{p,q}$-cellular objects.
Therefore, it remains to show that $\omega^{\infty}$ sends $\SH(k)^{\eff}(n)$ to $\SH^{S^1}(k)(n)$; this assertion is a consequence of \cite[Theorem 9.0.3]{Levineconiveautower}.  Note: the aforementioned theorem implicitly assumes $k$ is infinite and perfect because of its appeal to the techniques of proof used in that of \cite[Theorem 7.1.1.]{Levineconiveautower}, however as mentioned in \cite[Remark 9.0.4]{Levineconiveautower}, the assumption of infinitude of $k$ is not necessary for spectra satisfying Axiom A3 \cite[p. 233]{Levineconiveautower}.  That Axiom A3 always holds for spectra in the image of $\omega^{\infty}$ is a consequence of \cite[Corollary 9.4.3]{LevineChowMoving}; see Remark~\ref{rem:deloopingpreservescellularity} for additional clarification regarding these hypotheses. 

For the second functor, let $E \in \SH^{S^1}(k)_{\ge 0}$.
A variant of Construction \ref{const:smallobject} yields a (transfinite) sequence of pushout squares in $\SH^{S^1}(k)_{\ge 0}$
\begin{equation*}
\begin{CD}
\Sigma^\infty_{S^1} \Sigma^i S^{p,q} \wedge X_+ @>>> E_n \\
@VVV                        @VVV \\
0 @>>> E_{n+1}
\end{CD}
\end{equation*}
with $i \ge 0$ and $\colim_n E_n \weq \L^{p,q} E$.  

As in Corollary~\ref{cor:CMon-pushout-mixed}, consider the adjoint pair of functors 
\[
\xymatrix{
B^{\infty}_{mot}: \mathrm{CMon}(\Spc(k)) \ar@<.5ex>[r] & \ar@<.5ex>[l] \SH^{S^1}(k)_{\ge 0}: \Omega^{\infty}_{S^1}.
}
\]  
and $\Sigma^{\infty}_{S^1}: \Spc(k)_* \to \mathrm{CMon}(\Spc(k))$.  The functor $\Omega^{\infty}_{S^1}$ is fully faithful by appeal to Corollary \ref{thm:SHS-CMon-pres-colim}, so the above pushout square is obtained by applying $B^{\infty}_{mot}$ to a pushout square in $\mathrm{CMon}(\Spc(k))$
\begin{equation*}
\begin{CD}
\Omega^{\infty}\Sigma^{\infty}_{S^1} \Sigma^i S^{p,q} \wedge X_+ @>>> \Omega^{\infty}_{S^1} E_n \\
@VVV                        @VVV \\
0 @>>> E'_{n+1}.
\end{CD}
\end{equation*}
Corollary~\ref{cor:CMon-pushout-mixed} tells us that $E'_{n+1} \weq \Omega^{\infty}_{S^1} E_{n+1}$, and so $\Omega^{\infty}_{S^1} E_n \to \Omega^{\infty}_{S^1} E_{n+1}$ is an $S^{p,q}$-equivalence; hence so is $\Omega^\infty_{S^1} E_n \to \Omega^\infty_{S^1} E_{n+1}$ (see Amplification \ref{amplification:localizationpreservesmonoids}).
After all is done we conclude that $\Omega^\infty_{S^1} E \to \Omega^\infty_{S^1} \L^{p,q} E$ is an $S^{p,q}$-equivalence, as required.  The last statement follows by combining the preceding results.
\end{proof}

\begin{rem}
	\label{rem:deloopingpreservescellularity}
	Tracking the hypotheses in \cite{LevineChowMoving} and \cite{Levineconiveautower} used above; there is an implicit assumption that $k$ additionally has characteristic not equal to $2$.  This assumption seems to stem from appeal to a preliminary form of Morel's computations regarding the zeroth stable homotopy of the motivic sphere; this assumption is lifted in the published version.  Nevertheless, an alternative proof of the results of Levine used above can also be found in \cite[Theorem 5.1]{bachmann-slice}.
\end{rem}

\subsection{Examples}
\label{ss:cellularexamples}
In this section, we catalog weak cellular estimates of various standard constructions; many of these facts will be used frequently in the sequel.

\begin{ex}
	\label{ex:smashproducts}
	If $\mathscr{X} \in O(A)$ is a pointed space, then for any integer $n > 1$, $\mathscr{X}^{\times n} \in O(A)$ and $\mathscr{X}^{\sma n} \in O(A)$ as well.  Indeed, the first statement is an immediate consequence of Lemma~\ref{lem:localizationmappingspaces}.  The second statement follows by induction from the first and the observation that objects in $O(A)$ are stable under colimits and extensions, i.e., Lemma~\ref{lem:leftbousfieldcolimits}.  In more detail, repeated application of this observation shows that if $\mathscr{X},\mathscr{Y} \in O(A)$, then so are $\mathscr{X} \vee \mathscr{Y}$ and there is a cofiber sequence of the form
	\[
	\mathscr{X} \vee \mathscr{Y} \longrightarrow \mathscr{X} \times \mathscr{Y} \longrightarrow \mathscr{X} \wedge \mathscr{Y},
	\]
	which shows that $\mathscr{X} \wedge \mathscr{Y} \in O(A)$ as well.
\end{ex}
	
\begin{ex}	
	\label{ex:Jamesmodels}
	It follows that if $\mathscr{X} \in O(A)$ is in addition connected, then $\Omega \Sigma \mathscr{X} \in O(A)$ as well.  Indeed, by \cite[Theorems 2.1.1 and 2.4.2]{AWW} we have the James model $J(\mathscr{X})$ of $\Omega \Sigma \mathscr{X}$, which is a filtered colimit of spaces $J_n(\mathscr{X})$, where $J_0(\mathscr{X}) = \ast$ and for $i \geq 0$ there are cofiber sequences of the form
	\[
	J_{i}(\mathscr{X}) \longrightarrow J_{i+1}(\mathscr{X}) \longrightarrow \mathscr{X}^{\sma i+1}.
	\]
	Induction using the fact that $\mathscr{X}^{\sma i+1} \in O(A)$ shows that $J_i(\mathscr{X}) \in O(A)$ for all $i$, and thus $J(\mathscr{X}) = \colim_i J_i(\mathscr{A}) \in O(A)$ as well.
\end{ex}


\begin{ex}
	\label{ex:thomspaces}
	If $X$ is a smooth $k$-scheme and $\pi: \mathscr{E} \to X$ is a rank $r$ vector bundle on $X$, then $\Th(\pi) \in O(S^{2r,r})$.  This statement is immediate if $\pi$ is a trivial bundle by appeal to Example~\ref{ex:basicexample}.  In general, if we choose an open cover $u: U \to X$ along which $\pi$ trivializes, then there is an equivalence $\colim_{n \in \Delta^{\opcat}} C(u) \to X$.  It follows that $\Th(\pi)$ is realized as a colimit of Thom spaces of trivial bundles, and one concludes by appeal to Lemma~\ref{lem:leftbousfieldcolimits}.  
\end{ex}

\begin{ex}
	\label{ex:purityandconnectedness}
	If $X$ is an irreducible smooth $k$-scheme and $Z \subset X$ is a smooth closed subscheme everywhere of codimension $c$, then homotopy purity yields a cofiber sequence of the form:
	\[
	X \setminus Z \longrightarrow X \longrightarrow \Th(\nu_{Z/X}),
	\] 
	where $\nu_{Z/X}$ is the normal bundle of $Z$ in $X$.  If particular, if $X \setminus Z \in O(S^{p,q})$ for $p \geq 2c, q \geq c$, then $X \in O(S^{2c,c})$ as well.  
\end{ex}

\begin{ex}[Projective spaces]
	\label{ex:projectivespaces}
	The space ${\mathbb P}^n \in O(S^{2,1})$.  Indeed, there are cofiber sequences of the form ${\mathbb P}^{n-1} \to {\mathbb P}^n \to {\pone}^{\sma n}$.  Example~\ref{ex:smashproducts} shows that ${\pone}^{\sma n} \in O(S^{2,1})$ and we conclude by appeal to Lemma~\ref{lem:leftbousfieldcolimits}.  It follows also that ${\mathbb P}^{\infty} \weq K(\Z(1),2) \in O(S^{2,1})$.
	
	Likewise, we may consider the Panin-Walter quaternionic projective spaces $\mathrm{HP}^n$ \cite[\S 3]{PaninWalter}.  For every $n \geq 1$ there are cofiber sequences of the form
	\[
	\mathrm{HP}^n \longrightarrow \mathrm{HP}^{n+1} \longrightarrow \Th(\nu_n),
	\]
	where $\nu_n$ is a rank $2$ bundle.  Example~\ref{ex:thomspaces} shows that $\Th(\nu_n) \in O(S^{4,2})$ and $\mathrm{HP}^1 \weq S^{4,2}$ by \cite[Theorem 2]{ADF}.  By induction and appeal to Lemma~\ref{lem:leftbousfieldcolimits} we conclude that $\mathrm{HP}^{n} \in O(S^{4,2})$ and thus that $\mathrm{HP}^{\infty} \weq B\op{SL}_2 \in O(S^{4,2})$ as well.
\end{ex}

\begin{ex}
	\label{ex:bialynickibirula}
	Suppose $X$ is a connected smooth projective variety equipped with a $\gm{}$-action with isolated fixed points, then $X \in O(S^{2,1})$.  In this case, the variety $X$ admits a Bialynicki-Birula paving by affine spaces \cite{BialynickiBirula} and can be written as an increasing union of smooth open subschemes with smooth closed complements at each stage.  In this case, the result follows by induction appealing to Example~\ref{ex:purityandconnectedness} upon noticing that $X$ has an open dense subscheme isomorphic to ${\mathbb A}^n$.  It follows that if $\op{G}$ is a split, simply connected, semi-simple $k$-group, then $\op{G}/\op{P} \in O(S^{2,1})$ for any parabolic subgroup $\op{P} \subset \op{G}$.  In particular, when $\op{G} = \op{GL}_n$ we deduce that Grassmannians and generalized flag varieties lie in $O(S^{2,1})$.  Since $\op{BGL}_n \weq \colim_n Gr_{r,N}$, we see $\op{BGL}_n \in O(S^{2,1})$ as well. 
\end{ex}

\begin{ex}
	\label{ex:thomspacescobordism}
	The space $\op{MGL}_n$ is the Thom space of the universal rank $n$ vector bundle on $\op{BGL}_n$.  Writing the former as a colimit of the universal rank $n$ vector bundles over $\op{Gr}_{n,N}$, we conclude from Example~\ref{ex:thomspaces} that $\op{MGL}_n \in O(S^{2n,n})$.  Analogously, $\op{MSL}_n$ is constructed as the Thom space of the universal rank $n$ vector bundles over $B\op{SL}_n$ and these may be written as a colimit of vector bundles over finite-dimensional approximations to $\op{BSL}_n$.  Thus, $\op{MSL}_n \in O(S^{2n,n})$.  Similarly, $\op{MSp}_{2n} \in O(S^{4n,2n})$ since this may be realized as the Thom space of a universal rank $2n$ bundle on $\op{BSp}_{2n}$.
\end{ex}

\begin{ex}[Motivic Eilenberg--Mac Lane spaces]
	\label{ex:motivicEMspacesare2nnncellular}
	Assume $k$ is a field that has characteristic $0$.  Voevodsky showed in \cite{VZeroSlice} that if $n \geq 2$, then $K(\Z(n),2n) \in O(S^{2n,n})$, generalizing the first part of Example~\ref{ex:projectivespaces}; see Paragraph~\ref{par:motivicEMspaces} for further explanations of the notation.  Note: while this weak cellularity estimate holds without the assumption on the characteristic (the spectrum representing motivic cohomology is indeed very effective and thus the discussion of Remark~\ref{rem:veryeffectivespectracellularity} applies), Voevodsky gave a geometric argument establishing this cellularity that is illustrative.
	
	Indeed, \cite[Lemma 6.5]{VZeroSlice} shows that we may write $K(\Z(n),2n)$ as a filtered colimit of spaces $K_{i,n}$ such that $K_{0,n} = \ast$ and fitting into cofiber sequences of the form
	\[
	K_{i,n} \longrightarrow K_{i+1,n} \longrightarrow \Sym^i(S^{2n,n}),
	\]
	and \cite[Theorem 4.1]{VZeroSlice} implies that $\Sym^n(S^{2n,n}) \in O(S^{2n,n})$ (this latter statement uses the assumption that $k$ has characteristic $0$).  We will revisit and refine this observation in Section~\ref{s:symmetricpowers}.
\end{ex}

\begin{rem}
	\label{rem:veryeffectivespectracellularity}
More generally, provided $k$ is perfect, for any $E \in \SH(k)^{\veff}$ we have $E_n \in O(S^{2n,n})$.
Since $E_n = \Omega^\infty \Sigma^{2n,n} E$, this follows from Proposition \ref{prop:deloopingpreservesconnectivity}.
\end{rem}

\begin{proposition}
	\label{prop:cellularityofgroups}
	Suppose $k$ is a field and $n \geq 1$ is an integer.
	\begin{enumerate}[noitemsep,topsep=1pt]
		\item The simply-connected split $k$-group schemes $\op{SL}_{n+1}, \op{Sp}_{2n}, \op{Spin}_{n+2}$ and $\op{G}_2$ lie in $O(S^{3,2})$.
		\item The (Nisnevich) classifying spaces of any of the group schemes appearing in \textup{Point (1)} lies in $O(S^{4,2})$.
	\end{enumerate}
\end{proposition}

\begin{proof}
	Observe that $\op{SL}_2 \weq \op{Sp}_2 \weq \op{Spin}_3$ and all are equivalent to $S^{3,2}$.  Next, recall that for $n \geq 2$ there are fiber sequences of the form
	\[
	\begin{split}
		\op{Q}_{2n-1} \longrightarrow B_{\Nis}\op{SL}_{n-1} &\longrightarrow B_{\Nis}\op{SL}_n \\
		\op{Q}_{n} \longrightarrow  B_{\Nis}\op{Spin}_{n} &\longrightarrow B_{\Nis}\op{Spin}_{n+1}, \\
		\op{Q}_{4n-1} \longrightarrow B_{\Nis}\op{Sp}_{2n-2} &\longrightarrow B_{\Nis} \op{Sp}_{2n}, \\
		\op{Q}_6 \longrightarrow  B_{\Nis}\op{SL}_3 &\longrightarrow B_{\Nis}\op{G}_2;
	\end{split}
	\]
	the first and the third sequences are described in \cite{AHWII,AHWIII}, the fourth in \cite[Proposition 3.1.1]{AHWOctonion}, while the second is established in \cite{AHWQuadrics}.  
	
	For the first statement, since $\Omega B_{\Nis}\op{G} \weq \op{G}$ we can rotate the above fiber sequences to obtain sequences relating groups and suitable quadrics.  Since $\op{Q}_{2n} \weq S^{2n,n}$ and $\op{Q}_{2n-1} \weq S^{2n-1,n}$ (see by \cite[Theorem 2]{ADF}), the first statement then follows immediately from Proposition~\ref{prop:pqconnectivityandfibersequencesI}(2) and induction.
	
	The second statement can be established similarly.  Begin by observing that $B\op{SL}_2 \weq B_{\Nis}\op{SL}_2$ and the former is equivalent to $\op{HP}^{\infty} \in O(S^{4,2})$ by Example~\ref{ex:projectivespaces}.  For $r \geq 4$, $Q_{r} \in O(S^{4,2})$, so Proposition~\ref{prop:pqconnectivityandfibersequencesI}(1) and induction imply that $B_{\Nis}\op{SL}_n, B_{\Nis}\op{Sp}_{2n} \in O(S^{4,2})$ for all $n \geq 2$ and then that $B_{\Nis}\op{G}_2 \in O(S^{4,2})$.  For the spin groups, one proceeds in the same way after observing that $\op{Spin}_4 \cong \op{SL}_2 \times \op{SL_2}$ so $B_{\Nis}\op{Spin}_4 \weq B_{\Nis}\op{SL}_2 \times B_{\Nis}\op{SL}_2 \in O(S^{4,2})$ via Example~\ref{ex:smashproducts}.  
\end{proof}

\begin{cor}
	\label{cor:geometricclassifyingspacescellularity}
	The geometric classifying spaces $B\op{SL}_n$ and $B\op{Sp}_{2n} \in O(S^{4,2})$ for every $n \geq 1$.
\end{cor}


\begin{ex}
	\label{ex:2nnnullspaces}
	The classifying space $B\gm{} = K(\Z(1),2)$ (see Example~\ref{ex:weightone}) is $S^{4,2}$-null.  This observation has two natural generalizations: Voevodsky's motivic Eilenberg--Mac Lane space $K(\Z(n),2n)$ is $S^{2n+2,n+1}$-null.  The first non-trivial stage of the simplicial Postnikov tower of $K(\Z(n),2n)$ is the Eilenberg--Mac Lane space $K(\K^M_n,n)$, which is again $S^{2n+2,n+1}$-null.  All of these facts follow from standard facts about contractions; see, e.g., \cite[Example 2.4]{HWBK}.
\end{ex}



\section{Weakly-cellular resolutions of spaces and consequences}
\label{s:cellularWhiteheadtower}
In this section, we amplify some of the results of the previous section.  We begin in Section~\ref{ss:refinedwhitehead} by constructing refinements of the classical Whitehead and Postnikov towers taking weak $S^{p,q}$-cellularity into account: the $k$-invariants in the refined towers can be taken to be $\pone$-infinite loop spaces under suitable hypotheses (see Amplification~\ref{amplification:2-eff-whitehead}).  We then deduce a number of consequences of this refinement including refined cellularity estimates on various kinds of loop spaces (Proposition~\ref{prop:unstableconnectivityofloopspaces}), fibers and cofibers.  We also show that weak and strong cellularity agree under appropriate hypotheses.

\subsection{A weakly-cellular Whitehead tower}
\label{ss:refinedwhitehead}
If $\mathscr{X}$ is a space that is weakly $S^{p,q}$-cellular for some $(p,q)$, then there is no reason for the first non-trivial layer of the classical Postnikov tower of $\mathscr{X}$ to again be weakly $S^{p,q}$-cellular as the following example shows.

\begin{ex}
	\label{ex:postnikovlayersarenotcellular}
	Consider the motivic Eilenberg--MacLane space $K(\Z(n),2n)$, $n \geq 2$.  Note that $K(\Z(n),2n) \in O(S^{2n,n})$ if $k$ has characteristic $0$ by appeal to Example~\ref{ex:motivicEMspacesare2nnncellular}.  The first non-trivial stage of the Postnikov tower and the first non-trivial $k$-invariant yield a sequence of morphisms of the form
	\[
	K(\Z(n),2n) \longrightarrow K(\K^M_n,n) \stackrel{k_n}{\longrightarrow} K(\mathbf{H}^{n-1,n}_{\Z},n+1)
	\]
	where $\mathbf{H}^{n-1,n}_{\Z}$ is the sheaf associated with the presheaf $U \mapsto H^{n-1,n}(U,\Z)$.  Note that $(\mathbf{H}^{n-1,n}_{\Z})_{-i} = \mathbf{H}^{n-1-i,n-i}_{\Z}$.  Since $\mathbf{H}^{0,1}_{\Z} = 0$ by the computation of weight $1$ motivic cohomology, it follows that $K(\mathbf{H}^{n-1,n}_{\Z},n+1)$ is $S^{p+n,n}$-null for any $p \geq 0$ by appeal to Corollary~\ref{cor:truncatednessviahomotopysheaves}.  Observe that $K(\K^M_n,n)$ cannot be weakly $S^{p+n,n}$-cellular as the non-zero $k$-invariant above corresponds to a non-trivial map to an $S^{p+n,n}$-null space.  
\end{ex}

\subsubsection*{Weakly cellular covers of nilpotent motivic spaces}
Example~\ref{ex:postnikovlayersarenotcellular} suggests how cellularity can be ``recovered" by adding cells in higher degrees.  We now analyze this idea in greater detail.  We begin by making the following definition.

\begin{defn}
	\label{defn:cellularcover}
For a motivic space $\mathscr X$ set 
\[
\tau_{\ge(p,q)} \mathscr X := \fib(\mathscr X \to \L^{p,q} \mathscr X).
\]
\end{defn}

We aim to demonstrate that if $\mathscr{X}$ is a nilpotent motivic space, then $\tau_{\ge(p,q)} \mathscr{X}$ is the $S^{p,q}$-cellular cover of $\mathscr{X}$.  In the situation of classical algebraic topology, if $W$ is a space and $P_W$ is the associated nullification functor, then for any pointed connected space $X$, $\fib(X \to P_WX)$ is weakly $W$-cellular \cite[Theorem 1.H.2]{Farjoun}.  The proof of this assertion given by Dror uses fiberwise localization \cite[1.F]{Farjoun}, which is not immediately available in the motivic setting.  As a consequence, we use a more ad hoc approach.  We combine the fact that nilpotent motivic spaces have Postnikov towers that admit principal refinements and the fact that $\L^{p,q}$ interacts well with these fiber sequences via Corollary~\ref{cor:pqtruncationpreservesnilpotence}.  First, we treat the case of Eilenberg--Mac Lane spaces of strictly $\aone$-invariant sheaves, which is straightforward to analyze by passing to the stable setting.

\begin{lem}
	\label{lem:taupq-K}
	If $\mathbf{A}$ is a strictly $\aone$-invariant sheaf, then
	\begin{enumerate}[noitemsep,topsep=1pt]
		\item $\tau_{\ge (p,q)} K(\mathbf A, n) \weq \Omega^\infty_{S^1} \tau_{\ge (p,q)} \Sigma^n H\mathbf A$,
		\item if $p-q \le n$, then $\tau_{\ge (p,q)} K(\mathbf A, n) \weq \Omega^\infty_{S^1} f_q \Sigma^n H\mathbf A$, and
		\item there is an equivalence and weak-cellular estimate of the form:
		\[ 
		\Omega \tau_{\ge (n+q,q)} K(\mathbf A, n) \weq \tau_{\ge (n-1+q,q)} K(\mathbf A, n-1) \in O(S^{n-1+q,q}).
		\]
	\end{enumerate}
\end{lem}

\begin{proof}
	Since $K(\mathbf{A},n) = \Omega^{\infty}_{S^1} \Sigma^n H \mathbf{A}$, the first statement follows by combining Propositions~\ref{prop:pqlocalizationvstruncation} and \ref{prop:deloopingpreservesconnectivity}.  The second statement follows from unwinding the definitions (see Remark~\ref{rem:slicesvstruncation}).  The final equivalence follows from the preceding statements, and the weak-cellular estimate is a consequence of Proposition~\ref{prop:stablecellularityandlooping}.
\end{proof}

\begin{rem}
	At this stage, the space $\tau_{\geq(n+q,q)}K(\mathbf{A},n)$ can be quite strange.  For example, if $\mathbf A$ is a birational sheaf, then $\mathbf A_{-1} = 0$ and so $\tau_{\ge(n+q,q)}K(\mathbf{A},n) = 0$ for $q>0$.
\end{rem}




\begin{lem}
If $\mathscr Y \in O(S^{p,q})$, then $\Map(\mathscr Y, \tau_{\ge(p,q)} \mathscr X) \weq \Map(\mathscr Y, \mathscr X)$.
\end{lem}

\begin{proof}
Taking mapping spaces preserves fiber sequences and $\Map(\mathscr Y, \L^{p,q} \mathscr X) = *$ by definition.
\end{proof}

\begin{lem} \label{lem:taupq-inv}
If $\Omega^{q,q}\mathscr X$ is $d$-connective for $d \ge p$, then  $\tau_{\ge(p+q,q)}\mathscr X \weq \tau_{\ge(d+q,q)}\mathscr X$.
\end{lem}

\begin{proof}
Immediate from Proposition \ref{prop:pqtruncationpreservesconnectivity-v2}.
\end{proof}

We now combine the observations above to inductively build a ur-tower whose layers are weakly cellular--later we will modify this construction slightly.

\begin{construction} 
\label{cons:cellular-postnikov}
Suppose given a nilpotent motivic space $\mathscr X$, $p > 0$, and a principal refinement of $\mathscr X \to \mathscr X_{<p}$ as in Corollary \ref{cor:pqtruncationpreservesnilpotence}.  Applying the functor $\tau_{\ge(p+q,q)}$ to the fiber sequences $\mathscr X_{i+1} \to \mathscr X_i \to K(\mathbf A_i, n_i)$ yields sequences 
\[ 
\tau_{\ge(p+q,q)}\mathscr X_{i+1} \longrightarrow \tau_{\ge(p+q,q)} \mathscr X_i \longrightarrow \tau_{\ge(p+q,q)} K(\mathbf A_i, n_i). 
\] 
We call these a weakly-cellular Postnikov tower of $\tau_{\ge(p+q,q)} \mathscr X$ because it has the following properties: 
\begin{itemize}[noitemsep,topsep=1pt]
	\item $\tau_{\ge(p+q,q)} \mathscr X_0 = *$ (since $\mathscr X_0 = \mathscr X_{<p}$ is $S^{p+q,q}$-null)
	\item $\tau_{\ge(p+q,q)} K(\mathbf A_i, n_i) \weq \tau_{\ge(n_i+q,q)} K(\mathbf A_i, n_i)$ (by Lemma \ref{lem:taupq-inv}, since $n_i \ge p$), and so the connectivity of these spaces tends to infinity with $i$
	\item the displayed sequences above are fiber sequences (since $n_i > p$, so $\L^{p+q,q} \mathscr X_{i+1} \to \L^{p+q,q} \mathscr X_{i} \to \L^{p+q,q} K(\mathbf A_i, n_i)$ is a fiber sequence by Proposition \ref{prop:pqtruncationandfibersequences}).
	\item $\lim_i \tau_{\ge (p+q,q)} \mathscr X_i \weq \tau_{\ge (p+q,q)} \mathscr X$ (since the same holds for $\L^{p+q,q}$, by Corollary \ref{cor:pqtruncationpreservesnilpotence}).
\end{itemize}
\end{construction}

\begin{theorem} \label{thm:fiber-to-trunc-is-cell}
If $k$ is a perfect field and $\mathscr X \in \Spc(k)$ is a nilpotent motivic space, then $\tau_{\ge(p+q,q)} \mathscr X \in O(S^{p+q,q})$.
\end{theorem}

\begin{proof}
We use the weakly-cellular Postnikov tower of $\tau_{\ge(p+q,q)} \mathscr X$ from Construction~\ref{cons:cellular-postnikov}.  We first show by induction that each $\tau_{\ge(p+q,q)}\mathscr X_i \in O(S^{p+q,q})$; for $i = 0$, this is immediate.  Consider the rotated fiber sequence: 
\[ 
\Omega  \tau_{\ge(p+q,q)} K(\mathbf A_i, n_i) \longrightarrow \tau_{\ge(p+q,q)}\mathscr X_{i+1} \longrightarrow \tau_{\ge(p+q,q)}\mathscr X_i. 
\] 
Lemma~\ref{lem:taupq-inv} gives an estimate on the weak cellular class of $\Omega  \tau_{\ge(p+q,q)} K(\mathbf A_i, n_i)$.  Inductively, we deduce that $\tau_{\ge(p+q,q)}\mathscr X_{i+1} \in O(S^{p+q,q})$ by appeal to Lemma \ref{lem:localizationmappingspaces}(3)).

It remains to prove that $\tau_{\ge (p+q,q)}\mathscr X \in O(S^{p+q,q})$.
Let $\mathscr Y$ be $S^{p+q,q}$-null.
We have $\Map(\tau_{\ge(p+q,q)}\mathscr X, \mathscr Y) \weq \lim_n \Map(\tau_{\ge(p+q,q)}\mathscr X, \tau_{\le n} \mathscr Y)$ and we seek to prove that this space is contractible.
For any fixed $n$, and $i \gg n$, we have $\Map(\tau_{\ge(p+q,q)}\mathscr X, \tau_{\le n} \mathscr Y) \weq \Map(\tau_{\ge(p+q,q)}\mathscr X_i, \tau_{\le n} \mathscr Y)$, so it suffices to show that the latter space is contractible.
But since $\mathscr Y$ is $S^{p+q,q}$-null so is $\tau_{\le n} \mathscr Y$ by Lemma \ref{lem:truncationandpostnikovlayers}, and so the result follows from the first part.
\end{proof}

We also observe that in certain situations, simply taking weakly-cellular covers of the Postnikov layers yields a tower of nullifications.  

\begin{theorem}[Refined Postnikov tower] \label{thm:refined-Postnikov}
Suppose that $\mathscr X \in O(S^{q,q})$ is a simple motivic space.
We have 
\[ 
\tau_{\ge(q,q)}\mathscr \tau_{\le i-1} \mathscr X \weq \L^{i+q,q} \mathscr X. 
\]
Applying $\tau_{\ge(q,q)}$ to the fiber sequences $\mathscr X_{\le i} \to \mathscr X_{<i} \to K(\bpi_i \mathscr X, i+1)$ yields fiber sequences 
\[ 
\mathscr \L^{i+1+q,q} \mathscr X \longrightarrow \L^{i+q,q} \mathscr X \longrightarrow \tau_{\ge(i+1+q,q)}K(\bpi_i \mathscr X, i+1). 
\]
We have $\L^{q,q} \mathscr X = *$ and $\lim_n \L^{q+n,q} \mathscr X = \mathscr X$.
\end{theorem}

\begin{proof}
We apply Construction \ref{cons:cellular-postnikov} with $n_i =i$, $\mathscr X_i = \mathscr X_{<i}$.
The fiber sequences 
\[ 
\tau_{\ge(j+q,q)}K(\bpi_j \mathscr X, j) \longrightarrow \tau_{\ge(q,q)}\mathscr X_{j+1} \longrightarrow \tau_{\ge(q,q)}\mathscr X_{j} 
\] 
show that for $j \ge i$ the map $\tau_{\ge(q,q)}\mathscr X_{j+1} \to \tau_{\ge(q,q)}\mathscr X_{j}$ is an $S^{i+q,q}$-equivalence (i.e. $\tau_{\ge(j+q,q)}K(\bpi_j \mathscr X, j) \in O(S^{i+q,q})$), whereas for $j<i$ the space $\tau_{\ge(q,q)}\mathscr X_{j+1}$ is $S^{i+q,q}$-null if $\tau_{\ge(q,q)}\mathscr X_{j}$ is (i.e. $\tau_{\ge(j+q,q)}K(\bpi_j \mathscr X, j)$ is $S^{i+q,q}$-null).
Thus $\mathscr X \weq \tau_{\ge(q,q)} \mathscr X \to \tau_{\ge(q,q)}\mathscr X_{i}$ is an $S^{i+q,q}$-equivalence (use Corollary \ref{cor:limit-of-Spq-equiv}) to an $S^{i+q,q}$-null space, as needed for the first statement.

Now we prove the remaining statements.
The fiber sequence $\mathscr X_{\le i} \to \mathscr X_{<i} \to K(\bpi_i \mathscr X, i+1)$ is preserved by $\L^{q,q}$ by Proposition \ref{prop:pqtruncationandfibersequences}, and hence also by $\tau_{\ge(q,q)}$.
We have $\tau_{\ge(q,q)} K(\bpi_i \mathscr X, i+1) \weq \tau_{\ge(i+1+q,q)}K(\bpi_i \mathscr X, i+1)$ by Lemma \ref{lem:taupq-inv}.
Since $\mathscr X \in O(S^{q,q})$ we have $\L^{q,q} \mathscr X = *$.
Finally $\lim_n \L^{q+n,q} \mathscr X = \mathscr X$ by Lemma \ref{lem:Lpq-map-conn}.
\end{proof}
 
\subsubsection*{Refined Whitehead towers}
Recall also from Proposition~\ref{prop:functorialwhiteheadtowers} that a  pointed connected motivic space with nilpotent fundamental sheaf of groups has a functorially defined Whitehead tower.  We now show how to modify this tower in a fashion that preserves weak cellular classes.

\begin{construction} 
	\label{cons:cellular-whitehead}
	Let $\mathscr X \in \Spc(k)_{\ast}$ be a connected motivic space with $\bpi_1\mathscr{X}$ nilpotent and suppose further that $\mathscr X \in O(S^{q,q})$.  Let $n_i, \mathbf A_i$ be the integers and sheaves from Proposition~\ref{prop:functorialwhiteheadtowers}.  We can inductively and uniquely build commutative diagrams of fiber sequences
	\begin{equation*}
	\begin{CD}
	\mathscr X\lra{i+1,q} @>>> \mathscr X\lra{i,q} @>>> \tau_{\ge (n_i + q,q)} K(\mathbf A_i, n_i) \\
	@VVV @VVV @VVV \\
	\mathscr X\lra{i+1} @>>> \mathscr X\lra{i} @>>> K(\mathbf A_i, n_i)
	\end{CD}
	\end{equation*}
	such that the following statements hold:
	\begin{enumerate}[noitemsep,topsep=1pt]
	\item $\mathscr X\lra{0, q} = \mathscr X$
	\item $\mathscr X\lra{i, q} \in O(S^{n_i + q, q})$,
	\item $\Omega^{q,q} \mathscr X\lra{i, q} \weq \Omega^{q,q} \mathscr X\lra{i}$.
	\end{enumerate}
\end{construction}

\begin{proof}
Suppose we have obtained $\mathscr X\lra{i,q}$ and proved (2), (3) for it.
(Note that this is trivial for $i=0$.)
Since $\mathscr X\lra{i,q} \in O(S^{n_i + q, q})$, the composite 
\[ 
\mathscr X\lra{i,q} \longrightarrow \mathscr X\lra{i} \longrightarrow K(\mathbf A_i, n_i) 
\] 
factors uniquely through $\tau_{\ge (n_i + q,q)} K(\mathbf A_i, n_i)$.
We thus uniquely obtain the space $\mathscr X\lra{i+1,q}$.
The fiber sequence 
\[ 
\Omega \tau_{\ge (n_i + q,q)} K(\mathbf A_i, n_i) \longrightarrow \mathscr X\lra{i+1,q} \longrightarrow \mathscr X\lra{i,q} 
\] 
shows (via Lemmas~\ref{lem:taupq-K} and \ref{lem:localizationmappingspaces}) that $\mathscr X\lra{i+1,q} \in O(S^{n_i - 1 + q,q})$.
By construction and induction we have $\Omega^{q,q} \mathscr X\lra{i+1,q} \weq \Omega^{q,q} \mathscr X\lra{i+1}$, which is $n_{i+1}$-connective.
We claim that also $\mathscr X\lra{i+1,q}$ is connected; it will follow via Corollary \ref{cor:pqtruncationpreservesconnectivity} that $\mathscr X\lra{i+1,q} \in O(S^{n_{i+1} + q,q})$, as needed.

Proving that $\mathscr X\lra{i+1,q}$ is connected is only an issue when $n_i = 1$.\NB{this is a bit of a mess}
We must show that $\bpi_1 \mathscr X\lra{i,q} \to \bpi_1 \tau_{\ge (1 + q,q)} K(\mathbf A_i, 1)$ is surjective.
Consider the fiber sequence 
\[ 
F \longrightarrow \Sigma^\infty_{S^1} \mathscr X\lra{i,q} \longrightarrow f_q \Sigma H \mathbf A_i \weq \tau_{\ge (1+q,q)}\Sigma H\mathbf A_i \in \SH^{S^1}(k). \]
Since $\bpi_1 \mathscr X\lra{i,q} \to \bpi_1 \Sigma^\infty_{S^1} \mathscr X\lra{i,q} \weq (\bpi_1 \mathscr X\lra{i,q})^{ab}_{\A^1}$ is surjective \cite[Lemma 4.1(2)]{ABHWhitehead}, it suffices to prove that $F \in \SH^{S^1}(k)_{\ge 1}$.
But also $F \in \SH^{S^1}(k)(q)_{\ge 0}$, and so it is enough to prove that $\bpi_1(\Sigma^\infty_{S^1} \mathscr X\lra{i,q})_{-q} \to \bpi_1(f_q \Sigma H \mathbf A_i)_{-q}$ is surjective \cite[Lemma 6.2(3)]{BachmannYakerson}.
For this it suffices to show that the composite $\bpi_1(\mathscr X\lra{i,q})_{-q} \to \bpi_1(\Sigma^\infty_{S^1} \mathscr X\lra{i,q})_{-q} \to \bpi_1(f_q \Sigma H \mathbf A_i)_{-q}$ is surjective.
Note that $\bpi_1(f_q \Sigma H \mathbf A_i)_{-q} \weq (\mathbf A_i)_{-q}$ and $\bpi_1(\mathscr X\lra{i,q})_{-q} \weq \bpi_1(\mathscr X\lra{i})_{-q}$.
It is thus enough to prove that $\bpi_1(\mathscr X\lra{i}) \to \mathbf A_i$ is surjective, which holds since $\mathscr X\lra{i+1}$ is connected.
\end{proof}

\begin{rem} \label{rem:whitehead-nilpotent}
The construction shows that each $\mathscr X\lra{i,q}$ is a motivic space with nilpotent fundamental sheaf of groups.  Similarly if $\mathscr X$ is a nilpotent motivic space, then an induction argument making repeated appeal to \cite[Corollary 3.3.7]{AFHLocalization} together with the fact that the spaces $\tau_{\ge (n_i + q,q)} K(\mathbf A_i, n_i)$ are nilpotent (being grouplike commutative monoids) implies that each $\mathscr X\lra{i,q}$ is a nilpotent motivic space.
\end{rem}

\begin{amplification} 
	\label{amplification:2-eff-whitehead}
Let $k$ be a perfect field of exponential characteristic $e$.  Assume $\mathscr X \in \Spc(k)_{\ast}$ is a connected motivic space with $\bpi_1\mathscr{X}$ nilpotent and suppose further that $\mathscr X \in O(S^{q,q})$.
Assume further that one of the following holds:
\begin{itemize}[noitemsep,topsep=1pt]
\item $q \geq \epsilon(k)$ (e.g. $q \ge 2$), or
\item $q \ge \epsilon'(k)$ and $\mathbf A_i$ is $e$-periodic for all $i$.
\end{itemize}

If $\mathscr{X}\lra{i,q}$ is the tower from \textup{Construction~\ref{cons:cellular-whitehead}}, then there exist fiber sequences 
\[ 
\mathscr X\lra{i+1,q} \longrightarrow \mathscr X\lra{i,q} \longrightarrow \Omega^\infty \Sigma^{n_i + q,q} \mathscr B_i 
\] 
for some $\mathscr B_i \in \SH(k)^{\eff\heartsuit}$.  The connectivity of $\mathscr X\lra{i,q}$ as well as the sequence $n_i$ tends to $\infty$ with $i$.
\end{amplification}

\begin{proof}
The only thing that remains to be established is that $\tau_{\ge (n_i + q,q)} K(\mathbf A_i, n_i) \weq \Omega^\infty \Sigma^{n_i +q,q} \mathscr B_i$ for some $\mathscr B_i$ as claimed.  This fact follows from Theorem \ref{thm:conservativityofgmstabilization}.
\end{proof}

\begin{rem} \label{rmk:modified-layers}
In the proof of Amplification \ref{amplification:2-eff-whitehead}, we used (and established) the following crucial fact: $\tau_{\ge (n_i + q,q)} K(\mathbf A_i, n_i) \weq \Omega^\infty \Sigma^{n_i +q,q} \mathscr B_i$, that is, the layers of the modified tower are infinite $\mathbb P^1$-loop spaces.
\end{rem}

\todo{Add a question about $q = 1$?}

\subsubsection*{Digression: weak vs. strong cellularity}
\begin{cor} \label{cor:detect-iso-Spq-cellular}
Let $f: \mathscr X \to \mathscr Y$, with $\mathscr X, \mathscr Y \in O(S^{q,q})$ both motivic spaces with $\aone$-nilpotent fundamental sheaf of groups.  The map $f$ is an equivalence if and only if the induced map $\bpi_i(\mathscr X)_{-q} \to \bpi_i(\mathscr Y)_{-q}$ is an isomorphism for every $i \ge 2$,
and also the induced maps 
\[ 
(\bpi_1(\mathscr X)^{(j-1)}_{\A^1}/\bpi_1(\mathscr X)^{(j)}_{\A^1})_{-q} \longrightarrow (\bpi_1(\mathscr Y)^{(j-1)}_{\A^1}/\bpi_1(\mathscr Y)^{(j)}_{\A^1})_{-q} 
\] 
are isomorphisms for every $j \ge 1$.
Here $\bpi_1(\mathscr X)^{(j)}_{\A^1}$ refers to the $\A^1$-lower central series of \textup{Paragraph~\ref{par:lowercentralseries}}.
\end{cor}\todo{add question about $\bpi_1$?}
\begin{proof}
Let $\mathscr X\lra{\bullet,q}, \mathscr Y\lra{\bullet,q}$ be functorial refined Whitehead towers for $\mathscr X, \mathscr Y$, with layers $(\mathbf A_\bullet, n_\bullet)$ and $(\mathbf B_\bullet, n_\bullet)$.
We can assume that $n_i = 1$ for $i \le N+1$ (a common bound for the derived lengths of $\bpi_1 \mathscr X$, $\bpi_1 \mathscr Y$) and $N_{N+i} = i$.
Also $\mathbf A_{N+i} = \bpi_i \mathscr X$ for $i > 1$, and similarly for $\mathscr Y$.
Finally $\mathbf A_i \weq \bpi_1(\mathscr X)^{(i-1)}_{\A^1}/\bpi_1(\mathscr X)^{(i)}_{\A^1}$ for $i \le N+1$, and similarly for $\mathscr Y$.
Since $\mathscr X \weq \lim_i \mathscr X\lra{i,q}$ and similarly for $\mathscr Y$, the map $f$ is an equivalence if and only if the induced maps 
\[ 
\tau_{\ge (n_i+q,q)}K(\mathbf A_i, n_i) \longrightarrow \tau_{\ge (n_i+q,q)}K(\mathbf B_i, n_i) 
\] 
are equivalences as well.  Using \cite[Lemma 6.1(1)]{BachmannYakerson}, this just means that $(\mathbf A_i)_{-q} \xrightarrow{\weq} (\mathbf B_i)_{-q}$, which is what we wanted to show.
\end{proof}

\begin{proposition}
Let $\mathscr X \in O(S^{p+q,q})$ with $p \ge 2$.
Then $\mathscr X$ is strongly $S^{p+q,q}$-cellular, that is, $\mathscr X$ lies in the subcategory of $\Spc(k)$ generated under colimits by $S^{p+q,q} \wedge \Spc(k)_*$.
\end{proposition}
\begin{proof}
We may assume $\mathscr X \in \Spc(k)_*$.
We shall construct $\mathscr X_1 \to \mathscr X_2 \to \dots \to \mathscr X$ such that each $\mathscr X_i$ can be built using colimits of the asserted form, and such that the induced map $\bpi_n(\Omega^{q,q}\mathscr X_i) \to \bpi_n(\Omega^{q,q} \mathscr X)$ is an isomorphism for $n \le i$.
Using Corollary \ref{cor:detect-iso-Spq-cellular} we deduce that $\mathscr X \weq \colim_i \mathscr X_i$, concluding the proof.

It remains to construct the $\mathscr X_i$.
Since $\mathscr X$ is simply connected, we may put $\mathscr X_1 = *$; in fact we may put $\mathscr X_i = *$ for $i < p$.
Now inductively suppose $\mathscr X_{i-1}$ has been constructed; we shall construct $\mathscr X_{i}$.
By construction we may assume $i \ge p$.
First set \[ \mathscr Y = \mathscr X_{i-1} \vee \bigvee_{X_+ \wedge S^{i+q,q} \to \mathscr X} X_+ \wedge S^{i+q,q}, \] where the sum is over $X \in \Sm_k$ with a map as indicated.
Then $\bpi_n(\mathscr Y) \weq \bpi_n(\mathscr X_{i-1})$ for $n < i$ (since $\tau_{<i}$ preserves colimits), and hence $\bpi_n(\Omega^{q,q}\mathscr Y) \weq \bpi_n(\Omega^{q,q}\mathscr X_{i-1})$ for the same $n$.
Moreover by construction, $\bpi_i(\Omega^{q,q} \mathscr Y) \to \bpi_i(\Omega^{q,q} \mathscr X)$ is surjective.
Let $\mathscr F = \fib(\mathscr Y \to \mathscr X)$.
Set \[ \mathscr K = \bigvee_{X_+ \wedge S^{i+q,q} \to \mathscr F} X_+ \wedge S^{i+q,q} \] and form the pushout
\begin{equation*}
\begin{CD}
\mathscr K @>>> \mathscr Y \\
@VVV      @VVV \\
* @>>> \mathscr X_i.
\end{CD}
\end{equation*}
We get \[ \bpi_n(\mathscr X_i) \weq \bpi_n(\mathscr Y) \weq \bpi_n(\mathscr X_{i-1}) \] for $n < i$, and \[ \bpi_i(\mathscr X_i) \weq \bpi_i(\mathscr Y)/\bpi_i(\mathscr K) \] (the cofiber of $f: \mathscr Y \to \mathscr X_i$ is $\Sigma \mathscr K$, and so $\fib(f)$ is $i$-connective with $\bpi_i \fib(f) \weq \bpi_i \Omega \Sigma K$ \cite[Theorem 4.1]{AFComparison}).
It follows by induction (and Paragraph \ref{par:contraction}) that $\bpi_n(\Omega^{q,q}\mathscr X_i) \weq \bpi_n(\Omega^{q,q}\mathscr X)$ for $n<i$, and hence \[ \bpi_i(\Omega^{q,q}\mathscr Y)/\bpi_i(\Omega^{q,q}\mathscr F) \weq \bpi_i(\Omega^{q,q}\mathscr X). \]
By construction, $\bpi_i(\Omega^{q,q} \mathscr K) \to \bpi_i(\Omega^{q,q} \mathscr F)$ is surjective.
Using the relationship between $\Omega^{q,q}$ and contraction (see Paragraph \ref{par:contraction}), we conclude that \[ \bpi_i(\Omega^{q,q} \mathscr X_i) \weq \bpi_i(\mathscr X_i)_{-q} \weq \bpi_i(\mathscr Y)_{-q}/\bpi_i(\mathscr K)_{-q} \weq \bpi_i(\Omega^{q,q}\mathscr Y)/\bpi_i(\Omega^{q,q}\mathscr F) \weq \bpi_i(\Omega^{q,q} \mathscr X). \]
This concludes the proof.
\end{proof}

\subsection{Unstable cellularity of loop spaces and fibers}
Using the weakly cellular Whitehead tower constructed above we can deduce consequences about weak cellularity of loop spaces. 

\begin{proposition}
	\label{prop:unstableconnectivityofloopspaces}
Assume $\mathscr{X} \in O(S^{p,q})$ is a motivic space with $\aone$-nilpotent fundamental sheaf of groups.
\begin{enumerate}[noitemsep,topsep=1pt]
	\item The space $\Omega \mathscr{X}$ lies in $O(S^{p-1,q})$ (provided $p>0$).
	\item The space $\Omega^{1,1}\mathscr{X}$ lies in $O(S^{p-1,q-1})$ (provided $p,q>0$).
	\item The space $\Omega^{2,1}\mathscr{X}$ lies in $O(S^{p-2,q-1})$ (provided $p>1,q>0$).
\end{enumerate}
\end{proposition}

\begin{proof}
	The functors $\Omega$, $\Omega^{1,1}$ and $\Omega^{2,1}$ preserve fiber sequences.  Consider the weakly cellular Whitehead tower of Construction~\ref{cons:cellular-whitehead}. Looping, we obtain fiber sequences of the form
	\[
	\Omega^{\infty}_{S^1} \Omega^2 \mathscr{B}_i \longrightarrow \Omega \mathscr{X} \langle i+1,q \rangle \longrightarrow \Omega \mathscr{X} \langle i,q \rangle. 
	\]
	By assumption $\mathscr{B}_i \in O(S^{n_i+q,q})$, so the space $\Omega^2\Omega^\infty_{S^1} \mathscr{B}_i$ lies in $O(S^{n_i-2+q,q})$ by appeal to Proposition~\ref{prop:stablecellularityandlooping}(1) and Proposition~\ref{prop:deloopingpreservesconnectivity}.  Therefore, Proposition~\ref{prop:pqconnectivityandfibersequencesI}(1) implies that $\Omega \mathscr{X} \langle i+1,q \rangle \to \Omega \mathscr{X} \langle i,q \rangle$ is an $S^{p-1,q}$-equivalence for $i$ sufficiently large, say $i \ge N$.
	Since $\L^{p-1,q}$ preserves connectivity (Corollary \ref{cor:pqtruncationpreservesconnectivity}), we find that the space $\L^{p-1,q} \Omega \mathscr X \lra{N,q} \weq \L^{p,q} \Omega \mathscr X \lra{N,q}$ is $\infty$-connected, and hence vanishes.
	Next consider the fiber sequences \[ \Omega \mathscr{X} \langle i+1,q \rangle \longrightarrow \Omega \mathscr{X} \langle i,q \rangle \longrightarrow \Omega^{\infty}_{S^1} \Omega \mathscr{B}_i. \]
	By descending induction (starting with $i=N-1$) we see using Proposition \ref{prop:pqconnectivityandfibersequencesI} that $\Omega \mathscr{X} \langle i,q \rangle \in O(S^{p-1,q})$, as needed.

	The other statements are deduced in a completely analogous fashion, appealing to the other parts of Proposition~\ref{prop:stablecellularityandlooping}.  
	(Note in particular that (2) and (3) are trivially true if $q \in \{0,1\}$, so we may assume that $q \ge 2$ and hence use the version of the Whitehead tower with layers infinite $\mathbb{P}^1$-loop spaces, i.e., Amplification \ref{amplification:2-eff-whitehead}.)
\end{proof}

\begin{rem} \label{rem:prove-cellularity-whitehead}
Suppose given fiber sequences 
\[ 
\mathscr X_{i+1} \longrightarrow \mathscr X_i \longrightarrow \mathscr B_i 
\] 
for $i \ge 0$, with all spaces connected, $\mathscr B_i \in O(S^{p,q})$ for all $i$, $\Omega \mathscr B_i \in O(S^{p,q})$ for $i$ sufficiently large, and the connectivity of the $\mathscr X_i$ tending to infinity with $i$.
Then the argument in the proof of Proposition \ref{prop:unstableconnectivityofloopspaces} shows that each $\mathscr{X}_i \in O(S^{p,q})$.
\end{rem}

\begin{theorem}
	\label{thm:cellularityoffibers}
	If $f: \mathscr{X} \to \mathscr{Y}$ is a map of pointed, connected spaces with $\mathscr X \in O(S^{p-1,q}), \mathscr Y \in O(S^{p,q})$ and $p-q \geq 2$, then $\fib(f) \in O(S^{p-1,q})$.
\end{theorem}

\begin{proof}
	Consider the rotated fiber sequence $\Omega \mathscr Y \to \fib(f) \to \mathscr X$.
	Since $\Omega \mathscr Y \in O(S^{p-1,q})$ by Proposition \ref{prop:unstableconnectivityofloopspaces}, we deduce that $\L^{p-1,q} \fib(f) \weq \L^{p-1,q} \mathscr X$ (Proposition \ref{prop:pqconnectivityandfibersequencesI}).
	This vanishes by assumption.
\end{proof}

\subsection{Weak cellularity: comparing unstable and stable}



We can also refine the motivic Blakers--Massey theorem in a fashion that takes weak cellular estimates into account.

\begin{theorem}
	\label{thm:cofiberpqconnectivity}
	Suppose $\mathscr{F} \stackrel{\iota}{\to} \mathscr{E} \stackrel{f}{\to} \mathscr{B}$ is a fiber sequence of pointed connected motivic spaces. 
	Suppose further that $\mathscr{F} \in O(S^{p,q})$ and $\mathscr{B} \in O(S^{p',q'})$ is a motivic space with $\aone$-nilpotent fundamental sheaf of groups with $p'-q' \geq 1$.
	\begin{enumerate}[noitemsep,topsep=1pt]
		\item $\fib(\cof(\iota) \to \mathscr{B}) \in O(S^{p+p',q+q'})$
		\item The canonical map $\Sigma \mathscr{F} \to \cof(f)$ is an $S^{p+p',q+q'}$-equivalence with cofiber in $O(S^{p+p'+1,q+q'})$.
	\end{enumerate}
\end{theorem}

\begin{proof}
	For the first statement, recall that by \cite[Lemma 3.4(1)]{ABHWhitehead} we know that $\fib(\cof(\iota) \to \mathscr{B})$ is equivalent to $\Sigma \mathscr{F} \wedge \Omega \mathscr{B}$.  Since $\mathscr{B} \in O(S^{p',q'})$, we conclude that $\Omega \mathscr{B} \in O(S^{p'-1,q})$ by Proposition~\ref{prop:unstableconnectivityofloopspaces}(1).  On the other hand, $\Sigma \mathscr{F} \in O(S^{p+1,q})$ by Theorem~\ref{thm:cellularityofsmashproducts}.  It follows that $\Sigma \mathscr{F} \wedge \Omega \mathscr{B} \in O(S^{p+p',q+q'})$ by another application of Theorem~\ref{thm:cellularityofsmashproducts}.
	
	For the second statement, note that Proposition~\ref{prop:pqconnectivityandfibersequencesI} in conjunction with the fiber sequence
	\[
	\Sigma \mathscr{F} \wedge \Omega \mathscr{B} \longrightarrow \cof(\iota) \longrightarrow \mathscr{B}
	\]
	together with the connectivity assertion of the first point implies that $\cof(\iota) \to \mathscr{B}$ is an $S^{p+p',q+q'}$-equivalence.  Thus, $\cof(\cof(\iota) \to \mathscr{B}) \in O(S^{p+p',q+q'})$ by Lemma~\ref{lem:leftbousfieldcolimits}. But also $\cof(\cof(\iota) \to \mathscr{B})$ is $((p-q)+(p'-q'))$-connected \cite[Proposition 3.3]{ABHWhitehead}, and hence $\cof(\cof(\iota) \to \mathscr{B}) \in O(S^{p+p'+1,q+q'})$ by Corollary \ref{cor:separatingcellularity}.
	
	Writing $f$ as the composite $\mathscr{E} \to \cof(\iota) \to \mathscr{B}$ yields a cofiber sequence \cite[Lemma 3.4(2)]{ABHWhitehead} of the form
	\[
	\Sigma \mathscr{F} \longrightarrow \cof(f) \longrightarrow \cof(\cof(\iota) \to \mathscr{B})
	\]
	Since $\Sigma \mathscr{F}$ is $1$-connected and \cite[Proposition 3.3]{ABHWhitehead} implies that $\cof(f)$ is $1$-connected as well, we conclude that $\Sigma \mathscr{F} \to \cof(f)$ is an $S^{p+p',q+q'}$-equivalence by Theorem~\ref{thm:pqequivsof1connectedspacesaredetectedstably} (having $S^1$-stable fiber $\Sigma^{-1} \cof(\cof(\iota) \to \mathscr{B}) \in O(S^{p+p',q+q'})$), as claimed.   
\end{proof}

\begin{cor}
	\label{cor:unstablecellularityoffibersvscofibers}
	Suppose $f: \mathscr{X} \to \mathscr{Y}$ is a map of nilpotent motivic spaces.
	\begin{enumerate}[noitemsep,topsep=1pt]
		\item If $\fib(f) \in O(S^{p,q})$, then $\cof(f) \in O(S^{p+1,q})$.
		\item If $\cof(f) \in O(S^{p+1,q})$, $\mathscr{Y} \in O(S^{3,1})$, $\fib(f) \in O(S^{2,0})$, then $\fib(f) \in O(S^{p,q})$.   
	\end{enumerate}
	In particular, if $f: \mathscr{X} \to \mathscr{Y}$ is a map of $1$-connected spaces, $\mathscr{Y} \in O(S^{3,1})$ and $\fib(f) \in O(S^{2,0})$, then $\fib(f) \in O(S^{p,q})$ if and only if $\cof(f) \in O(S^{p+1,q})$.
\end{cor}

\begin{proof}
	Since $\fib(f) \in O(S^{p,q})$, $\Sigma \fib(f) \in O(S^{p+1,q})$.  In that case, Theorem~\ref{thm:cofiberpqconnectivity}(2) then implies that $\Sigma \fib(f) \to \cof(f)$ is an $S^{p+2,q}$-equivalence, so $\cof(f) \in O(S^{p+1,q})$ as well.
	
	For the second point, suppose $\mathscr{F} \in O(S^{r,r'})$; we know this holds for $r = 2$, $r' = 0$ by assumption.  Since $\mathscr{Y} \in O(S^{3,1})$ is simply connected, by Theorem~\ref{thm:cofiberpqconnectivity}(2) we conclude that $\Sigma \fib(f) \to \cof(f)$ is an $S^{r+3,r'+1}$-equivalence.  Since $\cof(f) \in O(S^{p+1,q})$ by assumption, we conclude that $\Sigma \fib(f) \to \cof(f)$ is more highly connected.  Since $\mathscr{F}$ is simply connected by assumption, Theorem~\ref{thm:pqequivsof1connectedspacesaredetectedstably} implies $\fib(f)$ has higher connectivity as well.  By induction, the result follows.    
\end{proof}

We now refine F. Morel's Freudenthal suspension theorem \cite[Theorem 6.61]{MField} by providing a weak cellular estimate on the fiber of the unit map of the simplicial loops-suspension adjunction.   

\begin{proposition}
	\label{prop:refinedS1Freudenthal}
	If $\mathscr{X} \in O(S^{p,q})$ with $p-q > 1$ (and $q \ge 0$), then $\fib(\mathscr{X} \to \Omega \Sigma \mathscr{X}) \in O(S^{2p-1,2q})$.
\end{proposition}

\begin{proof}
	The space $\Omega \Sigma \mathscr{X} \in O(S^{p,q})$ by appeal to Example~\ref{ex:Jamesmodels}.  Write $s$ for the unit map $\mathscr{X} \to \Omega \Sigma \mathscr{X}$.  Note that $\Sigma \Omega \Sigma \mathscr{X}$ splits as $\bigvee_{i \geq 0} \Sigma \mathscr{X}^{\sma i}$ by the Hilton--Milnor splitting \cite[\S 5]{WickelgrenWilliams} or \cite[Corollary 2.11]{DH}.  Thus, $\Sigma \cof(s) \weq \cof(\Sigma s) \weq \bigvee_{i \geq 2} \Sigma \mathscr{X}^{\sma i}$ and thus lies in $O(S^{2p+1,2q})$ by appeal to Theorem~\ref{thm:cellularityofsmashproducts}.  
	
	By assumption $\mathscr{X} \in O(S^{2,0})$ so Morel's suspension theorem \cite[Theorem 6.61]{MField} implies that $\fib(s) \in O(S^{2,0})$.  In fact, this establishes the result if $q = 0$.  By Corollary~\ref{cor:unstablecellularityoffibersvscofibers}(2) to show that $fib(s) \in O(S^{p,q})$ it suffices to show that $\cof(s) \in O(S^{p+1,q})$.  Therefore, it suffices to show that $\cof(s) \in O(S^{2p,2q})$.  Corollary~\ref{cor:unstablecellularityoffibersvscofibers}(1) shows that $\cof(s)$ is at least $1$-connected as well.   Since $\cof(s)$ is also simply connected, that follows from the weak-cellular estimate of  $\Sigma \cof(s)$ from the preceding paragraph by appeal to the weakly cellular Whitehead theorem ~\ref{thm:pqequivsof1connectedspacesaredetectedstably}.
\end{proof}


\subsubsection*{Weak cellularity and $R$-localization}
Assume $R \subset \Q$ is a subring.  Another consequence of the weakly cellular version of the Whitehead tower is the compatibility of $R$-localization and weak cellularity.  

\begin{proposition}
	\label{prop:Rlocalizationpreservescellularity}
Assume $\mathscr{X} \in \ho{k}_*$ is a nilpotent motivic space lying in $O(S^{p,q})$ for some integers $p \geq q \geq 0$.  If $R \subset \Q$ is a subring, then $\mathscr{X}_{R}$ lies in $O(S^{p,q})$ as well. 
\end{proposition}

\begin{proof}
Consider the refined Whitehead tower of $\mathscr X$ from Construction \ref{cons:cellular-whitehead}, consisting of fiber sequences 
\[ 
\mathscr X\lra{i+1,q} \longrightarrow \mathscr X\lra{i,q} \longrightarrow \mathscr B_i = \tau_{\ge (n_i+q,q)} K(\mathbf A_i, n_i). 
\]
Since all involved spaces are nilpotent \ref{rem:whitehead-nilpotent}, these fiber sequences are preserved by $R$-localization \cite[Theorem 4.3.11]{AFHLocalization}.
Appealing to Remark \ref{rem:prove-cellularity-whitehead}, it remains to prove that $(\mathscr B_i)_R \in O(S^{p,q})$ for all $i$, and $(\Omega \mathscr B_i)_R \in O(S^{p,q})$ for $i$ sufficiently large.
This is true because $\mathscr B_i = \Omega^\infty_{S^1} B_i$ for some $B_i \in \SH^{S^1}(k)(q)_{\ge p-q}$, $\Omega^\infty_{S^1}$ commutes with $R$-localization, and $\SH^{S^1}(k)(q)_{\ge p-q}$ is stable under $R$-localization, being stable under colimits.
\end{proof}

\subsection{Further examples}
\begin{ex}
	\label{ex:poneloopsonponesuspensions}
	One consequence of Proposition~\ref{prop:unstableconnectivityofloopspaces} is the following weak cellular estimate, generalizing Example~\ref{ex:Jamesmodels}: if $\mathscr{X}$ is a pointed, connected space that lies in $O(S^{p,q})$, then $\Omega^{2,1}\Sigma^{2,1}\mathscr{X} \in O(S^{p,q})$.  Since $\mathscr{X}$ is pointed and connected, we conclude that $\Sigma^{2,1}\mathscr{X}$ is in particular $1$-connected and thus nilpotent and furthermore lies in $O(S^{p+2,q+1})$.  In particular, $q+1 > 0$ and therefore, the cellularity estimate follows immediately from Proposition~\ref{prop:unstableconnectivityofloopspaces}(3).  More generally, it follows that $\Omega^{2n,n}\Sigma^{2n,n}\mathscr{X} \in O(S^{p,q})$ under the above hypotheses as well, and also that $Q(\mathscr{X}) \in O(S^{p,q})$. 
\end{ex}

\begin{ex}[Motivic connective covers of $BGL$]
	\label{ex:connectivecoversofBGLI}
	The space $\Z \times BGL$ is the $\pone$-infinite loop space of the spectrum $\mathbf{KGL}$.  For any integer $n \geq 1$, we set 
	\[
	BGL \langle 2n,n \rangle := \Omega^{\infty}\tau_{\geq (2n,n)} \mathbf{KGL}.
	\]  
	In particular, observe that $BGL \langle 2n,n \rangle \in O(S^{2n,n})$ by appeal to Proposition~\ref{prop:deloopingpreservesconnectivity}. 

	Recall from, e.g., \cite[\S 3.5]{AFSpheres} that Suslin described a morphism $S^{2n-1,n} \to GL$ via an explicit matrix; this matrix is the clutching function of a morphism $S^{2n,n} \to BGL$ by \cite[Theorem 4.3.4]{ADF} that stabilizes to the unit map $\1 \to \mathbf{KGL}$ \cite[Theorem 2.3]{SuslinMennicke}.  Observe that the map $S^{2n,n} \to BGL$ factors through $BGL\langle 2n,n\rangle$.  
\end{ex}

\begin{ex}[Motivic connective covers of $BGL$ II]
	\label{ex:connectivecoversofBGLII}
  	Continuing the discussion of Example~\ref{ex:connectivecoversofBGLI}, consider the composite map $S^{2n,n} \to BGL \to K(\Z(n),2n)$, where the second map is given by the $n$-th Chern class.  Such a map is uniquely determined by a morphism $\bpi_{n,0}(S^{2n,n}) = \K^{MW}_n \to \K^M_n$ (see the discussion of \cite[Proof of Theorem 4.2.3]{ADF}) and thus is $(n-1)!$ times a generator by \cite[Theorem 5.2]{SuslinMennicke}. 
	
	The equivalence $BGL_n \weq Gr_n \weq \colim_N Gr_{n,N}$ in conjunction with Example~\ref{ex:bialynickibirula} shows that $BGL_n$ and thus $BGL \in O(S^{2,1})$.  We may identify $BGL = BGL\langle 2,1\rangle$ and $BSL = BGL\langle 4,2\rangle$, the latter using the fact that $BSL$ is the homotopy fiber of $c_1: BGL \to K(\Z(1),2)$.
	
	Now, assume that we work over a field $k$ that has characteristic $0$.
	Since $s_i(\mathbf{KGL}) \weq \Sigma^{2i,i} H\Z$ there is an induced map $c_2: BSL \to K(\Z(2), 4)$ with
	\[
	BGL\langle 6,3 \rangle \weq \fib(c_2: BSL \to K(\Z(2),4)).
	\] 
	One may check that $c_2$ coincides with the usual definition of the second Chern class.
	We view the space $BGL \langle 6,3 \rangle$ as a motivic analog of the space $BU \langle 6 \rangle$.
\end{ex}

\section{Equivariant geometry of symmetric powers}
\label{s:symmetricpowers}

\subsection{The motivic Dold--Thom theorem revisited}
\label{ss:motivicdoldthom}
In this section, after recalling some facts about motivic homotopy theory of group scheme actions, we recall the basic theory of symmetric powers.  In particular, the motivic Eilenberg--Mac Lane space $K(\Z(n),2n)$ admits a model in terms of symmetric powers of the motivic spheres $S^{2n,n}$.  Our main goal is to record a form of Voevodsky's motivic Dold--Thom theorem \cite[Theorem 3.7 and 3.11]{VoeMEM} paying attention to the interaction between the geometry of symmetric powers and the bonding maps in the motivic spectrum corresponding to motivic cohomology (see Theorem~\ref{thm:motivicDoldThom}).  

\subsubsection*{Motivic homotopy theory and group scheme actions}
Assume $G$ is a finite group such that $|G|$ is invertible in $k$.  

\begin{entry}[$G$-quasiprojective varieties]
A $G$-scheme $X$ is called $G$-quasiprojective if it admits an ample $G$-equivariant line bundle.  Note that if $G$ is finite, then any quasi-projective variety equipped with a $G$-action is $G$-quasiprojective: if $\mathscr{L}$ is ample, then $\det  (\bigoplus_{g \in G} g^* \mathscr{L})$ is ample and $G$-equivariant.  
\end{entry}

\begin{entry}[$G$-spaces]
	We extend some of the discussion of Paragraph~\ref{par:motivicspacesvariants} to spaces with $G$-action.  Write $\Sch_k^G$ for the category of $G$-quasiprojective schemes and $\Sm^G_k$ for the subcategory of smooth $G$-quasi-projective varieties.  The categories $\Sm^G_k$ and $\Sch^G_k$ can be equipped with the $G$-Nisnevich topology; see, e.g., \cite[\S 3.3]{HoyoisEquiv} for a definition.  If $\mathrm{S}^G_k$ is either of the categories $\Sm^G_k$ or $\Sch^G_k$, we can consider the $\infty$-topos $\mathrm{P}(\mathrm{S}^G_k)$ of presheaves of spaces.  We write $\Shv_{Nis}^G(\Sm_k)$ for the $\infty$-category of $G$-Nisnevich sheaves of spaces on $\Sm_k^G$ and $\Shv^G_\Nis(\Sch_k)$ for the $\infty$-category of $G$-Nisnevich sheaves of spaces on $\Sch^G_k$. 
\end{entry}

\begin{entry}[motivic $G$-spaces]
We write $\Spc^G(k)$ for the full subcategory of $\Shv_\Nis^G(\Sm_k)$ consisting of $\aone$-invariant sheaves of spaces \cite[Definition 3.12]{HoyoisEquiv}.  Abusing notation, we write $\mathrm{L}_{mot}$ for the motivic localization functor in the $G$-equivariant context as well.  Once again, $\Spc^G(k)$ is a semi-topos \cite[Proposition 3.15]{HoyoisEquiv}.
\end{entry}

\begin{entry}[Quotients]
	If $X$ is a $G$-quasiprojective $k$-scheme, then a quotient of $X$ by the action of $G$ exists \cite[Expose V, Proposition 1.8]{SGA1} and is known to be a quasiprojective $k$-scheme.  If $X$ is furthermore assumed normal, then $X/G$ is once again normal by the universal property of categorical quotients.  The functor of taking the quotient is left adjoint to the functor $\Sch_k \to \Sch^G_k$ sending a quasi-projective $k$-scheme $X$ to $X$ equipped with the trivial $G$-action.  As such, the quotient construction preserves colimits.  
\end{entry}


\begin{entry}[Quotients and $\aone$-equivalences]
	\label{par:quotients}
	By left Kan extending the functor $X \longrightarrow (X/G)$ we obtain $\op{Quot}_{G}: \mathrm{P}(\Sm_k^G) \longrightarrow \mathrm{P}(\Sch_k)$ which induces a functor
	\[
	\op{Quot}_G: \Spc^{G}(k) \longrightarrow \Spc^{ft}_{Nis}(k)
	\]
	(see \cite[Proposition 2.10]{VoeMEM} or \cite[Proposition 45 p. 403]{DeligneVoe}).  
	Noting that $\op{Quot}_G(*)=*$ we see that this also induces a functor on pointed spaces (which, for formal reasons, still preserves colimits).
\end{entry}


\subsubsection*{Symmetric powers and motivic Eilenberg--Mac Lane spaces}
\begin{entry}[Symmetric powers]
\label{par:stabilizationI}
Assume $r \geq 0$ is an integer.  If $X$ is a smooth scheme, then we may consider $X^{\times r}$ with the action of $\Sigma_r$ by permutation of the factors.
Passing to sifted cocompletions we obtain $\mathrm{P}_\Sigma(\Sm_k) \to \mathrm{P}_\Sigma(\Sm_k^G)$.
This functor preserves Nisnevich and motivic equivalences \cite[Theorem 2.5]{VoeMEM} and hence induces a functor $\ho{k} \to \Spc^G(k)$.  In particular, if $\mathscr{X}$ is a pointed motivic space, we will write $\mathscr{X}^{\times r}$ for motivic space with $\Sigma_r$-action obtained from the functor just described.  We then set 
\[
\Sym^r \mathscr{X} := \op{Quot}_{S_r}(\mathscr{X}^{\times r}).
\]
If $\mathscr{X}$ is a pointed space, then the inclusion $S_{n-1} \subset S_{n}$ and the inclusion $\mathscr{X}^{\times n-1} \to \mathscr{X}^{n}$ (via the base-point in the last factor) yield morphisms
\[
\Sym^{n-1} \mathscr{X} \longrightarrow \Sym^n \mathscr{X}
\]
We define
\[
\Sym^{\infty} \mathscr{X} := \colim_n \Sym^n \mathscr{X},
\]
which has naturally the structure of a (strictly) commutative monoid.  
\end{entry}

\begin{entry}[Reduced symmetric powers]\label{par:reducedsymmpowers}
Working with pointed spaces and replacing cartesian products by smash products in the previous discussion, we obtain the \emph{reduced symmetric powers}.
In slightly more detail, we have the functor $\Sm_{k+} \to \Sm_{k+}^G$ sending $X_+$ to $(X_+)^{\wedge n} \weq X^{\times n}_+$.
Passing to sifted cocompletions and localizing one obtains $\ho{k}_* \to \Spc^G(k)_*, \mathscr X \mapsto \mathscr X^{\wedge r}$.
Finally \[ \tilde S^r \mathscr X := \op{Quot}_{S_r}(\mathscr{X}^{\wedge r}). \]
We can use the reduced symmetric powers to identify the cofiber of the stabilization map on unreduced symmetric powers: for $\mathscr X \in \Spc(k)_*$, we have a natural cofiber sequence \cite[Lemma 2.21]{VoeMEM} \[ \Sym^{n-1} \mathscr{X} \longrightarrow \Sym^n \mathscr{X} \to \tilde S^n \mathscr X. \]
\end{entry}

\begin{entry}[Symmetric powers and (pre)spectra]
	\label{par:assemblymapsymmetricpowers}
	For any motivic spaces $\mathscr{X}$ and $\mathscr{Y}$, there is an $S_r$-equivariant map $\mathscr{Y} \times \mathscr{X}^r \to (\mathscr{Y} \times \mathscr{X})^r$ induced by the diagonal.  If $\mathscr{Y}$ and $\mathscr{X}$ are additionally pointed, the above maps induce morphisms $\mathscr{Y} \wedge \Sym^r \mathscr{X} \longrightarrow \Sym^r (\mathscr{Y} \wedge \mathscr{X})$.  These morphisms are compatible with the stabilization morphisms $\Sym^n \mathscr{X} \to \Sym^{n+1} \mathscr{X}$ described in Paragraph~\ref{par:stabilizationI} in the evident sense and taking $\mathscr{Y} = S^{2,1}$ we obtain morphisms
	\[
	s_n: \Sigma^{2,1} \Sym^{\infty} \mathscr{X} \longrightarrow \Sym^{\infty}(\Sigma^{2,1}\mathscr{X}).
	\]
	We encapsulate this discussion as asserting that $\Sym^{\infty}$ passes to a functor on $S^{2,1}$-prespectra (see \cite[\S 2.2.1]{Robalo} for discussion of stabilization).  While we do not use this fact, $\Sym^{\infty}$ actually passes to a functor on $\SH(k)$, i.e., it preserves stable equivalences.  This follows by combining two facts: $\Sym^{\infty}$ preserves sifted colimits, and $\SH(k)$ can be obtained as a suitable localization of the category of $S^{2,1}$-prespectra. More precisely, the functor taking a prespectrum $E$ to its $n$-th space admits a left adjoint $f_n$ sending a pointed motivic space $\mathscr{X}$ to a shifted suspension $S^{2,1}$-prespectrum $f_n\mathscr{X}$ with $f_n\mathscr{X}_i = \ast$ for $i < n$ and $\Sigma^{2(n-i),n-i} \mathscr{X}$ for $n \geq i$.  In that case, $\SH(k)$ is obtained from $S^{2,1}$-prespectra by inverting the maps of prespectra $f_{n+1} \Sigma^{2,1} \mathscr{X} \to f_{n}\mathscr{X}$ for $\mathscr{X} \in \Spc(k)_*$.  Granted those two facts, we may appeal to \cite[Lemma 2.10]{bachmann-norms} to deduce the claim.
\end{entry}


\begin{entry}[Inverting the characteristic exponent]
	\label{par:invertingcharexp}
For an integer $d$ we write
\[
\Sym^{\infty}[1/d]\mathscr{X} := \colim( \Sym^{\infty}\mathscr{X} \stackrel{\times d}{\longrightarrow} \Sym^{\infty}\mathscr{X} \stackrel{\times d}{\longrightarrow} \cdots)
\]
where $\times d$ is multiplication defined by the monoid structure.  In the sequel, we will take $d$ to be the characteristic exponent of the base-field.  This construction is also compatible with stabilization morphism of Paragraph~\ref{par:stabilizationI} so there are induced morphisms
\[
\Sigma^{2,1} \Sym^{\infty}[1/d] \mathscr{X} \longrightarrow \Sym^{\infty}[1/d] (\Sigma^{2,1}\mathscr{X}).
\]
\end{entry}

\begin{entry}
	\label{par:correspondences}
	For $U$ a normal scheme and $X$ quasi-projective, we write $\Cor_k(U,X)$ (resp. $\Cor_k^{\eff}(U,X)$) for the presheaf of finite correspondences.  If $\mathrm{S}_k$ is an admissible category as in Paragraph~\ref{par:motivicspacesvariants}, then we write $\mathrm{Cor}(\mathrm{S}_k)$ for the category whose objects are those of $\mathrm{S}_k$ but whose morphisms are finite $k$-correspondences.  In the notation of \cite[p. 36]{SuslinVoevodskyRelCy}, the morphism between two objects of $\mathrm{S}_k$ in $\mathrm{Cor}(\mathrm{S}_k)$ is given by $c_{equi}(X \times_k Y/X,0)$.  This is a subgroup of the finite $k$-correspondences from $X$ to $Y$, i.e., the free abelian group generated by closed subschemes $Z \subset X \times_k Y$ such that $Z \to X$ is finite and dominates an irreducible component of $X$.  In general it is a proper subgroup \cite[Example 3.4.7]{SuslinVoevodskyRelCy}, but coincides with all finite correspondences when $X$ is regular \cite[Corollary 3.4.6]{SuslinVoevodskyRelCy}.  The category $\mathrm{Cor}(\mathrm{S}_k)$ is additive with direct sum given by disjoint union and has a symmetric monoidal structure such that the graph functor $\Gamma: \mathrm{S}_k \to \mathrm{Cor}(\mathrm{S}_k)$ is symmetric monoidal.
\end{entry}


\begin{entry}[Motivic Eilenberg--Mac Lane spaces]
	\label{par:motivicEMspaces}
 	Voevodsky's motivic Eilenberg spaces admit the following explicit model:
	\[
	K(\Z(n),2n) = \Z_{tr}({\mathbb A}^n)/\Z_{tr}({\mathbb A}^n \setminus 0).
	\]
	More generally, for any smooth scheme $X$, we write $\Z_{tr}(\Sigma^{2n,n} X_+)$ for the motivic space $\Z_{tr}(X \times {\mathbb A}^n)/\Z_{tr}(X \times {\mathbb A}^n \setminus 0)$.
\end{entry}

\begin{entry}
Consider the model of $S^{2n,n}$ given by ${\mathbb A}^n/{\mathbb A}^n \setminus 0$, i.e., the Thom space of a trivialized rank $n$ vector bundle over $\Spec k$.  There is a canonical map commutative monoids \cite[\S3.1]{VoeMEM}
\[
\Z_{tr}({\mathbb A}^n)/\Z_{tr}({\mathbb A}^n \setminus 0) \longrightarrow \Sym^{\infty} S^{2n,n};
\]
(going back to \cite[p. 81]{SuslinVoevodskySing}).
\end{entry}

\begin{ex}
	\label{ex:weightone}
	The case $n = 1$ of the map from the preceding point is particularly simple and we do not have to leave smooth schemes to form the relevant quotient.  In that case, $S^{2,1} \weq \pone$ and we have $\Sym^{\infty} \pone \weq {\mathbb P}^{\infty}$.  There is an $\aone$-equivalence ${\mathbb P}^{\infty} \weq B\gm{}$ \cite[\S 4 Proposition 4.8]{MV}.  On the other hand, the description of weight $1$ motivic cohomology shows that $\Z_{tr}(\aone)/\Z_{tr}(\aone \setminus 0) \weq B\gm{}$ \cite[\S 2]{SuslinVoevodskySing}.  Putting these observations together, we conclude that the map
	\[
	 K(\Z(1),2) \longrightarrow \Sym^\infty S^{2,1}
	\]
	is an equivalence.
\end{ex}

\subsubsection*{The motivic Dold--Thom theorem}
\begin{entry}
Let $E$ be a $\mathbb{P}^1$-prespectrum, that is, a sequence of spaces $E^0, E^1, \dots \in \Spc(k)_*$ together with bonding maps $\Sigma^{2,1} E^i \to E^{i+1}$.
The assembly maps of Paragraph~\ref{par:assemblymapsymmetricpowers} supply bonding maps in a $\mathbb{P}^1$-prespectrum $\Sym^\infty[1/d] E$ with entries $\Sym^\infty[1/d] E_i$.
\end{entry}

\begin{theorem}[motivic Dold-Thom theorem]
	\label{thm:motivicDoldThom}
	Assume $k$ is a perfect field with characteristic exponent $c$ and $n \geq 1$ is an integer.
	\begin{enumerate}[noitemsep,topsep=1pt] 
	\item The map
	\[
	\Sym^{\infty}[1/c] S^{2n,n}|_{\Sm_k} \longleftarrow K(\Z(n),2n)[1/c]
	\]
	is an isomorphism in $\Spc(k)$.
	\item The assembly morphisms $e_n: \Sigma^{2,1} K(\Z(n),2n)[1/c] \to K(\Z(n+1),2n+2)[1/c]$ are compatible with the equivalence of the first point, i.e., for every $n$ the diagram
	\[
	\xymatrix{
	\Sigma^{2,1} \Sym^{\infty}[1/c] S^{2n,n} \ar[d] & \ar[l] \Sigma^{2,1} K(\Z(n),2n)[1/c] \ar[d]^{e_n} \\
	\Sym^{\infty}[1/c] S^{2n+2,n+1} & \ar[l] K(\Z(n+1),2n+2)[1/c].
	}
	\]
	commutes.  
	\item The morphisms above induce an isomorphism of prespectra:
	\[
	\Sym^{\infty}[1/c] \1_k|_{\Sm_k} \longleftarrow H\Z[1/c].
	\]
	\item More generally for $X \in \Sm_k$ we have isomorphisms 
	\[ 
	\Sym^{\infty}[1/c] \Sigma^{2n,n} X_+|_{\Sm_k} \weq \Z_{tr}(\Sigma^{2n,n} X_+)[1/c] \in \Spc(k) 
	\] 
	and isomorphisms of pre-spectra:
	\[
	\Sym^{\infty}[1/c] \Sigma^{\infty} X_+|_{\Sm_k} \longleftarrow H\Z[1/c] \wedge \Sigma^{\infty} X_+.
	\]
	\end{enumerate}
\end{theorem}

\begin{proof}
The first statement is a consequence of \cite[Theorem 3.7, Propositions 3.11,3.15]{VoeMEM} and the definition of the motivic EM space given above.
Similarly for the first part of the fourth statement.
The third point follows from the second and the discussion of Paragraphs~\ref{par:assemblymapsymmetricpowers} and \ref{par:invertingcharexp}.  
The second half of the fourth point follows similarly, using that $H\Z[1/c] \wedge \Sigma^\infty X_+$ is equivalent to the prespectrum with entries $(\Z_{tr}(X_+)[1/c], \Z_{tr}(\Sigma^{2,1} X_+)[1/c], \dots)$. (This latter claim follows from the fact that $\DM(k, \Z[1/c]) \weq \Mod_{H\Z[1/c]}(k)$, which is proved if $c=1$ in \cite{RondigsOestvaermodules} and for $c>1$ in \cite[Theorem 5.8]{HKO}).

For the second statement, recall that the assembly morphism in the motivic Eilenberg--Mac Lane spectrum is constructed as the composite
\[
S^{2,1} \wedge K(\Z(n),2n) \longrightarrow K(\Z(1),2) \wedge K(\Z(n),2n) \longrightarrow K(\Z(n+1),2n+2);
\]
where, at the level of cycles, the first morphism is induced by the map sending $U \to S^{2,1}$ to its graph viewed as an element of $\Cor_k(U,S^{2,1})$ and the second morphism is induced by external product of cycles: using the models we have, this arises from the commutative diagram: 
\[
\xymatrix{
\Cor_k(U,\aone \setminus 0) \times \Cor_k(U,{\mathbb A}^n \setminus 0) \ar[r]\ar[d] & \Cor_k(U,{\mathbb A}^{n+1} \setminus 0) \ar[d] \\
\Cor_k(U,\aone) \times \Cor_k(U,{\mathbb A}^n) \ar[r] & \Cor_k(U,{\mathbb A}^{n+1})
}
\]
where the top map is the composite of the external product with the inclusion $\aone \setminus 0 \times {\mathbb A}^n \setminus 0 \subset {\mathbb A}^{n+1} \setminus 0$.

Now, the map $\pone \wedge \Sym^d (S^{2n,n}) \longrightarrow \Sym^d(S^{2n+2,n+1})$ is induced by the diagonal map.  On the other hand, the identification of Point (1) ultimately rests on \cite[Theorem 6.8]{SuslinVoevodskySing} (via \cite[Proposition 3.5]{VoeMEM}).  This identification tells us that for any finite type scheme $X$ and any connected normal scheme $U$, the group $\hom(U,\sqcup_{d \geq 0} \Sym^d X)$ coincides after inverting $c$ with the free commutative monoid on integral cycles on $U \times X$ that are finite and surjective over $U$.  The result then follows by comparing definitions: see \cite[\S 3.7 p. 59]{SuslinVoevodskyRelCy} for external product of cycles.  
\end{proof}

\subsection{Cellularity of the cofiber of the assembly map}
\label{ss:cellularityassemblymap}
Suppose $X \in \Sm_k$.  In this section, we establish the main geometric result, which amounts to a weak cellular estimate on the cofiber of the assembly maps in the spectrum $H\Z \wedge \Sigma^{\infty} X_+$.  We begin by giving an ``explicit" description of the cofiber of the assembly map as a colimit of spaces defined using representation theory of the symmetric group.  

After Theorem~\ref{thm:motivicDoldThom}, the discussion of Paragraph~\ref{par:assemblymapsymmetricpowers} identifies the assembly map in $H\Z \wedge \Sigma^{\infty} X_+$ as a colimit of maps of the form
\begin{equation}
	\label{eqn:unstableassemblymap}
\Sigma^{2,1} \Sym^{r}(\Sigma^{2n,n} X_+) \longrightarrow \Sym^{r}(\Sigma^{2n+2,n+1} X_+).
\end{equation}
Taking the cofiber of these maps from $r-1$ to $r$, using \S\ref{par:reducedsymmpowers} we obtain a map
\begin{equation}
	\label{eqn:unstableassemblymap-reduced}
\Sigma^{2,1} \tilde S^{r}(\Sigma^{2n,n} X_+) \longrightarrow \tilde S^{r}(\Sigma^{2n+2,n+1} X_+),
\end{equation}
which we can identify in the following terms. Write $\rho_r$ for the standard representation of the symmetric group $\Sigma_r$ on a based $r$-dimensional $k$-vector space by permutation of the basis vectors.  Assuming $r$ is coprime to the characteristic of $k$, the representation $\rho_r$ decomposes as $\mathbf{1} \oplus \bar{\rho}_r$ and we refer to $\bar{\rho}_r$ as the reduced standard representation.  View $\rho_r^{\oplus n}$ as an $\Sigma_r$-equivariant vector bundle over $\Spec k$.  We write ${\mathbb A}[\bar{\rho}_r]$ for the affine space attached to $\bar{\rho}_r$ (i.e., the spectrum of the symmetric algebra of the dual, viewed as a $\Sigma_r$-scheme) and $p: {\mathbb A}[\bar{\rho}_r] \setminus 0 \to \Spec k$ for the $\Sigma_r$-equivariant structure map.  

\begin{proposition}
	\label{prop:cofiberofassemblyHZsmashX}
	Assume $r! \in k^\times$.
	If $X \in \Sm_k$, then there is a canonical identification
	\[
	\cof(\Sigma^{2,1} \tilde S^r \Sigma^{2n,n} X_+ \to \tilde S^r \Sigma ^{2n+2,n+1} X_+)  \weq \Sigma^{2n+3,n+1} \op{Quot}_{\Sigma_r} (\Th(p^* \bar{\rho}_r^{\oplus n}) \wedge (X^{\times r})_+).
	\]
\end{proposition}

\begin{proof}
The map of \eqref{eqn:unstableassemblymap-reduced} is the map
\[
\Sigma^{2,1} \op{Quot}_{\Sigma_r} (\Th(\rho_r^{\oplus n}) \wedge (X^{\times r})_+) \longrightarrow \op{Quot}_{\Sigma_r} (\Th(\rho_r^{\oplus n+1}) \wedge (X^{\times r})_+)
\]
induced by the map $\mathbf{1} \oplus \rho_r^{\oplus n} \to \rho_r \oplus \rho_r^{\oplus n}$ that includes $\mathbf{1} \subset \rho_r$ in the first factor of the target.  This map identifies with a map
\[
\Sigma^{2n+2,n+1} \op{Quot}_{\Sigma_r}(\Th(\bar{\rho}_r)^{\oplus n} \wedge (X^{\times r})_+) \longrightarrow \Sigma^{2n+2,n+1} \op{Quot}_{\Sigma_r}(\Th(\bar{\rho}_r)^{\oplus n+1}  \wedge (X^{\times r})_+),
\]
which is itself obtained by suspension from the map
\[
\op{Quot}_{\Sigma_r}(\Th(\bar{\rho}_r)^{\oplus n}  \wedge (X^{\times r})_+) \longrightarrow \op{Quot}_{\Sigma_r}(\Th(\bar{\rho}_r)^{\oplus n+1}  \wedge (X^{\times r})_+)
\]
induced by the inclusion $\bar\rho^{\oplus n} \to \bar\rho^{\oplus n+1}$.
In order to identify the cofiber of this map, we simply observe that $\op{Quot}_{\Sigma_r}$ is a colimit preserving functor and unwind the definitions.  

Indeed, first consider the commutative square
\[
\xymatrix{
{\mathbb A}[\bar{\rho}_r^{\oplus n}] \setminus 0 \ar[r]\ar[d] & {\mathbb A}[\bar{\rho}_r^{\oplus n}] \ar[d] \\
{\mathbb A}[\bar{\rho}_r^{\oplus n+1}] \setminus 0 \ar[r] & {\mathbb A}[\bar{\rho}_r^{\oplus n+1}].
}
\]
The map on horizontal cofibers is a map $\Th(\bar{\rho}_r^{\oplus n}) \to \Th(\bar{\rho}_r^{\oplus n+1})$ that, after smashing with $(X^{\times r})_+$ and applying $\op{Quot}_{\Sigma_r}$ is the map in which we are interested in.  Observe that the left hand vertical map is $\Sigma_r$-equivariantly equivalent to the inclusion 
\[
({\mathbb A}[\bar{\rho}_r^{\oplus n}] \setminus 0) \times {\mathbb A}[\bar{\rho}_r] \longhookrightarrow {\mathbb A}[\bar{\rho}_r^{\oplus n+1}] \setminus 0,
\]
with closed complement $\Sigma_r$-equivariantly isomorphic to ${\mathbb A}[\bar{\rho}_r] \setminus 0$ and normal bundle $p^* \bar{\rho}_r^{\oplus n}$.  Therefore, $\Sigma_r$-equivariant homotopy purity \cite[Theorem 3.23]{HoyoisEquiv} yields an identification of the cofiber of the left hand vertical map with $\Th(p^* \bar{\rho}_r^{\oplus n})$.  

With those observations in mind, computing cofibers vertically, then yields a $\Sigma_r$-equivariant cofiber sequence of the form
\[
\Th(p^* \bar{\rho}_r^{\oplus n}) \longrightarrow \ast \longrightarrow \cof(\Th(\bar{\rho}_r^{\oplus n}) \to \Th(\bar{\rho}_r^{\oplus n+1})).
\]
Shifting the cofiber sequence to the right once, smashing with $(X^{\times r})_+$ and applying $\op{Quot}_{\Sigma_r}$, we therefore obtain an identification
\[
\cof(\op{Quot}_{\Sigma_r}(\Th(\bar{\rho}_r)^{\oplus n} \wedge (X^{\times r})_+) \to \op{Quot}_{\Sigma_r}(\Th(\bar{\rho}_r)^{\oplus n+1}) \wedge (X^{\times r})_+) \weq \Sigma \op{Quot}_{\Sigma_r} (\Th(p^* \bar{\rho}_r^{\oplus n}) \wedge (X^{\times r})_+),
\]
and the result follows.
\end{proof}

\begin{ex}
	For concreteness, let us describe what happens when $r = 2$ and $X = \Spec k$. In that case, $\bar{\rho}_2$ is the $1$-dimensional sign representation of $\Sigma_2$.  The action of $\Sigma_2$ on $\gm{} = {\mathbb A}(\bar{\rho}_2) \setminus 0$ is the sign representation, which is free and has quotient $\gm{}$ via the squaring map.  Thus, the cofiber in the lemma may be identified as:
	\[
	\Sigma^{2n+3,n+1} \op{Quot}_{\Sigma_2}\Th(p^*\bar{\rho}_2^{\oplus n}).
	\]
	Since the action of $\Sigma_2$ on $\gm{}$ is free, the bundle $p^*\bar{\rho}_2^{\oplus n}$ is $\Sigma_2$-equivariantly trivial and we conclude the cofiber in question is
	\[
	{\gm{}}_+ \wedge S^{4n+3,2n+1},
	\]
	which lies in $O(S^{4n+3,2n+1})$.
\end{ex}   

\begin{rem}
Intuitively speaking, since $\op{Quot}_{\Sigma r}$ is a colimit preserving functor, we expect that it preserves weak cellularity classes as well.  We thus expect that the space appearing in Proposition~\ref{prop:cofiberofassemblyHZsmashX} lies in $O(S^{4n+3,2n+1})$ as well.  Unfortunately, the example above is essentially the only case in which the quotients that appear are smooth schemes.  In dealing with singular schemes, we are forced to invoke resolution of singularities.  Voevodsky observed that there is a well-defined $S^1$-stable homotopy type for singular schemes, which we review in the next section.  At the moment, we content ourselves with producing an equivariant version of a weak cellular estimate that we will eventually use.
\end{rem}

We begin by making the following definition. In contrast with Section~\ref{s:weakcellularityandnullity} we write down the cellular class directly instead of as the kernel of a localization since this is enough for our purposes.
It seems likely that the more general theory can also be developed.
\begin{defn}
	We will say that a pointed motivic $G$-space $\mathscr{X}$ lies in $O^G(S^{p,q})$ if it lies in the subcategory generated under colimits and cofiber extensions by objects of the form $\Sigma^{p,q} X_+$ for $X \in \Sm^G_k$.
\end{defn}


Consider the $\Sigma_r$-action on ${\mathbb A}[\bar{\rho}] \setminus 0$.  The stabilizers of the action of $\Sigma_r$ on ${\mathbb A}[\bar{\rho}_r] \setminus 0$ are always {\em proper} subgroups of $\Sigma_r$.  We may stratify ${\mathbb A}[\bar{\rho}] \setminus 0$ or equivalently ${\mathbb A}[\bar{\rho}]$ by stabilizer type; the resulting stratification is frequently called the Luna stratification.

\begin{proposition}[Luna stratification]
	\label{prop:stabilzerstratification}
	Assume $k$ is a perfect field, $X$ is an irreducible smooth affine $G$-scheme and $G$ a finite group whose order is invertible in $k$.  There exists a decomposition of $X$ into locally closed closed subschemes $X(H)$ indexed by conjugacy classes of subgroups $H \subset G$ such that the following properties hold:
	\begin{enumerate}[noitemsep,topsep=1pt]
		\item all points of $X(H)$ have stabilizer conjugate to $H$;
		\item all $X(H)$ are smooth; 
		\item there is a subgroup $H_0$ (the generic stabilizer) and open dense subscheme $X(H_0)$;
		\item the closure of $X(H)$ is a union of $X(H')$ where $H'$ runs through the conjugacy classes of subgroups of $G$ containing the conjugacy class of $H$.
	\end{enumerate}
\end{proposition}

\begin{proof}
	The existence of this stratification is ``well-known" and we content ourselves with providing some references.  Since $|G|$ is invertible in $k$, the group scheme $G$ is a linearly reductive \'etale group scheme.  Since $k$ is perfect, it suffices to prove these properties after base-change to an algebraic closure of $k$.  The existence and properties of the stratification as stated are consequences of Luna's \'etale slice theorem which is established in characteristic $0$ in \cite{Luna}.  In positive characteristic the slice theorem holds for actions of linearly reductive finite groups by \cite{BardsleyRichardson}.  The existence and properties of the stratification above are described in \cite[III.2-4]{Luna}.  
\end{proof}


\begin{lem}
	\label{lem:splittingbundlebystabilizertype}
	Assume $k$ is a perfect field, $G$ is a finite group whose order is invertible in $k$, and $X$ is a smooth $G$-quasi-projective $k$-scheme.  If $V$ is a $G$-equivariant vector bundle on $X$ such that for every conjugacy class of subgroups $H \subset G$,  $V|_{X(H)^H}$ splits $H$-equivariantly as $\mathbf{1}_{X(H)^H}^{\oplus n} \oplus V'_H$, then $\Th(V) \in O^G(S^{2n,n})$.
\end{lem}

\begin{proof}
	If $W$ is a smooth $G$-scheme and $V$ is a $G$-equivariant vector bundle that splits $G$-equivariantly as $V \weq \mathbf{1}_W \oplus V'$, then $\Th(V) \weq \Sigma^{2n,n} \Th(V')$ and hence $\Th(V) \in O^G(S^{2n,n})$.
	Similarly if $V$ locally on $W$ splits off trivial vector bundles of rank $n$, then $\Th(V) \in O^G(S^{2n,n})$ by descent (recall that $O^G(S^{2n,n})$ is closed under colimits by definition).
	
	Given a $G$-equivariant open immersion $U \subset W$ with smooth closed complement $Z$ and a $G$-equivariant vector bundle $V$ on $W$, observe that the inclusion of $V|_U \hookrightarrow V$ has closed complement $V|_Z$ and the normal bundle of $V|_Z \hookrightarrow V$ is identified $G$-equivariantly with $V|_Z \oplus \nu_{Z/W}$.  In that case, by functoriality of the homotopy purity isomorphisms \cite[Corollary 3.25]{HoyoisEquiv}, there are $G$-equivariant cofiber sequences of the form
	\[
	\Th(V|_U) \longrightarrow \Th(V) \longrightarrow \Th(\nu_{Z/X} \oplus V|_Z).
	\]
	Using these fiber sequences for an induction as well as the fact that $O^G(S^{2n,n})$ is closed under colimits and cofiber extensions by definition we see that it will be enough to do the following: We may supply a stratification of $X$ with smooth strata $Z_i$ such that $V|_{Z_i}$ locally $G$-equivariantly splits off a trivial vector bundle of rank $n$.

	The stratification we use is the Luna stratification.
	The strata $Z_i = X(H)$ are smooth as needed.
	Note that $X(H) \weq X(H)^H \times^{N} G$, where $N \subset G$ is the normalizer of $H$.
	Thus in order for $V|_{X(H)}$ to locally split off a trivial $G$-bundle of rank $n$, it suffices that $V|_{X(H)^H}$ locally splits off a trivial $N$-bundle of rank $n$.\NB{ref?}
	Set $V' = (V|_{X(H)^H})^H$.
	Since $H \subset N$ is normal and $H$ acts trivially on $X(H)^H$, $V'$ is an $N$-vector bundle, and in fact an $N/H$-vector bundle.
        It is also a summand of $V|_{X(H)^H}$ (since $|G| \in k^\times$).
	Now note that $N/H$ acts freely on $X(H)^H$, and so $N/H$-vector bundles on $X(H)^H$ are the same as ordinary vector bundles on $X(H)^H/(N/H)$\NB{ref?}.
	In particular $V'$ is locally free.
	A trivial $H$-vector subbundle of $V|_{X(H)^H}$ has trivial $H$-action ($H$ acting freely on $X(H)^H$) and consequently lies in $V'$.
	By assumption such a subbundle of rank $n$ exists; hence $V'$ has rank $\ge n$.
	Thus $V'$ is the required locally free summand.
\end{proof}

\begin{theorem}
	\label{thm:cellularityofHZXbeforequotient}
	If $k$ has characteristic $0$, and $X$ is a smooth $k$-scheme, then  
	\[
	\Th(p^* \bar{\rho}_r^{\oplus n}) \wedge (X^{\times r})_+ \in O^{\Sigma_r}(S^{2n,n}).
	\]
\end{theorem}
\begin{proof}
  	It suffices to establish that $\Th(p^* \bar{\rho_r}^{\oplus n}) \in O^{\Sigma_r}(S^{2n,n})$.  Consider the stratification of ${\mathbb A}[\bar{\rho}_r]$ by orbit type.  The assumption that $k$ has characteristic $0$ guarantees that $|\Sigma_r|$ is invertible in $k$ for all $r$.  By appeal to Lemma~\ref{lem:splittingbundlebystabilizertype}, to establish the result, it suffices to show that the restriction of $p^*\bar{\rho}_r$ to the locally closed subscheme of ${\mathbb A}[\bar{\rho}_r] \setminus 0$ described by each stabilizer type equivariantly splits off a trivial bundle of rank $1$.

  	Now if $H$ is a subgroup of $\Sigma_r$ occurring as the stabilizer of a point, then $H$ is a partition subgroup, i.e., it is isomorphic to $\prod_{i=1}^d \Sigma_{n_i}$ where $\sum_i n_i = r$.  The stabilizers of the $\Sigma_r$ action on ${\mathbb A}[\bar{\rho}_r] \setminus 0$ are always {\em proper}.  Indeed, it suffices to observe that the fixed point locus is empty, and this can be checked after base change to an algebraic closure of $k$, in which case it is clear.  

	Thus, suppose $H$ is a proper partition subgroup of $\Sigma_r$.  It suffices to prove that $\op{Res}^{\Sigma_r}_H \bar{\rho}_r$ splits off a trivial bundle. We have $H =  \Sigma_{n_1} \times \cdots \times \Sigma_{n_d}$ and $\op{Res}^{\Sigma_r}_H(\rho_r) \weq \rho_{n_1} \oplus \cdots \oplus \rho_{n_d}$. Since $d \ge 2$ ($H$ being a proper subgroup) we find that $\op{Res}^{\Sigma_r}_H(\rho_r)$ splits off at least two trivial representations, whence $\op{Res}^{\Sigma_r}_H \bar{\rho}_r$ splits off at least one.
\end{proof}

\subsection{Weak cellularity and quotients}
\label{ss:cellularityemchar0}
We denote by 
\[ 
R: \mathrm{P}(\Sch_k^{ft}) \longrightarrow \mathrm{P}(\Sm_k) 
\] 
the restriction functor along the inclusion $\Sm_k \to \Sch_k^{ft}$.
\begin{lem} \label{lem:R-pres-stuff}
The functor $R: \mathrm{P}(\Sch_k^{ft})_* \to \mathrm{P}(\Sm_k)_*$ satisfies $R(\mathscr A \wedge \mathscr B) \weq R(\mathscr A) \wedge R(\mathscr B)$ and, for $X \in \Sm_{k*}$, $R(X) \weq X$.
Moreover $R$ preserves colimits.

The same holds for the variants on $\mathrm{P}_\Sigma, \Shv_{\Nis}, \Spc$.
\end{lem}
\begin{proof}
The first statement is clear since the smash product is defined sectionwise on both categories.
The second statement follows from fully faithfulness of the left adjoint of $R$.
The third statement holds because colimits are computed sectionwise on both categories.

The variants follow because $R$ preserves the relevant equivalences.
\end{proof}

\begin{proposition}
	\label{prop:quotientspreservecellularity}
	If $\mathscr{X} \in \Spc^G(k)_*$ lies in $O^G(S^{p,q})$, then $R\op{Quot}_G\mathscr{X} \in O(S^{p,q})$ as well.
\end{proposition}
\begin{proof}
Since $O(S^{p,q})$ is closed under colimits and cofiber extensions and $R\op{Quot}_G$ preserves colimits, cofiber extensions and $\Sigma^{p,q}$ (by Lemma \ref{lem:R-pres-stuff}), for this it suffices to check that if $\mathscr Y \in \Spc(k)_*$ then $\Sigma^{p,q} \mathscr Y \in O(S^{p,q})$.
This is clear.
\end{proof}

\begin{theorem}
	\label{thm:cofiberofmotivicEMassemblycellularityestimate}
	Assume $k$ has characteristic $0$.  If $E \in \Mod^{\veff}_{H\Z}(k)$, then writing $E^n := \Omega^{\infty}\Sigma^{2n,n}E$, we have
	\[
	\cof(\Sigma^{2,1} E^n \to E^{n+1}) \in O(S^{4n+3,2n+1}).
	\]
\end{theorem}

\begin{proof}
	We know that $\Mod^{\veff}_{H\Z}(k)$ is generated under sifted colimits by spectra of the form $H\Z \wedge \Sigma^{\infty} X_+$ (see Proposition \ref{prop:generationundercolimits}).
	The functor $\Omega^\infty$ preserves these colimits.\NB{ref?}
	Since weakly cellular classes are closed under colimits (see Lemma \ref{lem:leftbousfieldcolimits}), it suffices to prove the result for $E = H\Z \wedge \Sigma^{\infty} X_+$. Set $C_r = \cof(\Sigma^{2,1} \Sym^r \Sigma^{2n,n} X_+ \to \Sym^r \Sigma ^{2n+2,n+1} X_+)$. Proposition~\ref{prop:cofiberofassemblyHZsmashX} together with \S\ref{par:reducedsymmpowers} tells us that
	\[
	\cof(C_{r-1} \to C_r)  \weq \Sigma^{2n+3,n+1} \op{Quot}_{\Sigma_r} (\Th(p^* \bar{\rho}_r^{\oplus n}) \wedge (X^{\times r})_+),
	\]
	while Theorem~\ref{thm:cellularityofHZXbeforequotient} tells us that
	\[
	\Sigma^{2n+3,n+1} (\Th(p^* \bar{\rho}_r^{\oplus n}) \wedge (X^{\times r})_+) \in O^{\Sigma_r}(S^{4n+3,2n+1}).
	\]
	In that case, Proposition~\ref{prop:quotientspreservecellularity} tells us that weak cellularity is preserved by taking quotients, so we see that \[ \cof(C_{r-1} \to C_r) \in O(S^{4n+3,2n+1}). \]
	Since also $C_1 = *$ and weakly cellular classes are stable under cofiber extensions, we deduce that $C_r \in O(S^{4n+3,2n+1})$.
	Taking the colimit over $r$ shows that \[ \cof(\Sigma^{2,1} E^n \to E^{n+1}) \in O(S^{4n+3,2n+1}), \] as needed.
\end{proof}

\subsection{Modifications in positive characteristic}
\label{ss:cellularityemcharp}
Dealing with cellularity of quotients in positive characteristic is slightly more subtle. When analyzing symmetric products, we have to deal with non-tame group actions, since the characteristic exponent of the base field will divide the order of $\Sigma_r$ for $r$ sufficiently large.  For this reason, we will work directly in Voevodsky's category of motives appealing to the results of \cite{Kellyldh}.

\subsubsection*{Assembly map}
The standard representation $\rho_r$ of $\Sigma_r$ always contains the trivial subrepresentation $\mathbf{1} \subset \rho_r$, consisting of vectors with equal coordinates.
If $r \ge p$ this no longer splits.
However, it does so if we restrict to a subgroup $G \subset \Sigma_r$ with $|G| \in k^\times$.
We write $\bar\rho_r$ for a complementary summand.
\begin{proposition}
	\label{prop:cofiberofassemblyHZsmashX-charp}
	Let $G \subset \Sigma_r$ be of order coprime to $p$.
	If $X \in \Sm_k$, then there is a canonical identification
	\begin{gather*} 
	\cof(\Sigma^{2,1} \op{Quot}_G (\Sigma^{2n,n} X_+)^{\wedge r} \to \op{Quot}_G (\Sigma^{2n+2,n+1} X_+)^{\wedge r}) 
	\\ \weq \Sigma^{2n+3,n+1} \op{Quot}_{G} (\Th(p^* (\op{Res}^{\Sigma_r}_{G}\bar{\rho}_r)^{\oplus n}) \wedge (X^{\times r})_+).
	\end{gather*}
\end{proposition}
\begin{proof}
The proof is essentially the same as for Proposition \ref{prop:cofiberofassemblyHZsmashX}.
\end{proof}

\subsubsection*{Motives of singular schemes in positive characteristic}
Assume $k$ is a field that has characteristic exponent $p > 0$ and $\ell$ is a prime different from $p$.
Write $\DM^{\ell dh}_{ft}(k, \Z_{(\ell)})$ for the variant of Voevodsky's category of motives built out of finite type schemes and the $\ell dh$-topology.
There is an evident functor 
\[ 
\Mod_{H\Z_{(\ell)}}(k) \weq \DM(k, \Z_{(\ell)}) \longrightarrow \DM_{ft}^{\ell dh}(k, \Z_{(\ell)}) 
\] 
whose right adjoint we denote by $r_{H\Z_{(\ell)}}$.
In fact this is an adjoint equivalence \cite[Theorem 4.0.13]{Kellyldh}.

\begin{theorem}
There is a commutative square of the form
\[
\xymatrix{
\Spc^{ft}_{Nis}(k) \ar[r]^{R}\ar[d]^{\underline{M}} & \ho{k} \ar[d]^{M}\\
\DM_{ft}^{\ell dh}(k, \Z_{(\ell)}) \ar[r]^{r_{H\Z_{(\ell)}}} & \Mod_{H\Z_{(\ell)}}(k).
}
\]
\end{theorem}

\begin{proof}
	This is \cite[Theorem 3.1]{HKO} in view of \cite[Proposition 4.10]{HKO} and \cite[Corollary 4.27]{HKO}.  \todo{Check published references.}
\end{proof}

\subsubsection*{Transferring}
Assume $k$ is a field that has characteristic exponent $p > 0$.  If $G$ is a finite group, then $G$ is linearly reductive if and only if $|G|$ is coprime to $k$.  In Theorem~\ref{thm:cellularityofHZXbeforequotient}, we needed to consider symmetric groups $\Sigma_r$ for arbitrary $r$, in particular $r$ might be $\ge p$ and therefore the subgroups $H \subset \Sigma_r$ we consider need not be linearly reductive.  

We now discuss how to avoid this issue, at the expense of inverting the prime $p$, or localizing at a prime $\ell$ different from $p$ via a standard transfer argument.  In practice, we will use the fact that if $G$ is a finite group and $H$ is an $\ell$-Sylow subgroup, then $G/H$ has order coprime to $\ell$.  If $X \in \Sm_k^G$, then there is an evident map $\op{Res}^G_H(X)/H \to X/G$ that is finite and flat morphism of degree $|G/H|$ and thus coprime to $\ell$.  This map induces a covariant transfer map of motives and pullback followed by transfer is multiplication by $|G/H|$.  The following lemma is a precise version of this observation.

\begin{lem}
	\label{lem:transferargument}
	Assume $k$ is a field having characteristic exponent $p$ and let $\ell$ be a prime different from $p$.  Let $G$ be a finite group and $H \subset G$ a finite subgroup of index coprime to $\ell$ and suppose $X \in \Sm_k^G$.  The object $\underline{M}(X/G)$ is a retract of $\underline{M}(\op{Res}^G_H(X)/H)$ in $\DM^{\ell dh}_{ft}(k, \Z_{(\ell)})$.
\end{lem}

\begin{proof}
	This is a consequence of \cite[Proposition A.8(1)]{VoeMEM}: loosely speaking pullback followed by pushforward is multiplication by $|G/H|$ which is invertible in $\Z_{(\ell)}$.
\end{proof}

\subsubsection*{Conclusion}
We now provide a variant of Theorem~\ref{thm:cellularityofHZXbeforequotient} in positive characteristic.  If $G$ is a finite group, then by definition, the $\ell$-Sylow subgroup $Syl_{\ell}(G) \subset G$ has index coprime to $\ell$ and has order invertible in $k$.  As in Section~\ref{ss:cellularityassemblymap} consider the standard representation $\rho_r$ of $\Sigma_r$.  Abusing notation, we write $p$ for the $Syl_{\ell,r}$-equivariant structure map ${\mathbb A}[\op{Res}^{\Sigma_r}_{Syl_{\ell,r}}(\rho_r)] \setminus 0$.  

\begin{theorem}
	\label{thm:cellularityofHZbeforequotientposchar}
	Assume $k$ has characteristic exponent $p > 0$, and $\ell$ is a prime different from $p$ and $X$ is a smooth $k$-scheme.  If $Syl_{\ell,r} \subset \Sigma_r$ is an $\ell$-Sylow subgroup, then 
	\[
	\Sigma^{2n+3,n+1} (\Th(p^* (\op{Res}^{\Sigma_r}_{Syl_{\ell,r}}\bar{\rho}_r)^{\oplus n}) \wedge (X^{\times r})_+),
	\]
	lies in $O^{Syl_{\ell,r}}(S^{4n+3,2n+1})$.
\end{theorem}

\begin{proof}
	Consider the stratification of ${\mathbb A}[\op{Res}^{\Sigma_r}_{Syl_{\ell,r}}(\rho_r)] \setminus 0$ by stabilizer type from Proposition \ref{prop:stabilzerstratification} for the action of $Syl_{\ell,r}$.  In this case, the stabilizers are intersections of the form $Syl_{\ell,r} \cap H$ where $H$ is a \emph{proper} partition subgroup of $\Sigma_r$ (indeed $\mathbf{1} \subset \rho_r$ is the set of $\Sigma_r$-fixed points, and hence $\bar\rho_r^{\Sigma_r} = \{0\}$).  Since the order of $Syl_{\ell,r}$ is invertible in $p$, it follows that the orders of each of the groups $Syl_{\ell,r} \cap H$ are also invertible in $k$, in particular, all these groups are linearly reductive.  
	
	By appeal to Lemma~\ref{lem:splittingbundlebystabilizertype}, it suffices to show that for group of the form $Syl_{\ell,r} \cap H$ the representation $\op{Res}^{\Sigma_r}_{Syl_{\ell,r} \cap H}\bar{\rho_r}$ splits off a trivial summand.  This follows in exactly the same way as at the end of Theorem~\ref{thm:cellularityofHZXbeforequotient}.
\end{proof} 

Finally, we may establish a version of Theorem~\ref{thm:cofiberofmotivicEMassemblycellularityestimate} in positive characteristic.

\begin{theorem}
	\label{thm:cofiberofmotivicEMassemblycellularityestimatecharp}
	Assume $k$ is a field having characteristic exponent $p > 1$. Let $n>0$ and $E \in \Mod^{\veff}_{H\Z[1/p]}(k)$. For $E_n := \Omega^{\infty}\Sigma^{2n,n}E$ we have
	\[
	\cof(\Sigma^{2,1} E_n \to E_{n+1}) \in O(S^{4n+3,2n+1}).
	\]
\end{theorem}

\begin{proof}
	As before we may reduce to $E = H\Z[1/p] \wedge \Sigma^{\infty} X_+$ for some $X \in \Sm_k$. By continuity, $X$ is pulled back from a perfect subfield of finite cohomological dimension (the perfect closure of a finitely generated extension of $\mathbb F_p$). By essentially smooth base change (Lemma \ref{lem:cellularity-ess-sm-bc}) we may thus assume that $k$ has finite cohomological dimension. Since $\Omega^\infty E_n \in O(S^{2n,n})$ by Proposition \ref{prop:deloopingpreservesconnectivity}, we see that 
	\[ 
	Y := \cof(\Sigma^{2,1} E_n \to E_{n+1}) \in O(S^{2n+2,n+1}) \subset O(S^{4,2}). 
	\]
	Combining this with Proposition \ref{prop:pdivisibility-unstable-implies-stable} we find that also 
	\[ 
	\Sigma^\infty_{S^1} Y \in O(S^{4,2})[1/p]. 
	\]
	\TODO{We seem not to have introduced the notation for $p$-inverted cellular classes.}
	We may thus apply the weak-cellular motivic Whitehead theorem \ref{thm:pqequivsof1connectedspacesaredetectedstably}(4) to reduce to showing that $\Sigma^\infty Y \wedge H\Z \in O(S^{4n+3,2n+1})$.
	By standard arguments it suffices to prove the same for $\Sigma^\infty Y \wedge H\Z_{(\ell)}$ for all $\ell \ne p$.
	This follows from Lemma~\ref{lem:transferargument} in conjunction with Theorem \ref{thm:cellularityofHZbeforequotientposchar}, Proposition \ref{prop:cofiberofassemblyHZsmashX-charp} and Theorem \ref{thm:motivicDoldThom}.\todo{more details?}
\end{proof}

\section{Weak cellular estimates of the fiber of the unit map}
\label{s:maintheorem}
In this section, we put everything together to deduce the motivic Freudenthal suspension theorem.  For a pointed connected motivic space $\mathscr{X}$, write
\[
s_{\mathscr{X}}: \mathscr{X} \longrightarrow \Omega^{2,1}\Sigma^{2,1}\mathscr{X};
\]
for the stabilization map, i.e., the unit map of the loop-suspension adjunction.  Our goal is to, with suitable assumptions on the weak cellular class of $\mathscr{X}$, provide an estimate for the weak cellular class of $\fib(s_{\mathscr{X}})$.  We proceed by a series of reductions.

\subsection{The unit of stabilization and fiber sequences}
In this section, we analyze how the weak cellular class of the fiber of the stabilization map behaves in fiber sequences.  Proposition~\ref{prop:fibersofstabilization} provides fiber sequences linking the fiber of stabilization in the base, fiber and total space of a fiber sequence.  Then Proposition~\ref{prop:fiberconnectivityinduction} yields the main inductive tool we use to analyze the fiber of stabilization.  

\begin{proposition}[Fibers]
	\label{prop:fibersofstabilization}
Suppose we have a fiber sequence $\mathscr{F} \to \mathscr{E} \to \mathscr{B}$ of pointed, connected motivic spaces.   Set $\mathscr{F'} := \fib(\Sigma^{2,1} \mathscr{E} \to \Sigma^{2,1} \mathscr{B})$.
\begin{enumerate}[noitemsep,topsep=1pt]
	\item There is an induced factorization
	\[
	\mathscr{F} \longrightarrow \Omega^{2,1}\Sigma^{2,1} \mathscr{F} \longrightarrow \Omega^{2,1}\mathscr{F}',
	\]
	where the first morphism is the unit map.  
	\item There is an induced fiber sequence
	\[
	\fib(\mathscr{F} \to \Omega^{2,1}\mathscr{F}') \longrightarrow \fib(s_{\mathscr{E}}) \longrightarrow \fib(s_{\mathscr{B}})
	\]
	\item There is an induced fiber sequence 
	\[
	\fib(s_{\mathscr{F}}) \longrightarrow \fib(\mathscr{F} \to \Omega^{2,1}\mathscr{F}') \longrightarrow \Omega^{2,1}\fib(\Sigma^{2,1}\mathscr{F} \to \mathscr{F}')
	\]
\end{enumerate}
\end{proposition}

\begin{proof}
	Since the composite map $\mathscr{F} \to \mathscr{B}$ is null, so is the map $\Sigma^{2,1}\mathscr{F} \to \Sigma^{2,1}\mathscr{B}$.  It follows that there is a map $\Sigma^{2,1} \mathscr{F} \to \mathscr{F}'$, which by looping yields the map $\Omega^{2,1} \Sigma^{2,1}\mathscr{F} \to \Omega^{2,1}\mathscr{F}'$ and by adjunction yields the map $\mathscr{F} \to \Omega^{2,1}\mathscr{F}'$ factoring as claimed.
	
	Consider the commutative diagram:
	\[
	\xymatrix{
		\mathscr{E} \ar[r]\ar[d] & \mathscr{B} \ar[d] \\
		\Omega^{2,1}\Sigma^{2,1} \mathscr{E} \ar[r] & \Omega^{2,1} \Sigma^{2,1} \mathscr{B},
	}
	\]
	where the vertical maps are unit maps.  The fiber sequence of the second point arises by comparing the vertical and horizontal homotopy fibers in this diagram.
	
	The third fiber sequence arises by taking the fiber of the composite via \cite[Proposition 3.1]{ABHWhitehead} for the factorization in the first point.
\end{proof}

\begin{proposition}
	\label{prop:fiberconnectivityinduction}
	Suppose $\mathscr{F} \stackrel{\iota}{\to} \mathscr{E} \stackrel{f}{\to} \mathscr{B}$ is a fiber sequence of pointed motivic spaces, lying in $O(S^{m,n})$, $m-n \geq 2, n \geq 1$.   Set $\mathscr{F'} := \fib(\Sigma^{2,1} \mathscr{E} \to \Sigma^{2,1} \mathscr{B})$.
	\begin{enumerate}[noitemsep,topsep=1pt]
		\item We have $\fib(\Sigma^{2,1}\mathscr{F} \to \mathscr{F}') \in O(S^{2m+1,2n+1})$.
		\item If $\fib(s_{\mathscr{F}}) \in O(S^{2m-1,2n})$, then $\fib(s_{\mathscr{E}}) \in O(S^{2m-1,2n})$ if and only if $\fib(s_{\mathscr{B}}) \in O(S^{2m-1,2n})$.
		\item If $\fib(s_{\mathscr{B}}) \in O(S^{2m-1,2n})$, then $\fib(\fib(s_{\mathscr{F}}) \to \fib(s_{\mathscr{E}})) \in O(S^{2m-2,2n})$
	\end{enumerate}
\end{proposition}

\begin{proof}
	\noindent {{\bf Point 1}.} To deduce the stated weak cellular estimate on  $\fib(\Sigma^{2,1}\mathscr{F} \to \mathscr{F}')$ it suffices by Corollary \ref{cor:unstablecellularityoffibersvscofibers}(2), whose hypotheses are satisfied by the assumption of the statement, to show that $\cof(\Sigma^{2,1} \mathscr{F} \to \mathscr{F}') \in O(S^{2m+2,2n+1})$.  There is a commutative square of the form
	\[
	\xymatrix{
	\Sigma^{3,1} \mathscr{F} \ar[r]\ar[d] & \Sigma \mathscr{F}' \ar[d] \\
	\Sigma^{2,1} \cof(f) \ar[r] & \cof(\Sigma^{2,1}f),
	}
	\]
	where the left vertical map is the suspension of the map $\Sigma \mathscr{F} \to \cof(f)$, the right vertical map is the map $\Sigma \mathscr{F}' \to \cof(\Sigma^{2,1}f)$ and the top horizontal map is the suspension of the map $\Sigma^{2,1}\mathscr{F} \to \mathscr{F}'$ from the statement.  The bottom horizontal map is an equivalence, while the cofiber of the top map is $\Sigma \cof(\Sigma^{2,1} \mathscr{F} \to \mathscr{F}')$.  By the weak-cellular Whitehead theorem \ref{thm:pqequivsof1connectedspacesaredetectedstably}, it is enough to show that $\Sigma^2 \cof(\Sigma^{2,1} \mathscr{F} \to \mathscr{F}') \in O(S^{2m+4,2n+1})$.  
	
	Comparing vertical and horizontal cofibers arising from the above commutative diagram we obtain a cofiber sequence of the form:
	\[
	\cof(\Sigma^{3,1}\mathscr{F} \to \Sigma^{2,1}\cof(f)) \longrightarrow \cof(\Sigma \mathscr{F}' \to \cof(\Sigma^{2,1}f)) \longrightarrow \Sigma^2 \cof(\Sigma^{2,1} \mathscr{F} \to \mathscr{F}').
	\]
	By weak-cellular motivic Blakers--Massey \ref{thm:cofiberpqconnectivity}(2), since $\mathscr{F} \in O(S^{m,n})$ and $\mathscr{B} \in O(S^{m,n})$, we conclude that $\cof(\Sigma \mathscr{F} \to \cof(f)) \in O(S^{2m+1,2n})$ and thus that $\cof(\Sigma^{3,1}\mathscr{F} \to \Sigma^{2,1}\cof(f)) \in O(S^{2m+3,2n+1})$.  Likewise, since $\Sigma^{2,1}\mathscr{B} \in O(S^{m+2,n+1})$ and $\mathscr{F}' \in O(S^{m+2,n+1})$, we conclude that $\cof(\Sigma \mathscr{F}' \to \cof(\Sigma^{2,1}f)) \in O(S^{2m+5,2n+2})$.  In that case, the fiber of $\cof(\Sigma^{3,1}\mathscr{F} \to \Sigma^{2,1}\cof(f)) \to \cof(\Sigma \mathscr{F}' \to \cof(\Sigma^{2,1}f))$ then lies in $O(S^{2m+3,2n+1})$ as well by appeal to Theorem~\ref{thm:cellularityoffibers}.  Yet another application of weak-cellular motivic Blakers--Massey then implies that $\Sigma^2 \cof(\Sigma^{2,1} \mathscr{F} \to \mathscr{F}') \in O(S^{2m+4,2n+1})$ as required.
	
	
	\noindent {\bf Point 2.}  By Proposition~\ref{prop:fibersofstabilization}(3) there is a fiber sequence
	\[
	\fib(s_{\mathscr{F}}) \longrightarrow \fib(\mathscr{F} \to \Omega^{2,1}\mathscr{F}') \longrightarrow \Omega^{2,1} \fib(\Sigma^{2,1}\mathscr{F} \to \mathscr{F}').
	\]
	By the conclusion of the first point in conjunction with Proposition~\ref{prop:unstableconnectivityofloopspaces}, we conclude that the last term lies in $O(S^{2m-1,2n})$.  If $\fib(s_{\mathscr{F}}) \in O(S^{2m-1,2n})$, then Proposition~\ref{prop:pqconnectivityandfibersequencesI} tells us that $\fib(\mathscr{F} \to \Omega^{2,1}\mathscr{F}') \in O(S^{2m-1,2n})$ as well.  
	
	Next, there is a fiber sequence
	\[
	\fib(\mathscr{F} \to \Omega^{2,1}\mathscr{F}') \longrightarrow \fib(s_{\mathscr{E}}) \longrightarrow \fib(s_{\mathscr{B}}).
	\]
	by Proposition~\ref{prop:fibersofstabilization}(2).  In that case, another application of Proposition~\ref{prop:pqconnectivityandfibersequencesI} in conjunction with the estimate of the preceding paragraph implies that $\fib(s_{\mathscr{B}}) \in O(S^{2m-1,2n})$, if and only if $\fib(s_{\mathscr{E}}) \in O(S^{2m-1,2n})$. \newline
	
	\noindent{\bf Point 3.}  Proposition~\ref{prop:fibersofstabilization} shows that the map $\fib(s_{\mathscr{F}}) \to \fib(s_{\mathscr{E}})$ factors as 
	\[
	\fib(s_{\mathscr{F}}) \longrightarrow \fib(\mathscr{F} \to \Omega^{2,1}\mathscr{F}') \longrightarrow \fib(s_{\mathscr{E}}).
	\]
	The proposition also shows that the fiber of the first map is $\Omega^{3,1} \fib(\Sigma^{2,1}\mathscr{F} \to \mathscr{F}')$, while the fiber of the second map is $\Omega \fib(s_{\mathscr{B}})$.  Assuming $\fib(s_{\mathscr{B}}) \in O(S^{2m-1,2n})$, we conclude that $\Omega \fib(s_{\mathscr{B}}) \in O(S^{2m-2,2n})$.  The first point yields $\Omega^{3,1}\fib(\Sigma^{2,1}\mathscr{F} \to \mathscr{F}') \in O(S^{2m-2,2n})$.  In that case Corollary~\ref{cor:cellularityofcompositesoffibers} implies that $\fib(\fib(s_{\mathscr{F}}) \to \fib(s_{\mathscr{E}})) \in O(S^{2m-2,2n})$ as well.
\end{proof}

\subsection{The main result}
\label{ss:freudenthalChowdiagonal}
	In this section, we combine all the tools developed so far to analyze $\pone$-suspension.  If $E$ is any $\pone$-spectrum, then set 
	\[
	\underline{E}^{a,b} := \Omega^{\infty}\Sigma^{a,b} E.
	\]  
	We may now state the main result.  

\begin{theorem}[The $\pone$-stabilization theorem]
	\label{thm:freudenthal}
	Assume $k$ is a perfect field with characteristic exponent $p$, $\mathscr{X}$ is a pointed motivic space and $s_{\mathscr{X}}: \mathscr{X} \to \Omega^{2,1}\Sigma^{2,1}\mathscr{X}$ is the unit of the loop suspension adjunction. Assume $m$ and $n$ are integers with $m \leq 2n$, $m-n \geq 2$, and $n \geq 2$.  If $\mathscr{X} \in O(S^{m,n})$ and $\bpi_*(\mathscr{X})$ is uniquely $p$-divisible, then $\fib(s_{\mathscr{X}}) \in O(S^{2m-1,2n})$. 
\end{theorem}

Before delving into the proof, we outline the strategy.  We quickly reduce to the case $m = 2n$ and by appeal to the weakly cellular Whitehead tower, and the inductive machinery described in the previous section, we reduce the analysis of the weak cellular class of the fiber of the unit map $s_{\mathscr{X}}$ to the case of $\pone$-infinite loop spaces; essentially all the effort is focused on treating this latter case.  

Intuitively, we aim to analyze the fiber of the stabilization map for $\pone$-infinite loop spaces by ``fracturing" with respect to the motivic Hopf map $\eta$ (recall the conventions of Paragraph~\ref{par:milnorwittktheory}).  In particular, we would like to separately analyze $\eta$-complete and $\eta$-local $\pone$-infinite loop spaces.  In practice, the geometric analysis of the cofiber of the assembly map for Eilenberg--Mac Lane spaces will give us control over the $\eta$-complete situation.  

There is an equality of the form $1 - \eta \rho = 1 + \eta[-1] = \langle -1 \rangle$ and therefore, $1 + \varepsilon = \eta \rho$ in $\pi_{0,0}(\1_k)$.  Building on this observation, the $\eta$-local regime can be accessed, if we additionally invert $2$, by studying the $\rho$-localization, which in turn can be analyzed by means of real-\'etale topological techniques.  To build a suitable fracture diagram, we need to study $\eta$ and $2$-torsion, which we do by appeal to a modification of Levine's slice convergence argument.  There are a number of ways to perform this kind of ``fracture" in practice: here we analyze the torsion and rational parts separately, but other configurations are also possible.  

\begin{proof}
	We proceed by a string of reductions.  \newline\newline
	\noindent {\bf Step 1.} (Reduction to the Chow diagonal) We may assume that $m = 2n$.  Indeed, if $\mathscr{X} \in O(S^{m,n})$, then $\Sigma^{2n-m}\mathscr{X} \in O(S^{2n,n})$.  Since $\mathscr{X}$ is $1$-connected by assumption, so is $\Sigma^{2n-m}\mathscr{X}$.  If $\bpi_*(\mathscr{X})$ is uniquely $p$-divisible, then $\Sigma^{\infty}_{S^1}\mathscr{X}$ is as well, and thus so is $\Sigma^{\infty}_{S^1} \Sigma^{2n-m}\mathscr{X}$.  Therefore, Proposition~\ref{prop:pdivisibility} implies that $\bpi_*(\Sigma^{2n-m}\mathscr{X})$ is uniquely $p$-divisible as well.  In that case, assuming $\fib(s_{\Sigma^{2n-m}\mathscr{X}}) \in O(S^{4n-1,2n})$, appealing inductively to Lemma~\ref{lem:reductionto2nn}, we conclude that $\fib(s_{\mathscr{X}}) \in O(S^{2m-1,2n})$ as required. \newline
	
	\noindent {\bf Step 2.} (Reduction to $\pone$-infinite loop spaces) To establish the result for spaces, it suffices to establish the result for $\pone$-infinite loop spaces; this reduction is the content of Lemma~\ref{lem:reductiontothestablecase} (using Remark \ref{rmk:modified-layers}).  In more detail, under the hypotheses, to show that $\fib(s_{\mathscr{X}}) \in O(S^{4n-1,2n})$, it suffices to establish that for $E$ is $\SH(k)^{\veff}[\frac{1}{p}]$ we have $\fib(s_{\underline{E}^{2n,n}}) \in O(S^{4n-1,2n})$.  Moreover, Proposition~\ref{prop:fiberconnectivityinduction}(2) implies that the collection of $E \in \SH(k)^{\veff}[1/p]$ satisfying this estimate on the weak-cellular class of the fiber of stabilization is stable under extensions and cofibers; it is also stable under filtered colimits.  Thus, it even suffices to analyze $\fib(s_{\underline{E}^{2n,n}})$ with $E = \Sigma^{\infty}_+ X[1/p]$ and $X \in \Sm_k$ (see Proposition \ref{prop:generationundercolimits}).  \newline
	
	\noindent {\bf Step 3.} (Reduction to the case of finite vcd) We may assume that $k$ has finite vcd.  Indeed, since $X \in \Sm_k$ has finite-presentation, it is necessarily pulled back from the perfect closure of some finitely generated extension of the prime field.  Since cellularity estimates are stable under essentially smooth base-change (Lemma \ref{lem:cellularity-ess-sm-bc}), we conclude.  \newline
	
	\noindent {\bf Step 4.} (Reduction to slices and effective covers) It suffices to establish the weak cellular estimate on the fiber of the stabilization map for the effective cover $f_{N}E$ for $N$ sufficiently large and $s_iE$ for arbitrary $i$; see Step 5 for the latter, and proceed to Step 7 for the former.  Indeed, there are fiber sequences of the form
	\[
	f_{i+1}E \longrightarrow f_{i}E \longrightarrow s_iE,
	\]  
	and thus corresponding fiber sequences after applying $\Omega^{\infty}$.  In that case, the assumption that $\fib(s_{\Omega^{\infty}f_{N}E}) \in O(S^{4n-1,2n})$ for $N$ large and $\fib(s_{\Omega^{\infty}s_iE})$ for all $i \geq 0$ implies that $\fib(s_{\Omega^{\infty}f_iE}) \in O(S^{4n-1,2n})$ by repeated appeal to Proposition~\ref{prop:fiberconnectivityinduction}(2).  In that case, the required weakly cellular estimate follows by descending induction.\newline
	
	\noindent {\bf Step 5.} (Reducing from slices to very effective motives) If $E \in \SH(k)^{\veff}$, then $s_i E$ remains very effective and admits the structure of an $H\Z[1/p]$-module (see Paragraph~\ref{par:slices} for more details and references).\newline
	
	\noindent {\bf Step 6.} (Treating effective motives) The weakly cellular estimate on the fiber of the stabilization map holds for $E \in \Mod^{\veff}_{H\Z[1/p]}(k)$.  Indeed, in that case, the relevant weak cellular estimate is contained in Lemma~\ref{lem:caseofveryeffectivemotives}. \newline

	\noindent {\bf Step 7.} (Treating the rational plus part of effective covers) To establish the result for $f_NE$ as above, we proceed in two steps. First, we treat the rationalization $f_NE_\Q$; this breaks up into $\pm$ pieces under the decomposition of $\SH(k)_{\Q}$ into eigenspaces of $\varepsilon$ (see the conventions of Paragraph~\ref{par:varepsilonconventions}).  The $+$-part is described in terms of Voevodsky motives as detailed in Theorem~\ref{thm:rationalizedpluspart}.  Lemma~\ref{lem:effectivecoverspluspart} implies (recall that we have already reduced to $E = \Sigma^\infty_+X$, $X \in \Sm_k$) that $(f_NE_\Q)^+ = 0$, so we conclude that $f_NE_{\Q} \weq (f_NE_{\Q})^{-}$ (for $N \gg 0$). \newline

	\noindent {\bf Step 8.} (Reducing the rational minus part of effective covers to the real-\'etale case) The $-$ part of the decomposition of $\SH(k)_{\Q}$ was identified in terms of the derived category of real-\'etale sheaves of $\Q$-vector spaces (Theorem~\ref{thm:rationalminuspartviaret}).  In particular, we conclude $f_N E_\Q \weq (f_NE_\Q)^{-}$ has real-\'etale descent. \newline

	\noindent{\bf Step 9.} (Separating out the fields.) Let us set $\mathscr X_i = \Omega^\infty \Sigma^{2n,n} f_i E$. We want to prove that $\fib(s_{\mathscr X_i}) \in O(S^{4n-1,2n})$ for some $i$.  Taking $i>2n$ ensures\footnote{In fact, $i \ge n$ suffices.} that $\fib(s_{\mathscr X_i}) \in O(S^{2n+1,2n})$ and hence by Corollary \ref{cor:separatingcellularity} it will be enough to show that $\bpi_r(\fib(s_{\mathscr X_i})) = 0$ for $r \le 2n-2$. Since $s_i(E)$ has already been dealt with (step 5), we have $\fib(s_{\Omega^\infty \Sigma^{2n,n} s_i E}) \in O(S^{4n-1,2n})$. Now Proposition~\ref{prop:fiberconnectivityinduction}(3) shows that $\bpi_r(\fib(s_{\mathscr X_{i+1}})) \to \bpi_r(\fib(s_{\mathscr X_i}))$ is surjective for $r \le 2n-2$.
	Since homotopy sheaves are unramified, it hence suffices to prove: for every finitely generated extension $K/k$ there exists $i \gg 0$ with $\bpi_r(\fib(s_{\mathscr X_i}))(K) = 0$ for $r \le 2n-2$. \newline

	\noindent {\bf Step 10.} Consider the commutative diagram
\begin{equation*}
	\begin{CD}
	\mathscr X_i @>>> L_\ret \mathscr X_i \\
		@VVV @VVV \\
	\Omega^{2,1} \Sigma^{2,1} \mathscr X_i @>>> \Omega^{2,1} L_\ret \Sigma^{2,1} \mathscr X_i.
	\end{CD}
\end{equation*}
	Since $(\mathscr X_i)_\Q$ satisfies real étale descent (step 8), the same is true for $(\Sigma^{2,1} \mathscr X_i)_\Q$ (use that being a real étale sheaf is equivalent to being $\rho$-periodic by Proposition \ref{prop:unstable-rholoc}), and so Lemma \ref{lem:torsioncaseconnectivityestimate} shows that on sections over $K$, the horizontal maps are arbitrarily connective (for $i \gg 0$). It follows that \[ \fib(s_{\mathscr X_i}) \to \fib(L_\ret \mathscr X_i \to \Omega^{2,1} L_\ret \Sigma^{2,1} \mathscr X_i) \] is arbitrarily connective on sections over $K$ (for $i \gg 0$).
	We conclude since the right hand space lies in $O(S^{4n-1,2n})$ by Lemma \ref{lem:therholocalcase}.
\end{proof}

\begin{lem}
	\label{lem:reductionto2nn}
	Suppose $\mathscr{X} \in O(S^{m,n})$ with $m < 2n$ and $m-n \geq 2$ (and $n \ge 0$).  If $\fib(s_{\Sigma \mathscr{X}}) \in O(S^{2m+1,2n})$, then $\fib(s_{\mathscr{X}}) \in O(S^{2m-1,2n})$.
\end{lem}

\begin{proof}
	Consider the diagram
	\[
	\xymatrix{
		\mathscr{X} \ar[d]^{s_{\mathscr{X}}}\ar[r] &  \Omega \Sigma \mathscr{X} \ar[d]^{\Omega s_{\Sigma \mathscr{X}}}\\
		\Omega^{2,1}\Sigma^{2,1}\mathscr{X} \ar[r] & \Omega^{3,1}\Sigma^{3,1} \mathscr{X},
	}
	\]
	where the bottom horizontal map is $\Omega^{2,1}$ applied to the map $\Sigma^{2,1}\mathscr{X} \to \Omega \Sigma^{3,1}\mathscr{X}$.  Both composites are equivalent to the evident unit map, so the diagram commutes. 
	
	Since $\mathscr{X} \in O(S^{m,n})$, we conclude that $\Sigma^{2,1}\mathscr{X} \in O(S^{m+2,n+1})$.  Proposition~\ref{prop:refinedS1Freudenthal} tells us the fiber of top horizontal map lies in $O(S^{2m-1,2n})$ and that 
	\[
	\fib(\Sigma^{2,1}\mathscr{X} \longrightarrow \Omega \Sigma^{3,1} \mathscr{X}) \in O(S^{2m+3,2n+2}).
	\]
	Then, Proposition~\ref{prop:unstableconnectivityofloopspaces} tells us that $\Omega^{2,1}\fib(\Sigma^{2,1}\mathscr{X} \to \Omega \Sigma \mathscr{X}) \in O(S^{2m+1,2n+1})$.  On the other hand, the hypothesis is that $\fib(s_{\Sigma \mathscr{X}}) \in O(S^{2m+1,2n})$ and another application of Proposition~\ref{prop:unstableconnectivityofloopspaces} tells us that $\Omega \fib(s_{\Sigma \mathscr{X}}) \in O(S^{2m,2n})$.
	
	The fiber sequence attached to a composite of morphisms \cite[Proposition 3.1]{ABHWhitehead} yields:
	\[
	\fib(\mathscr{X} \to \Omega \Sigma \mathscr{X}) \longrightarrow \fib(\mathscr{X} \to \Omega^{3,1}\Sigma^{3,1}\mathscr{X}) \longrightarrow \Omega \fib(s_{\Sigma \mathscr{X}}).
	\]
	The cellularity estimates above in conjunction with Proposition~\ref{prop:pqconnectivityandfibersequencesI} tell us that $\fib(\mathscr{X} \to \Omega^{3,1}\Sigma^{3,1}\mathscr{X}) \in O(S^{2m-1,2n})$ as well.
	
	Another application of \cite[Proposition 3.1]{ABHWhitehead} yields a fiber sequence of the form
	\[
	\fib(s_{\mathscr{X}}) \longrightarrow \fib(\mathscr{X} \to \Omega^{3,1}\Sigma^{3,1}\mathscr{X}) \longrightarrow \Omega^{2,1}\fib(\Sigma^{2,1}\mathscr{X} \to \Omega \Sigma^{3,1}\mathscr{X}),
	\]
	which rotates to a fiber sequence of the form
	\[
	\Omega^{3,1}\fib(\Sigma^{2,1}\mathscr{X} \to \Omega \Sigma^{3,1}\mathscr{X}) \longrightarrow \fib(s_{\mathscr{X}}) \longrightarrow \fib(\mathscr{X} \to \Omega^{3,1}\Sigma^{3,1}\mathscr{X})
	\]
	We have $\Omega^{3,1}\fib(\Sigma^{2,1}\mathscr{X} \to \Omega \Sigma^{3,1}\mathscr{X}) \in O(S^{2m,2n+1})$ by appeal to Proposition~\ref{prop:unstableconnectivityofloopspaces}.  Another application of Proposition~\ref{prop:pqconnectivityandfibersequencesI} implies that $\fib(s_{\mathscr{X}}) \in O(S^{2m-1,2n})$ as required.
\end{proof}

\begin{lem}
	\label{lem:reductiontothestablecase}
	Suppose $\mathscr{X} \in O(S^{2n,n})$, $n \geq 2$.   The space $\fib(s_{\mathscr{X}}) \in O(S^{4n-1,2n})$ as soon as for $i \ge 0$ the spaces $\fib(s_{\tau_{\ge (n,n)} K(\bpi_i \mathscr X,i)}) \in O(S^{4n-1,2n})$.
\end{lem}

\begin{proof}
	Consider the Postnikov tower 
	\[ 
	\tau_{\ge (n,n)} K(\bpi_i \mathscr X,i) \longrightarrow \L^{2n+i+1,n} \mathscr X \longrightarrow \L^{2n+i,n} \mathscr X 
	\] 
	of Theorem \ref{thm:refined-Postnikov}.
	By induction using  Proposition~\ref{prop:fiberconnectivityinduction}(2) we find that $\fib(s_{\L^{2n+i,n} \mathscr X}) \in O(S^{4n-1,2n})$ for every $i$.
	Consider the fiber sequence $\mathscr F_i \to \mathscr X \to \L^{2n+i,n} \mathscr X$.
	In the language of Proposition \ref{prop:fibersofstabilization}, both $\mathscr F_i$ and $\mathscr F_{i}'$ are $i$-connective (use Lemma \ref{lem:Lpq-map-conn}), and hence $\fib(s_{\mathscr X}) \to \fib(s_{\L^{2n+i,n} \mathscr X})$ has $(i-2)$-connective fiber.
	The result now follows via Corollary \ref{cor:limit-of-Spq-equiv}.
\end{proof}

\begin{lem}
	\label{lem:caseofveryeffectivemotives}
	If $E \in \Mod^{\veff}(H\Z)[1/p]$ and $n \geq 0$, then 
	\[
	\fib(s_{\underline{E}^{2n,n}}) \in O(S^{4n-1,2n}).
	\]
	(If $p>1$, assume $n \ge 1$.)
\end{lem}

\begin{proof}
    Observe that there is a factorization of the identity map
	\[
	\underline{E}^{2n,n} \stackrel{s_{\underline{E}^{2n,n}}}{\longrightarrow} \Omega^{2,1}\Sigma^{2,1} \underline{E}^{2n,n} \stackrel{\Omega^{2,1}a_n}\longrightarrow \Omega^{2,1} \underline{E}^{2n+2,n+1} \weq \underline{E}^{2n,n} 
	\]
	where the first map is the unit map and the map $a_n$ is the assembly map for the spectrum.  By appeal to \cite[Proposition 3.1]{ABHWhitehead} we thus obtain a fiber sequence of the form
	\[
	\fib(s_{\underline{E}^{2n,n}}) \longrightarrow \ast \longrightarrow \Omega^{2,1}\fib(a_n),	
	\]
	i.e., an equivalence $\Omega^{3,1}\fib(a_n) \weq \fib(s_{\underline{E}^{2n,n}})$. 
	
	If $k$ has characteristic $0$, then Theorem~\ref{thm:cofiberofmotivicEMassemblycellularityestimate} yields a cellularity estimate for the cofiber of $a_n$. If $k$ has positive characteristic, Theorem~\ref{thm:cofiberofmotivicEMassemblycellularityestimatecharp} achieves the same. In either case, Corollary~\ref{cor:unstablecellularityoffibersvscofibers}(2) allows us to turn this into a weak cellular estimate for $\fib(a_n)$; taken together, we conclude that $\fib(a_n) \in O(S^{4n+2,2n+1})$.  Proposition~\ref{prop:unstableconnectivityofloopspaces} then allows us to give a cellularity estimate for the loop space $\Omega^{3,1}(-)$: writing $\Omega^{3,1}\fib(a_n) = \Omega \Omega^{2,1}\fib(a_n)$, we conclude that $\Omega^{3,1}\fib(a_n) \in O(S^{4n-1,2n})$ as required.
\end{proof}

\begin{lem}
	\label{lem:effectivecoverspluspart}
	If $X \in \Sm_k$, then $(f_N \Sigma^{\infty} X_+)_{\Q}^+ = 0$ for $N$ sufficiently large depending on $X$.  
\end{lem}

\begin{proof}
	This argument is a modification of one due to Levine \cite[Proposition 6.9(2)]{LevineConvergence}.  Since we are working with rational coefficients, and strongly dualizable objects generate $\SH(k)[\frac{1}{p}]$ (see Proposition~\ref{prop:dualizablesgenerate}), we know that $\Sigma^{\infty} X_+$ can be built out finitely many $\Sigma^{\infty}Y_+$ with $Y$ smooth and projective.  We will show that we may take $N$ to be the maximum of the dimensions of $Y$ and we reduce to demonstrating the vanishing above for $X$ smooth and projective.
	
	Let $d = \dim X$.  For any $E \in \SH(k)$ the map $f_NE \to E$ is universal among maps from spectra in $\Sigma^{2n,n}\SH(k)^{\eff}$ to $E$.  Since $\Sigma^{2n,n}\SH(k)^{\eff}$ is the localizing subcategory generated by $\Sigma^{q,q} \Sigma^{\infty} U_+$ for $q \geq n$, to show $f_NE = 0$ for $N > d$, it suffices to show that $\hom_{\SH(k)_{\Q}}(\Sigma^{*+N,N}\Sigma^{\infty}U_+,E) = 0$ for all $N > d$. 
	
	By the equivalence of $\SH(k)_{\Q}^+$ with $\Mod_{H\Q}(k)$ (see Theorem~\ref{thm:rationalizedpluspart}), in order to demonstrate $f_NE^+$ vanishes we need only establish the vanishing of
	\[
	\hom_{\Mod_{H\Q}(k)}(\mathbf{M}_{\Q}(U)[*+N](N),\mathbf{M}_{\Q}(X)).
	\]
	Since $X$ is smooth and projective of dimension $d$, Atiyah duality yields identifications of the form:    
	\[
	\begin{split}
	\hom_{\Mod_{H\Q}(k)}(\mathbf{M}_{\Q}(U)[*+N](N),\mathbf{M}_{\Q}(X)) &\weq
	\hom_{\Mod_{H\Q}(k)}(\mathbf{M}_{\Q}(U \times X),\Q(d-N)[2d-N-*]) \\
	& \weq H^{2d-N-*,d-N}(U \times X,\Q)
	\end{split}
	\] 
	For any smooth $k$-scheme, motivic cohomology vanishes in negative weight (this follows either from the identification of motivic cohomology with higher Chow groups \cite[Corollary 2]{Vmotchow}, or from the definition of motivic complexes in conjunction with the cancellation theorem \cite[Corollary 4.10]{VoeCancellation}), so the preceding motivic cohomology group vanishes when $N > d$ as required.  
\end{proof}

\begin{lem}
	\label{lem:therholocalcase}
    If $\mathscr{X} \in O(S^{2n,n})$, $n \geq 2$, then $\fib(\Lret \mathscr{X} \to \Omega^{2,1} \Lret \Sigma^{2,1}\mathscr{X}) \in O(S^{4n-1,2n})$.
\end{lem}

\begin{proof}
	Since $\mathscr{X} \in O(S^{2n,n}), n \geq 2$, it follows that $\mathscr{X}$ is $1$-connected.  In that case, Proposition~\ref{prop:unstableretloc-basics} implies that $\Lret \mathscr{X}$ is $1$-connected as well.  Moreover, Proposition~\ref{prop:unstable-rholoc} implies that being a real-\'etale sheaf is equivalent to being $\rho$-periodic.  In that case, it follows that $\Lret\Sigma^{2,1}\mathscr{X} \weq \Sigma^{2,1} \Lret \mathscr{X} \weq \Sigma \Lret\mathscr{X}$.  Likewise, $\Omega^{2,1} \Sigma \Lret \weq \Omega \Sigma \Lret\mathscr{X}$ (use again Proposition~\ref{prop:unstable-rholoc}). Moreover both spaces lie in $O(S^{q,q})$ for any $q$, being $\rho$-periodic. The result follows by appeal to Proposition~\ref{prop:refinedS1Freudenthal}.
\end{proof}

\begin{lem}
\label{lem:torsioncaseconnectivityestimate}
	 Assume $k$ is a field of finite virtual cohomological dimension and $d$ and $b$ are positive integers.  Suppose $\mathscr{X} \in O(S^{2,0})$.  For any finitely generated field extension $K/k$ of transcendence degree $\leq d$ and any integer $r$, there exists an integer $N$ (depending on $r$, $d$, $b$ and $K$, but not $\mathscr X$) such that if $\mathscr{X} \in O(S^{N,N})$ and $\mathscr X_\Q$ satisfies real étale descent, then 
	 \[
	 \bpi_{a}(\mathscr{X})_{-b}(K) \longrightarrow \bpi_{a}(\Lret \mathscr{X})_{-b}(K)
	 \]
	 is an isomorphism for $a \leq r$.
\end{lem}

\begin{proof}
We claim that there exists a function $N(v)$ such that for any field $K$ with $vcd_2(K) \le v$ and $N>N(v)$ the maps 
\begin{equation}\label{eq:ret-stable-trick-goal}
\bpi_a(\mathscr X)_{-b}(K) \longrightarrow (a_\ret \bpi_{a}(\mathscr X))_{-b}(K)
\end{equation}
are isomorphisms, and also that  $a_\ret \bpi_a(\mathscr X)$ is strictly $\A^1$-invariant.
Assuming this, from Proposition \ref{prop:unstableretloc-basics} we deduce that $\bpi_a \L_\ret \mathscr X \weq a_\ret \bpi_{a}(\mathscr X)$, and hence the first part of the claim concludes the proof.

For a nilpotent space $\mathscr Y$, put $\mathscr Y_{tors} = \fib(\mathscr Y \to \mathscr Y_\Q)$.
Recall that the inclusion of real étale sheaves into Nisnevich (or even Zariski) sheaves is exact \cite[Proposition 19.2.1]{Scheiderer}, and that the homotopy sheaves of $\mathscr X_\Q$ are real étale sheaves (Proposition \ref{prop:unstableretloc-basics}).
Combining these yields a long exact sequence of Nisnevich sheaves 
\[ 
\dots \longrightarrow a_\ret \bpi_i \mathscr X_{tors} \longrightarrow a_\ret \bpi_i \mathscr X \longrightarrow \bpi_i \mathscr X_\Q \longrightarrow \dots. 
\]
To show that the middle term is strictly $\A^1$-invariant, it suffices to establish this for the outer ones (increasing $N$ by one).
Assuming this, the sequence being exact in strictly $\A^1$-invariant sheaves, it remains so after applying $(\ph)_{-b}(K)$.
This way we reduce to proving the claim for $\mathscr X_{tors}$.

Consider the Postnikov tower 
\[ 
\tau_{\ge (q,q)} K(\bpi_i \mathscr X,i) \longrightarrow \L^{n+i+1,n} \mathscr X \longrightarrow \L^{n+i,n}\mathscr X 
\] 
of Theorem \ref{thm:refined-Postnikov}, which induces a tower of the form:
\[ 
(\tau_{\ge q,q} K(\bpi_i \mathscr X,i))_{tors} \longrightarrow (\L^{n+i+1,n} \mathscr X)_{tors} \longrightarrow (\L^{n+i,n}\mathscr X)_{tors}. 
\]
By induction, it will suffice to prove the claim for \[ (\tau_{\ge (q,q)} K(\bpi_i \mathscr X,i))_{tors} \weq (\Omega^\infty  f_q \Sigma^i H\bpi_i \mathscr X)_{tors} \weq \Omega^\infty (f_q \Sigma^i H\bpi_i\mathscr X)_{tors}. \]  We may thus assume that $\mathscr X = \Omega^\infty E$, where $E \in \SH(k)^{\eff}(N)_{\ge 0}$ is torsion.

Note that by appeal to \cite[Lemma 5.23]{BEO} there is a function $N(v)$ such that the homotopy groups of $(E/\rho)(K)$ vanish for $a < r$ and fixed $b$, for any torsion $E \in \Sigma^{N,N} \SH(k)^{\eff}_{\ge 0}$.  In particular 
\begin{equation} \label{eq:ret-stability-trick}
\bpi_{a}(\Omega^\infty E)_{-b}(K) \weq \bpi_{a}(\Omega^\infty E[\rho^{-1}])_{-b}(K), 
\end{equation}
for such $a$ and $b$.

Let $X$ be an essentially smooth, real local scheme (i.e. stalk for the real étale topology) with closed point $x$.  Since $X$ is henselian local and contains a field of characteristic zero, by rigidity \cite[Theorem 3.4.4]{EHIK} we have \[ \bpi_a(\Omega^\infty E)_{-b}(X) \weq \bpi_a(\Omega^\infty E)_{-b}(x), \] and similarly with $E[\rho^{-1}]$ in place of $E$.  However $x$ is real closed, so has $vcd_2(x) = 0$ (see Paragraph~\ref{par:realetaletopology}).  In particular by \eqref{eq:ret-stability-trick} we learn that \[ \bpi_{a}(\Omega^\infty E)_{-b}(x) \weq \bpi_{a}(\Omega^\infty E[\rho^{-1}])_{-b}(x). \]
Since this holds for all $X$ as above, we deduce that \[ a_\ret \bpi_{a}(\Omega^\infty E)_{-b} \weq a_\ret \bpi_{a}(\Omega^\infty E[\rho^{-1}])_{-b}. \]
Using that $\bpi_a(\Omega^\infty E[\rho^{-1}])_{-b}$ is a real étale sheaf we see that \[ a_\ret \bpi_{a}(\Omega^\infty E)_{-b} \weq \bpi_{a}(\Omega^\infty E[\rho^{-1}])_{-b}, \] which is strictly $\A^1$-invariant.
Combining once more with \eqref{eq:ret-stability-trick}, claim \eqref{eq:ret-stable-trick-goal} is proved.
This concludes the proof.
\end{proof}

\subsection{Complements}
The statement of Theorem~\ref{thm:freudenthal} had a restriction $m \leq 2n$ on the weak cellular class of the space under consideration.  We now remove this assumption.

\label{ss:complements}
\begin{theorem}
	\label{thm:beyondthediagonal}
	Assume $k$ is a perfect field with characteristic exponent $p$.  If $\mathscr{X} \in O(S^{m,n})$, $m-n \geq 2$, $n \ge 2$, and $\bpi_*(\mathscr{X})$ is uniquely $p$-divisible, then 
	\[
	\fib(s_{\mathscr{X}}) \in O(S^{a,2n}),
	\]  
	where $a = \min(2m-1,m+2n-1)$ (which agree when $m = 2n$).
\end{theorem}

\begin{proof}
	When $m \leq 2n$, this is precisely Theorem~\ref{thm:freudenthal}.  We thus treat the case where $m > 2n$.  We describe an inductive machine to reduce to the case $m = 2n$.  By assumption $\mathscr{X}$ is $1$-connected, and write $\mathscr{G} := \Omega \mathscr{X}$ .  In that case, recall that $\mathscr{X} \weq B\mathscr{G}$ (see Paragraph~\ref{par:kanloopgroup}), i.e., $\mathscr{X}$ can be written as a colimit over $\Delta^{\opcat}$ with terms isomorphic to finite products $\mathscr{G}$.  If $\mathscr{G} \in O(S^{m,n})$ then the same thing is true for any finite product $\mathscr{G}^{\times r}$ by appeal to Lemma~\ref{lemm:LA-products}.  
	
	Assume $\fib(s_{\mathscr{G}^{\times n}}) \in O(S^{a,b})$.  The functor $\Omega^{2,1}\Sigma^{2,1}(-)$ preserves sifted colimits by appeal to Proposition~\ref{prop:tateloopssiftedcolimits}(3).  Since sifted colimits are stable under pullbacks by Proposition~\ref{prop:realizationfibrations} we conclude that there is a fiber sequence of the form
 	\[
 		\colim_{\Delta^{\opcat}} \fib(B_{\bullet}\mathscr{G} \longrightarrow \Omega^{2,1}\Sigma^{2,1} B_{\bullet}\mathscr{G}) \longrightarrow \mathscr{X} \longrightarrow \Omega^{2,1}\Sigma^{2,1}\mathscr{X}.
 	\]
 	We claim that $\colim_{\Delta^{\opcat}} \fib(B_{\bullet}\mathscr{G} \longrightarrow \Omega^{2,1}\Sigma^{2,1} B_{\bullet}\mathscr{G}) \in O(S^{a+1,b})$.
	It is clear that the colimit lies in $O(S^{a,b})$, since each term does.
	It hence remains to prove that the colimit is $(a-b)$-connected (Corollary \ref{cor:separatingcellularity}).
	This holds since each term in the geometric realization is $(a-b)$-connective and the zeroth term is contractible.\footnote{Since geometric realizations of connected spaces commute with loops, this reduces to the claim that if $X_\bullet$ is a simplicial set with $X_0 = *$ then $\pi_0 |X_\bullet| = *$, which is clear.}
 	
 	If $m > 2n$, it follows that $\Omega^{m-2n}\mathscr{X} \in O(S^{2n,n})$.  In that case, Theorem~\ref{thm:freudenthal} implies that $\fib(s_{\Omega^{m-2n}\mathscr{X}}) \in O(S^{4n-1,2n})$.  Therefore by $m-2n$-fold application of the comparison result above, we conclude that 
 	\[
 	\fib(s_{\mathscr{X}}) \in O(S^{4n-1+(m-2n),2n}),
 	\]
 	which is exactly what we needed to show.
\end{proof}

\begin{rem}
	\label{rem:freudenthalhypotheses}
	Let us make some remarks on the hypotheses of the main theorem.  Assume $\mathscr{X} \in O(S^{m,n})$.  Without some hypothesis on $n$, Theorem~\ref{thm:freudenthal} is false.  Indeed, consider the map $S^n \to S^{n+1,1}$ for any integer $n \geq 2$.  In that case, \cite[Corollary 6.43]{MField} implies that the map $\bpi_n(S^n) \to \bpi_n(S^{n+1,1})$ is the canonical map $\Z \to \mathbf{GW}$.  In particular, this map fails to be an isomorphism.  The assumption that $n \geq 2$ arises in exactly one place: via the implicit appeal to Theorem~\ref{thm:conservativityofgmstabilization}.  In view of Remark~\ref{rem:morelimprovement}, Theorem~\ref{thm:freudenthal} in fact holds whenever $n \geq 1$, conditional on F. Morel's announced results.
\end{rem}

	
\begin{theorem}
	Let $\mathscr{X} \in O(S^{p,q})$ with $p \ge q \ge 2$, $p-q \ge 2$.  In that case, for any integer $i \geq 1$
	\[
	\fib(\mathscr{X} \to \Omega^{2i,i}\Sigma^{2i,i}\mathscr{X}) \in O(S^{a,2q}),
	\]
	where $a = \min(2p-1,p+2q-1)$, and likewise, 
	\[
	\fib(\mathscr{X} \to Q\mathscr{X}) \in O(S^{a,2q})
	\]
	as well.
\end{theorem}

\begin{proof}
	The case $i = 1$ is Theorem~\ref{thm:beyondthediagonal}, so assume $i > 0$ and consider the map $s_{\Sigma^{2i-2,i-1}\mathscr{X}}$.  Since $\mathscr{X} \in O(S^{p,q})$, it follows that $\Sigma^{2i-2,i-1}\mathscr{X} \in O(S^{p+2i-2,q+i-1})$.  In that case, Theorem~\ref{thm:beyondthediagonal} implies that $\fib(s_{\Sigma^{2i-2,i-1}\mathscr{X}}) \in O(S^{a'(i),2(q+i-1)})$ where $a'(i) = min(2(p+2i-2)-1,p+2i-2+2(q+i-1))$.  Applying $\Omega^{2i-2,i-1}$, we see that 
	\[
	\Omega^{2i-2,i-1}\fib(s_{\Sigma^{2i-2,i-1}\mathscr{X}}) \in O(S^{a(i),b})
	\]
	where $a(i) = min(2p + 2i-2-1,p+2i-2+2q + i-1)$.  The result follows by appeal to \cite[Proposition 3.1]{ABHWhitehead}.  The second statement follows from the first.
\end{proof}

\section{Applications}
\label{s:applications}
In this section, we deduce various concrete consequences of Theorem~\ref{thm:freudenthal}.   In particular, we establish Theorem~\ref{thmintro:murthy} from the introduction.

\subsection{Murthy's corank 1 splitting conjecture and complements}
The strategy of argument is laid out in \cite{AFpi3a3minus0} and it is observed in \cite{AFOBW} that a suitable version of the Freudenthal suspension theorem would suffice to establish Murthy's conjecture; see \cite[\S 4.3-4.4]{AFICM} for further details of this approach.  The discussion of \cite{AFICM} reduces Murthy's conjecture to the triviality of a certain sheaf \cite[Conjecture 4.3.2]{AFICM}; we describe in the next section how to establish this vanishing statement.  However, to establish Murthy's conjecture we actually need significantly less, as we now describe.

\begin{theorem}[Murthy's conjecture in char. 0]
	\label{thm:murthy}
	Assume $k$ is an algebraically closed field having characteristic $0$ and $X$ is a smooth affine $d+1$-fold, $d \geq 1$.  If $E$ is a rank $d$ vector bundle on $X$, then $E$ splits off a trivial rank $1$ summand if and only if $0 = c_{d}(E) \in CH^d(X)$.
\end{theorem}

\todo{Can we do this away from char $2$ and $3$?}

\begin{proof}
	By \cite[Theorem 1]{AHW}, the vector bundle $E$ corresponds to an element of $[X,BGL_d]$.  Moreover, by \cite[Theorem 2.2.4]{AHWII}, there is a fiber sequence
	\[
	S^{2d-1,d} \longrightarrow BGL_{d-1} \longrightarrow BGL_d,
	\]
	where we have used the identification $GL_d/GL_{d-1} \weq \op{Q}_{2d-1} \weq S^{2d-1,d}$.  	We proceed by analyzing the Moore--Postnikov factorization of the stabilization map $BGL_{d-1} \to BGL_d$ (see \cite[Theorem 6.1.1]{AFpi3a3minus0} for a convenient summary).  
	
	Recall \cite[Proposition 6.3.1]{AFpi3a3minus0}: since $X$ has dimension $d$, a lift of an element of $[X,BGL_d]_{\aone}$ along the map $BGL_{d-1} \to BGL_d$ exists if and only if a primary and secondary obstruction vanish.  Moreover, the primary obstruction is the twisted Euler class 
	\[
	e(E) \in H^d(X,\bpi_{d-1,0}(S^{2d-1,d})(\det E)),
	\] 
	but under the hypothesis that $k$ is algebraically closed the primary obstruction vanishes if and only if $c_{d}(E)$ vanishes.  By hypothesis, it suffices to show that the secondary obstruction vanishes, and by \cite[Corollary 6.3.3]{AFpi3a3minus0} this obstruction lies in 
	\[
	o_{2}(E) \in H^{d+1}(X,\bpi_{d,0}(S^{2d-1,d})(\det E)).
	\]    
	We will deduce this vanishing from Theorem~\ref{thm:freudenthal} via analysis of the Gersten complex of the homotopy sheaf $\bpi_{d,0}(S^{2d-1,d})$.
	
	We now recall some facts about the sheaf $\bpi_{d,0}(S^{2d-1,d})$.  First, an explicit construction via ``Suslin matrices" yields maps: 
	\[
	S^{2n-1,n} \longrightarrow \Omega^{-2n+1,-n}(\Z \times B_{\et}\op{O}); 
	\]
	these are the maps $\Psi_n$ from \cite[Definition 3.3.5]{AFKODegree} in conjunction with \cite[Theorem 2.2.2]{AFKODegree} where we write $\Z \times B_{\et}\op{O}$ for $OGr$; we make sense of the negative loopings of $\Z \times B_{\et }\op{O}$ by geometric Bott periodicity via the work of \cite{SchlichtingTripathi}.  At the level of homotopy sheaves, the induced maps
	\[
	\bpi_{d+j,j}(S^{2d-1,d}) \longrightarrow \mathbf{GW}^{d-j}_{d+1-j}
	\] 
	are epimorphisms for $j \geq d-3$ \cite[Theorem 5]{AFKODegree}, though are not known to be so for smaller $j$.  
	
	The cohomology of strictly $\aone$-invariant sheaves may be computed by appeal to Gersten resolutions.  In more detail, if $\mathbf{A}$ is a strictly $\aone$-invariant sheaf on $\mathscr{X}$ and $L$ is a line bundle on $X$, then $H^{d+1}(X,\mathbf{A}(L))$ can be computed by means of a Gersten resolution and is the cokernel of a map
	\[
	\bigoplus_{x \in X^{(d)}} \mathbf{A}_{-d}(\kappa_x;L_x) \longrightarrow \bigoplus_{x \in X^{(d+1)}} \mathbf{A}_{-d-1}(\kappa_x;L_x).
	\]
	The fact that the morphism $\bpi_{d+j,j}(S^{2d-1,d}) \to \mathbf{GW}^{d-j}_{d+1-j}$ described in the preceding paragraph is an epimorphism for $j \geq d-3$ implies that the induced map on cohomology
	\[
	H^{d+1}(X,\bpi_{d,0}(S^{2d-1,d})(\det E)) \longrightarrow H^{d+1}(X,\mathbf{GW}^d_{d+1}(\det E))
	\]
	is itself an epimorphism.  Under the assumptions on $k$, the latter group vanishes by appeal to \cite[Theorem 3.7.1]{AFpi3a3minus0}.  In other words, the image of the secondary obstruction $o_2(E)$ in the target is zero.  Set
	\[
	\mathbf{A}(d) := \ker(\bpi_{d,0}(S^{2d-1,d}) \longrightarrow \mathbf{GW}^d_{d+1}).
	\]
	Even though the morphism $\bpi_{d,0}(S^{2d-1,d}) \to \mathbf{GW}^d_{d+1}$ may fail to be surjective, if we write $\mathbf{B}(d)$ for its image (which is again strictly $\aone$-invariant), then the discussion above shows that $\mathbf{B}(d)$ and $\mathbf{GW}^d_{d+1}$ both have the same cohomology (twisted by $\det E$) in degree $d+1$.  Therefore, to establish the result it suffices to show that $H^{d+1}(X,\mathbf{A}(d)(\det E))$ vanishes.  For this vanishing, we appeal to Theorem~\ref{thm:freudenthal}.
	
	Indeed, Theorem~\ref{thm:freudenthal} tells us that the map $S^{2d-1,d} \to \Omega^{2,1}S^{2d+1,d+1}$ has fiber lying in $O(S^{4d-3,2d})$.  In particular, if $d \geq 3$ and $j \geq 0$, then there are epimorphisms
	\[
	\bpi_{d+j,j}(S^{2d-1,d}) \longrightarrow \bpi_{d+j,j}(\Omega^{2,1}S^{2d+1,d+1}).
	\]
	On the other hand, by \cite[Theorem 4.4.1]{AFpi3a3minus0} we may identify $\mathbf{A}(3)$ with the sheaf $\mathbf{F}_5$ which fits into an exact sequence of the form:
	\[
	\mathbf{I}^6 \longrightarrow \mathbf{F}_5 \longrightarrow \K^M_5/24 \longrightarrow 0.
	\]
	The contractions of these sheaves have been computed (see, e.g., \cite[Example 2.3.2]{AFpi3a3minus0}).  Indeed, we deduce that $\mathbf{A}(3)_{-4}$ is an extension of $\K^M_1/24$ by $\mathbf{I}^2$.  More generally, the discussion above implies that $\mathbf{A}(d)_{-d-1}$ is a quotient of $(\mathbf{F}_5)_{-4}$.  In particular, for $x \in X^{(d+1)}$, since $\kappa_{x}$ is algebraically closed, we see that $\mathbf{I}^2(\kappa_x) = \K^M_1/24(\kappa_x) = 0$.  The required vanishing of
	\[
	H^{d+1}(X,\mathbf{A}(d)(\det \xi)),
	\]
	then follows from the description of the Gersten resolution above and the fact that $\mathbf{A}(d)_{-d-1}(\kappa_x;(\det \xi)_x)$ is isomorphic as an abstract group to $\mathbf{A}(d)_{-d-1}(\kappa_x)$ for any $x \in X^{(d+1)}$.
\end{proof}

\begin{rem}
The vanishing of $H^{d+1}(X,\bpi_{d,0}(S^{2d-1,d})(\det E))$ established in the theorem above also implies Suslin's cancellation conjecture in corank $1$ \cite[Conjecutre 4.5.1]{AFICM} over algebraically closed fields having characteristic $0$ by appeal to \cite[Theorem 1.2]{DuEnumerating}.  However, in \cite{FaselCancellation}, J. Fasel establishes Suslin's cancellation conjecture in corank $1$ for smooth affine $d$-folds over an algebraically closed field $k$ in which $(d-1)!$ is invertible without establishing the above vanishing.  
\end{rem}

\subsection{Computations}
\label{ss:computations}
We now describe computations of unstable homotopy sheaves of motivic spheres.  Recall from the proof of Theorem~\ref{thm:murthy} that the unit map $\1 \to \mathbf{KO}$ desuspends to morphisms 
\[
S^{2n-i,n} \longrightarrow \Omega^{-2n+i,-n}(\Z \times B_{\et}\op{O}).
\]
We make sense of the deloopings of $\Z \times B_{\et}\op{O}$ by means of geometric Bott periodicity \cite[Theorems 1 and 3]{SchlichtingTripathi}; we write $B_{\et}\op{O}$ for $OGr$.  It was observed in \cite[Theorem 5]{AFKODegree} that the maps with source $S^{2n-1,n}$ are given by explicit ``Suslin matrices", and  in \cite[Theorem 4.3.6]{ADF} that those with source $S^{2n,n}$ are realized by explicit vector bundles.  We complement this here by observing that there are induced morphisms:
\[
S^{2n-i,n} \longrightarrow \tau_{\geq 2n-i,n}\Omega^{-2n+i,-n}(\Z \times B_{\et}\op{O});
\]
the spaces on the right can be thought of as unstable analogs of the spaces arising in the very effective cover of $\mathbf{KO}$.

On the other hand, consider the motivic Hopf map $\nu: S^{7,4} \to S^{4,2}$, obtained by performing the Hopf construction on the multiplication map $\op{SL}_2 \times \op{SL_2} \to \op{SL_2}$.  This map is also part of the motivic $J$-homomorphism.  The motivic Hopf map induces morphisms $\nu_d: S^{2d+2,d+2} \to S^{2d-1,d}$.  The induced map on homotopy sheaves factors through a map 
\[
(\nu_d)_*: \K^M_{d+2-j}/24 \longrightarrow \bpi_{d+j,j}^{\aone}(S^{2d-1,d});
\]
this is \cite[Theorem 3.18(2)]{AFWSuslin}, which builds on \cite[Theorem 5.2.5]{AWW}.  With those observations in mind, the following result establishes \cite[Conjecture 4.3.2]{AFICM}, which is a stronger form of \cite[Conjecture 5.2.9]{AWW}.  

\begin{theorem}
	\label{thm:pid}
	Assume $k$ is a field having characteristic $0$.  For any integer $d \geq 4$ there is a short exact sequence of the form
	\[
	0 \longrightarrow \K^M_{d+2-j}/24 \stackrel{(\nu_d)_*}{\longrightarrow} \bpi_{d+j,j}^{\aone}(S^{2d-1,d}) \longrightarrow \mathbf{GW}^{d-j}_{d+1-j},
	\]
	which is an isomorphism if $j \geq d-3$.
\end{theorem}

\begin{proof}
	Combine Theorem~\ref{thm:freudenthal} with \cite[Theorem 5.5]{RSO}.
\end{proof}

We include here a number of further unstable computations of homotopy sheaves.


\begin{theorem}
	If $k$ is a perfect field, then the following statements hold.
	\begin{enumerate}[noitemsep,topsep=1pt]	
	 \item For any $d \geq 0$ there are isomorphisms
	 \[
	 \bpi_{4+r,r}(S^{3+d,d}) \isomto \bpi_{d+r,r}(S^{2d-1,d}).
	 \]	
	 \item For any $d \geq 1$ there is a short exact sequence
	 \[
	 0 \longrightarrow \coker(\bpi_{6+r,r}(S^{4+d,d}) \longrightarrow \K^{MW}_{2d-r}/h^{\frac{(1 + (-1)^{d+1})}{2}}) \longrightarrow  \bpi_{3+r,r}(S^{2+d,d}) \longrightarrow \bpi_{4+r,r}(S^{3+d,d}) \longrightarrow 0.
	 \]
	 \item If $k$ has characteristic $0$, then for any $d \geq 2$, there is a short exact sequence of the form
	 \[
	 0 \longrightarrow \K^M_{d+2-r}/24 \longrightarrow \bpi_{4+r,r}(S^{3+d,d}) \longrightarrow \mathbf{GW}^{d-r}_{d+1-r},
	 \] 
	 where the right vertical map is surjective for $r \geq d-3$.
 	\end{enumerate}
\end{theorem}

\begin{proof}
	The first point is an immediate consequence of the simplicial EHP exact sequence \cite[Theorem 3.2.1]{AWW}.  The second point follows from the simplicial EHP sequence in conjunction with the computation of the $d_1$-differential \cite[Theorem 4.4.1]{AWW}: when $d$ is even, the differential is $0$, while if $d$ is odd, it is multiplication by $h$.  The final point follows from the first point in conjunction with Theorem~\ref{thm:pid}.
\end{proof}
	
\begin{rem}
	In a range, depending on $r$, the cokernel in Point (2) of the previous statement coincides with $\K^{MW}_{2d-r}/h^{\frac{(1 + (-1)^{d+1})}{2}}$ because the source groups are zero.  Lacking applications, we do not make this precise.
\end{rem}

\begin{footnotesize}
\bibliographystyle{alpha}
\bibliography{freudenthalref}
\end{footnotesize}
\Addresses
\end{document}